\documentclass[a4paper,11pt]{article}

\usepackage{anyfontsize}

\usepackage{booktabs}
\usepackage{colortbl}
\usepackage[pdftex,dvipsnames,table]{xcolor}
\usepackage{xfrac}
\usepackage{xargs}
\usepackage[utf8]{inputenc}
\usepackage[mathscr]{euscript}
\usepackage{xspace}
\usepackage{graphicx}
\usepackage{color}
\usepackage{enumerate}
\usepackage[english]{babel}
\usepackage{verbatim}
\usepackage{float}
\usepackage{soul}
\usepackage{rotating}
\usepackage{caption}
\usepackage[draft]{fixme}
\usepackage{longtable}
\usepackage{mathtools}
\usepackage[shortlabels]{enumitem}
\usepackage{tabularx}
\usepackage{bbm}
\usepackage[inkscapelatex=false]{svg}
\usepackage{makecell}
\usepackage{algorithm}
\usepackage{stackengine}

\usepackage{caption}
\usepackage[labelfont=sf]{subcaption}
\definecolor{lightgray}{gray}{0.9}
\floatstyle{plain}
\restylefloat{table}

\makeatletter
\newcommand{\multiline}[1]{%
  \begin{tabularx}{\dimexpr\linewidth-\ALG@thistlm}[t]{@{}X@{}}
    #1
  \end{tabularx}
}
\makeatother

\usepackage{algorithm}%
\usepackage[noend]{algpseudocode}%

\usepackage{hyperref}
\usepackage{cleveref}

\makeatletter
\newcounter{HALG@line}
\renewcommand{\theHALG@line}{\thealgorithm.\arabic{ALG@line}}
\makeatother

\hypersetup{
    colorlinks,
    linkcolor={red!50!black},
    citecolor={blue!50!black},
    urlcolor={blue!80!black}
}

\newcommand{\mytitle}[1]{#1}

\newcommand{\mc}[1]{\mathcal{#1}}

\newcommand{\R}{\mathbb{R}}
\newcommand{\Z}{\mathbb{Z}}

\newcommand{\Q}{\mathbb{Q}}

\newcommand{\mb}[1]{\mathbb{#1}}

\newcommand{\tsc}[1]{\textsc{#1}}
\newcommand{\oa}[1]{\vec{#1}}
\newcommand{\allones}{\mathbf{1}}

\makeatletter
\DeclareRobustCommand{\cev}[1]{%
  {\mathpalette\do@cev{#1}}%
}
\newcommand{\do@cev}[2]{%
  \vbox{\offinterlineskip
    \sbox\z@{$\m@th#1 x$}%
    \ialign{##\cr
      \hidewidth\reflectbox{$\m@th#1\vec{}\mkern4mu$}\hidewidth\cr
      \noalign{\kern-\ht\z@}
      $\m@th#1#2$\cr
    }%
  }%
}
\makeatother

\makeatletter
\DeclareRobustCommand{\overrarrow}[1]{%
  {\mathpalette\do@overrarrow{#1}}%
}
\newcommand{\do@overrarrow}[2]{%
  \vbox{\offinterlineskip
    \sbox\z@{$\m@th#1 x$}%
    \ialign{##\cr
      \hidewidth$\m@th#1\vec{}\mkern2mu$\hidewidth\cr
      \noalign{\kern-\ht\z@}
      $\m@th#1#2$\cr
    }%
  }%
}
\makeatother

\providecommand{\keywords}[1]{\textit{Keywords:} #1}

\DeclareMathOperator{\proj}{\rm{proj}}
\DeclareMathOperator{\conv}{\rm{conv}}

\newcommand{\T}{\mathsf{\scriptscriptstyle T}} %

\newcommand{\Xsub}{\hyperref[set:subtour]{\mc{X}_{\tsc{sub}}}}
\newcommand{\vrpsd}[1]{\hyperref[problem:vrpsd]{\tsc{vrpsd}(#1)}}
\newcommand{\Xcvrp}{\hyperref[set:cvrp]{\mc{X}_{\tsc{cvrp}}}}
\newcommand{\Fq}[1]{\hyperref[set:feasible_region]{\mc{F}(#1)}}
\newcommand{\Qc}{\hyperref[eq:formula_dror]{\mc{Q}_C}}
\newcommand{\Xh}[1]{\hyperref[set:x_h]{\mc{X}_{=}(#1)}}
\newcommand{\supsetXh}[1]{\hyperref[set:x_supseteq_h]{\mc{X}_{\supseteq}(#1)}}

\newcommand{\action}{\hyperref[definition:recourse_policy]{\mc{Y}}}
\newcommand{\policy}{\hyperref[definition:recourse_policy]{\Pi}}
\newcommand{\kbar}[1]{\hyperlink{definition:k_bar}{\bar{k}(#1)}}
\newcommand{\kscen}[1]{\hyperlink{definition:k_xi}{k_\xi(#1)}}

\newcommand{\optQ}{\hyperref[definition:scen_opt_policy]{\mc{Q}^*}}
\newcommand{\optQscen}{\hyperref[definition:scen_opt_policy]{\mc{Q}^*_\xi}}

\newcommand{\dualset}{\hyperlink{definition:set_dual_multipliers}{\mc{A}}}

\newcommand{\Whs}{\hyperref[eq:whs]{W_{HS}}}
\newcommand{\Wof}{\hyperref[eq:wof]{W_{OF}}}
\newcommand{\Wdl}{W_{DL}}

\newcommand{\D}{\hyperref[subsection:input]{D}}
\newcommand{\flow}{\hyperref[lemma:flow_formulation]{\textsc{flow}}}
\newcommand{\sri}{\hyperref[theorem:sri_formulation]{\textsc{sri}}}
\newcommand{\srihat}{\hyperref[set:sri_bar]{\widehat{\textsc{sri}}}}
\newcommand{\set}{S}
\newcommand{\phiv}{\hyperref[eq:phi]{\phi_v}}
\newcommand{\Gammaset}
{\hyperref[set:gamma_single_inequality]{\Gamma}}
\renewcommand{\P}{\hyperref[set:P]{\mc{P}}}

\newcommand{\nufunction}[1]{\hyperref[eq:nu]{\nu(#1)}}
\newcommand{\lpopt}{\hyperref[problem:lagrangian_dual]{z^*}}
\newcommand{\sigmax}[1]{\hyperref[definition:sigma_x]{\sigma_x(#1)}}
\newcommand{\sigmay}[1]{\hyperref[definition:sigma_y]{\sigma_y(#1)}}

\usepackage[margin=1in]{geometry}
\usepackage{amsmath,amsthm, amssymb, amsfonts, amsbsy}
\usepackage{thmtools} %
\usepackage{thm-restate} %
\usepackage{authblk}
\providecommand{\keywords}[1]{\textit{Keywords:} #1}
\usepackage{natbib}
\setcitestyle{authoryear,round}

\makeatletter
\newtheoremstyle{mystyle}%
{\topsep}{\topsep}
{\itshape}{}
{\bfseries}{}
{0.5em}
{\thmname{\@ifempty{#3}{#1}\@ifnotempty{#3}{#3}}}
\makeatother

\theoremstyle{mystyle}

\newcommand{\Halmos}{\qed}

\theoremstyle{plain}
\newtheorem{theorem}{Theorem}
\newtheorem{proposition}{Proposition}
\newtheorem{corollary}{Corollary}
\newtheorem{lemma}{Lemma}
\newtheorem{claim}{Claim}
\newtheorem{fact}{Fact}

\theoremstyle{definition}

\newtheorem{definition}{Definition}
\newtheorem{example}{Example}
\newtheorem{remark}{Remark}
\newtheorem{assumption}{Assumption}

\newtheoremstyle{named}{}{}{}{}{\bfseries}{.}{.5em}{#1 #3}
\theoremstyle{named}

\newenvironment{APPENDICES}{\appendix}{}

\usepackage{silence}
\begin{document}
\title{\mytitle{On vehicle routing problems with stochastic demands -- Scenario-optimal recourse policies}}

\author[1]{Matheus J. Ota\thanks{(mjota@uwaterloo.ca)}}
\author[1]{Ricardo Fukasawa\thanks{(rfukasawa@uwaterloo.ca)}}
\affil[1]{University of Waterloo, Waterloo, Ontario, Canada}

\maketitle

Two-Stage Vehicle Routing Problems with Stochastic Demands (VRPSDs) form a class of stochastic combinatorial optimization problems where routes are planned in advance, demands are revealed upon vehicle arrival, and recourse actions are triggered whenever capacity is exceeded. Following recent works, we consider VRPSDs where demands are given by an empirical probability distribution of scenarios. Existing approaches rely on integer L-shaped (ILS) cuts, whose coefficients are tailored for specific recourse policies. In contrast, we propose a framework that casts recourse policies as solutions of a higher-dimensional mixed-integer program, and we characterize its convex hull in the original lower-dimensional space via a new class of inequalities called scenario recourse inequalities (SRIs). We show that SRIs are valid for any recourse policy satisfying mild assumptions and are sufficient for formulating the VRPSD under a scenario-optimal recourse policy, where the recourse actions are chosen optimally for each scenario. Under this latter policy, we also demonstrate that SRIs dominate several known classes of ILS cuts. We conduct computational experiments on the VRPSD with scenarios under both the classical and the scenario-optimal recourse policies. By using the SRIs, our algorithm solves 329 more instances to optimality  than the previous state-of-the-art ILS algorithm.


\section{Introduction}
\label{section:intro}

\textit{Two-Stage Vehicle Routing Problems with Stochastic Demands} (VRPSDs) constitute a class of stochastic variants of the classic \textit{Capacitated Vehicle Routing Problem} (CVRP) in which routing decisions are fixed in advance, customer demands are revealed upon vehicle arrival, and capacity violations trigger additional \emph{recourse actions}. Since their introduction over 50 years ago~\citep{tillman1969multiple}, VRPSDs have attracted considerable attention, with strong interest in recent years~\citep{gendreau201650th,louveaux2018exact,hoogendoorn2023improved,florio2022recent, florio2020, ota2024hardness,parada2024disaggregated, Salavati2019175,Salavati2019,salavati2019trsc,
legault2025superadditivity, part1, hoogendoorn2025evaluation}.

A key limitation of most existing literature is its reliance on simplifying assumptions on the demand distributions; in particular, most studies assume independence~\citep{gendreau201650th, laporte2002, jabali2014, Salavati2019175,Salavati2019,salavati2019trsc,hoogendoorn2023improved, parada2024disaggregated, legault2025superadditivity}. Motivated by scenario-based approaches in the general stochastic programming literature~\citep{birge2011introduction},~\cite{part1} recently addressed this limitation by using \textit{demand scenarios}. The authors also generalized previous \emph{integer L-shaped} (ILS) formulations for VRPSDs, allowing them to derive the first branch-and-cut algorithm for the VRPSD with scenarios under the \emph{classical recourse policy}, where a vehicle traverses the route and, as soon as the capacity gets exceeded, it returns to the depot to unload. 

For the most part, the approach of~\cite{part1} is independent of the scenario assumption and does not rely on a specific choice of a recourse policy. However, they use ILS cuts to capture recourse costs, and the coefficients of these cuts depend on the chosen recourse policy. As a result, ILS cuts derived for one policy cannot be directly applied to other policies. This policy-specific dependence is also present in all other ILS-based methods, whether they consider the classical~\citep{jabali2014, parada2024disaggregated, gendreau95}, \emph{optimal}~\citep{Salavati2019175, hoogendoorn2023improved, legault2025superadditivity}, or \emph{rule-based}~\citep{salavati2019trsc} recourse policies.

In this work, we develop an approach tailored to scenario-based VRPSDs that yields cuts whose validity is not tied to specific recourse policies. To accomplish this, we propose a common framework that can accommodate all recourse policies satisfying a natural assumption. Within this framework, recourse policies are represented as feasible solutions of a higher-dimensional mixed-integer program (MIP). This naturally leads us to define \emph{scenario-optimal recourse policies} as the optimal solutions of such a MIP, which correspond to policies that optimally select the recourse actions for each scenario. 

Beyond unifying the modeling of recourse policies, this representation reveals a previously unexplored polyhedral structure, which we use to derive our main results. First, although recourse policies are defined in a higher-dimensional space, we characterize their convex hull using inequalities in the original space. This leads to the introduction of \emph{scenario recourse inequalities} (SRIs), which are valid for any recourse policy satisfying the assumptions of our framework. Second, we show that SRIs enable a formulation for the VRPSD under a scenario-optimal recourse policy that includes decision variables for both the route edges and the recourse actions. Third, we project this last formulation onto the space of ILS formulations. Using our convex hull characterization, we express recourse lower bounds for scenario-optimal recourse policies as optimal values of linear programs (LPs). This allows us to apply LP duality to show that several known classes of ILS cuts are dominated by SRIs.

In summary, our contributions are:
\begin{itemize}
    \item We propose a common framework that accommodates several recourse policies and requires only mild assumptions (see Section~\ref{section:recourse_policies} and Assumptions~\ref{assumption:large_demands},~\ref{assumption:disaggregation},~\ref{assumption:recourse_lower_bound} and~\ref{assumption:recourse_lower_bound2});
    \item We describe the convex hull of recourse policies (Theorem~\ref{theorem:convex_hull});
    \item We define scenario-optimal recourse policies as optimal solutions within this framework, and we formulate the VRPSD under a scenario-optimal recourse policy using SRIs (Definition~\ref{definition:scen_opt_policy} and
    Theorem~\ref{theorem:sri_formulation});
    \item We demonstrate that SRIs are valid for the VRPSD with scenarios under any recourse policy that satisfies the assumptions of our framework (Proposition~\ref{proposition:valid_cuts_recourse});
    \item We show that, under a scenario-optimal recourse policy, several ILS inequalities are dominated by the SRIs (Theorem~\ref{theorem:dominance_sri}, Corollary~\ref{corollary:dominance_sri} and Theorem~\ref{theorem:dominance_set_cut});
    \item We present computational experiments that validate our theoretical findings and show that SRIs can substantially enhance the performance of a state-of-the-art branch-and-cut algorithm (Section~\ref{section:computation}). In particular, by incorporating SRIs, we are able to solve 329 more instances to optimality compared to the approach in~\cite{part1}.
\end{itemize}

The rest of this paper is organized as follows. Section~\ref{section:problem} formally defines the class of problems we study and briefly reviews concepts introduced by~\cite{part1}. Section~\ref{section:recourse_policies} defines recourse policies and characterizes their convex hull. Section~\ref{section:scenario_optimal} introduces scenario-optimal recourse policies and a VRPSD formulation based on SRIs. Section~\ref{section:projection} examines the projection of the polytope defined by SRIs onto the space of ILS formulations. Section~\ref{section:computation} presents computational results, and Section~\ref{section:conclusion} concludes the paper.

\noindent
\paragraph{\mytitle{\textbf{Notation.}}}
We let~$\R_+$ and~$\R_{++}$ denote the nonnegative and positive real numbers, respectively, with analogous notation for~$\Q$ and~$\Z$. For any real number~$a$, we define~$(a)^+ \coloneqq \max\{0, a\}$. For any integer~$a$, we define~$[a] \coloneqq \{1, \ldots, a\}$ if $a > 0$, and $[a] \coloneqq \emptyset$ otherwise. The symbol~$\mb{I}(\cdot)$ represents the indicator function. For any vector~$f$, we use~$f_i$ and~$f(i)$ interchangeably. For any function (or vector)~$f$ and a subset~$S$ of its domain (or coordinates),~$f(S) \coloneqq \sum_{i \in S} f(i)$. We write~$\mathbf{1}$ and~$\mathbf{0}$ to refer to the all-ones and all-zeroes vectors, respectively.

For an undirected graph~$G$, $V(G)$ and $E(G)$ refer to the sets of vertices and edges, while for a directed graph we denote its set of arcs by~$A(G)$. We sometimes write an edge~$\{u,v\}$ (or an arc~$(u,v)$) as~$uv$. For every~$S \subseteq V(G)$,~$\delta_G(S)$ (respectively,~$E_G(S)$) is the set of edges with exactly one endpoint (respectively, two endpoints) in~$S$. In addition, for any two disjoint subsets~$S,T \subseteq V(G)$, we define~$E_G(S,T) \coloneqq \{e \in E(G) : |e \cap S| = |e \cap T| = 1\}$. When the graph is clear from context, we omit~$G$ from the subscript. If~$G$ is a directed graph, then~$\delta_G^+(S)$ (respectively,~$\delta_G^-(S)$) denotes the sets of arcs~$uv \in A(G)$ with~$u \in S$ and~$v \notin S$  (respectively,~$u \notin S$ and~$v \in S$). If~$S$ is a singleton~$\{v\}$, we may write~$S = v$. 

\section{\mytitle{Problem description and ILS formulations}}
\label{section:problem}

As mentioned earlier, we consider in this work the \emph{VRPSD with scenarios}~\citep{part1}, a class of routing problems where customer demands are uncertain and modeled via a finite set of scenarios. We adopt the \emph{two-stage} (or \emph{a priori}) paradigm~\citep{OYOLA2018193, gendreau201650th}, where the \emph{first-stage} decisions specify a feasible \emph{routing plan}, while during the \emph{second-stage}, the planned routes are traversed, and the customer demands are revealed upon vehicle arrival. Consequently, a vehicle may have insufficient capacity to serve a customer, leading to a \emph{route failure}. To handle such failures, a given \emph{recourse policy} prescribes certain \emph{recourse actions} that the vehicle should execute --- typically, these are either \emph{back-and-forth} trips between the failure location and the depot, or \emph{preventive returns}, where the vehicle returns to the depot to unload and then proceeds directly to the next customer in the route. 

In this sense, all the different recourse policies previously proposed in the literature share the key property that the vehicle never carries more load than its capacity, and they only differ in when the trips to the depot are performed to unload~\citep{dror89, Yee1980, yang2000stochastic, Salavati2019175, Salavati2019, salavati2019trsc}. The goal of the VRPSD with scenarios is to find a routing plan that minimizes the sum of the first-stage routing costs and the expected costs of executing the recourse actions.

In Section~\ref{subsection:input}, we describe the input of the VRPSD with scenarios, and in Section~\ref{subsection:vrpsd}, we cast it as an instance of the more general class of problems introduced in~\cite{part1}. This allows us to represent several variants of the VRPSD with scenarios studied in this paper under a common notation. Sections~\ref{subsection:basic_formulation} and~\ref{subsection:ils} then briefly review key concepts from the ILS-based approach of~\cite{part1}, which will serve as a baseline for our method. 

We highlight that while Section~\ref{subsection:ils} presents in detail the ILS cuts from~\cite{part1}, these cuts are not used in this paper until Section~\ref{section:projection}, where we compare certain ILS cuts with the SRIs.

\subsection{\mytitle{Input data}}
\label{subsection:input}

Let~$G = (V, E)$ be a complete undirected graph with edge weights~$c \in \Q^E_+$. The vertex set~$V$ is partitioned as~$\{0\} \dot\cup V_+$, where~$0$ represents the depot and~$V_+$ denotes the set of customers. We denote by~$D = (V, A)$ the digraph obtained from~$G$ by replacing each edge with two arcs in opposite directions. Each vehicle has a capacity of~$C \in \Q_{++}$. 

Customer demands are modeled with a random vector~$d$, whose components~$d(v)$ correspond to the random demand of customer~$v \in V_+$. The random vector~$d$ follows a probability distribution~$\mb{P}$ that is \emph{given by~$N \in \Z_{++}$ scenarios.} We refer to any~$\xi \in [N]$ as a \emph{scenario}, and we associate with this scenario a \emph{demand vector}~$d^\xi \in \Q^{V_+}_+$ and a \emph{realization probability}~$p_\xi \in (0, 1] \cap \Q_+$. These parameters satisfy~$\sum_{\xi \in [N]} p_\xi = 1$ and~$\mb{P}(d = d^\xi) = p_\xi$, for all~$\xi \in [N]$. For convenience, we use~$\bar{d}$ as an abbreviation for~$\mb{E}[d]$. An instance of the VRPSD with scenarios is given by the tuple~$\mc{I} = (G, c, C, d^1, \ldots, d^N, p_1, \ldots, p_N)$, which we assume to be fixed throughout the entire paper. 

As in~\cite{part1}, we also assume that the customer demands in each scenario are never larger than the vehicle capacity:
\begin{assumption}
    \label{assumption:large_demands}
    For every~$\xi \in [N]$ and~$v \in V_+$, we have that~$d^\xi(v) \leq C$.
\end{assumption}

\noindent
The reasoning for Assumption~\ref{assumption:large_demands} is that, whenever scenario~$\xi$ is realized, the vehicle is guaranteed to execute~$\lfloor d^\xi(v) / C \rfloor$ trips between the depot and customer~$v$ in the second stage. Therefore, if~$\lfloor d^\xi(v) / C \rfloor \geq 1$, we can add~$p_\xi (2 c_{0v}) \cdot \lfloor d^\xi(v) / C \rfloor$ to the objective function and preprocess the demand accordingly. %

\subsection{\mytitle{Problem description}}
\label{subsection:vrpsd}

A \emph{route}~$R$ is a simple undirected cycle in~$G$ that includes the depot, i.e.,~$V(R) = \{0, v_1, v_2, \ldots, v_\ell\}$ and~$E(R) = \{\{0, v_1\}, \{v_1, v_2\}, \ldots, \{v_\ell, 0\}\}$, where~$v_1, \ldots, v_\ell \in V_+$ are all distinct. The set of customers in~$R$ is denoted~$V_+(R) = \{v_1, \ldots, v_\ell\}$ and we often represent~$R$ by the tuple~$(v_1, \ldots, v_\ell)$ (implicitly assuming~$v_0 = v_{\ell+1} = 0$). A \emph{subroute} of~$R$ is any route of the form~$R' = (v_i, \ldots, v_j)$ with~$1 \leq i \leq j \leq \ell$, and we write~$R' \subseteq R$ to indicate that~$R'$ is a subroute of~$R$ (even though~$R'$ is not necessarily a subgraph of~$R$). The notation~$c(R)$ is a shorthand for~$c(E(R))$. A \emph{routing plan}~$\mc{R} = \{R_1, \ldots, R_k\}$ is a collection of routes such that~$\{V_+(R_i)\}_{i = 1}^k$ forms a partition of~$V_+$ (note that~$k$ is not necessarily fixed here). 

When discussing specific recourse policies, it is common to distinguish between the two orientations of a route. Accordingly, we associate with each route~$R = (v_1, \ldots, v_\ell)$ two \emph{directed routes}~$\oa{R}, \cev{R} \subseteq D$, which are digraphs with the same vertex set as~$R$, but with arc sets~$A(\oa{R}) = \{(0, v_1), \ldots, (v_\ell, 0)\}$ and~$A(\cev{R}) = \{(0, v_\ell), \ldots, (v_1, 0)\}$. Similarly to undirected routes, we represent directed routes using tuples, so~$\oa{R} = (v_1, \ldots, v_\ell)$ and~$\cev{R} = (v_\ell, \ldots, v_1)$.

It is well known~\citep{TothV02, DantzigFJ54} that routing plans can be represented as integer vectors inside the polytope
\begin{equation}
\tag{$\mc{X}_{\tsc{sub}}$}
\label{set:subtour}
\mc{X}_{\tsc{sub}} =
\left\{x \in [0, 2]^{E}:~
\begin{aligned}
& x(\delta(v)) = 2, & \forall v \in V_+ \\
& x(E(S)) \leq |S| - 1, & \forall \emptyset \subsetneq S \subseteq V_+
\end{aligned}
\right\}.
\end{equation}
In this sense, we sometimes refer to a vector~$\bar{x} \in \Xsub \cap \Z^E$ as a routing plan, and we use~$\mc{R}(\bar{x})$ to denote its corresponding collection of routes. As mentioned before, the routing plans~$\bar{x} \in \Xsub\cap \Z^E$ may not all contain the same number~$k$ of routes.

Most works on the VRPSD (see Table~1 of~\cite{hoogendoorn2025evaluation}) assume that feasible routing plans satisfy the classical \emph{CVRP feasibility} conditions with respect to the expected demands. That is,~$k \in \Z_{++}$ is now given as an input parameter and routing plans~$x$ belong to the set
\begin{equation}
\tag{$\mc{X}_{\tsc{cvrp}}$}
\label{set:cvrp}
\mc{X}_{\tsc{cvrp}} := \mc{X}_{\tsc{sub}} \cap 
\left\{x \in [0, 2]^{E}:~
\begin{aligned}
& x(\delta(0)) = 2 k, & \\
& x(E(S)) \leq |S| - \bar{k}(S), & \forall \emptyset \subsetneq S \subseteq V_+
\end{aligned}
\right\},
\end{equation}
where \hypertarget{definition:k_bar}{$\bar{k}(S) \coloneqq \lceil \bar{d}(S) / C \rceil$}. Thus, in contrast to~$\Xsub$, routing plans~$\bar{x} \in \Xcvrp \cap \Z^E$ use exactly~$k$ routes.

\cite{hoogendoorn2025evaluation} question the use of~$\Xcvrp$ as somewhat arbitrary, and show that, in some settings, using~$\Xcvrp$ instead of~$\Xsub$ can significantly increase the solution cost. Therefore, we consider both sets in this work, and henceforth, whenever we write~$\mc{X}$, we assume that~$\mc{X} \in \{\Xsub, \Xcvrp\}$ (if~$\mc{X} = \Xcvrp$, we also assume that~$k$ is part of the input~$\mc{I}$).

To model the expected recourse cost of a route, we let~$\mc{Q}$ be a \emph{recourse function}~\citep{ota2024hardness}, meaning that it takes the fixed input~$\mc{I}$ as a parameter and maps each route~$R$ to a nonnegative rational number, i.e.,~$\mc{Q}(R ; \mc{I}) \in \Q_+$ for every route~$R$. Since~$\mc{I}$ is fixed, we omit it from the notation. For this fixed input, we define the \emph{VRPSD with scenarios}
with respect to~$\mc{X}$ and~$\mc{Q}$ as
\begin{equation}
\tag{$\tsc{vrpsd}(\mc{X}, \mc{Q})$}
\label{problem:vrpsd}
\min\left\{ \sum_{R \in \mc{R}(x)} [c(R) + \mc{Q}(R)] : x \in \mc{X} \cap \Z^E \right\}.
\end{equation}

We remark that the problem above has precisely the form of the \emph{\mytitle{vehicle routing problems with recourse} (VRPRs)} introduced by~\cite{ota2024hardness}. Problem~\ref{problem:vrpsd} can thus be viewed as a subclass of the VRPRs where the input~$\mc{I}$ is fixed as in Section~\ref{subsection:input}, and the recourse function~$\mc{Q}$ takes as input the route~$R$ (together with the scenario demands and probabilities) and outputs the expected cost of the trips to the depot that ensures the vehicle never carries more load than its capacity. Section~\ref{section:recourse_policies} will more formally define the conditions on~$\mc{Q}$.

With this notation, we can conveniently refer to several variants of the VRPSD with scenarios. For instance, if~$\Qc$ corresponds to the \emph{classical recourse policy} (see Example~\ref{example:classic_recourse}), then the problem considered in the computational experiments of~\cite{part1} is denoted as~$\vrpsd{\Xcvrp, \Qc}$. Following the work of~\cite{hoogendoorn2025evaluation}, we may also drop the assumption of CVRP feasibility, which yields the variant~$\vrpsd{\Xsub, \Qc}$. Furthermore, Section~\ref{section:scenario_optimal} introduces a new \emph{scenario-optimal} recourse function~$\optQ$, which gives rise to problems~$\vrpsd{\Xcvrp, \optQ}$ and~$\vrpsd{\Xsub, \optQ}$.

\subsection{\mytitle{Recourse disaggregations and feasible regions}}
\label{subsection:basic_formulation}

In this subsection, we review the \emph{recourse disaggregation} model studied in~\cite{part1}, which was motivated by the computational success of different algorithms that disaggregate the recourse function~\citep{hoogendoorn2023improved, parada2024disaggregated, legault2025superadditivity}. We refer the reader to~\cite{part1} for a more detailed discussion. For our purposes, we take from that work the following definition and assumption:
\begin{definition}
    \label{definition:recourse_disaggregation}
    For each route~$R$, let~$\{\mc{Q}(R, v)\}_{v \in V_+}$ be nonnegative rational numbers such that~$\mc{Q}(R) = \sum_{v \in V_+(R)} \mc{Q}(R, v)$. We refer to the values~$\{\mc{Q}(R, v) : \text{route~$R$,~$v \in V_+$}\}$ as a \emph{disaggregation} of~$\mc{Q}$. 
\end{definition}

\begin{assumption}
    \label{assumption:disaggregation}
    We have access to a disaggregation~$\{\mc{Q}(R, v) : \text{route~$R$,~$v \in V_+$}\}$ of~$\mc{Q}$ such that, for every route~$R$,~$\mc{Q}(R, v) = 0$, for all~$v \notin V_+(R)$.
\end{assumption}

\noindent
\cite{part1} has a discussion on why Assumption~\ref{assumption:disaggregation} is both natural and theoretically justified. From now on, whenever we write~$\mc{Q}(R, v)$, we implicitly refer to the values of some disaggregation of~$\mc{Q}$.

For any recourse function~$\mc{Q}$, there is always a choice of disaggregation that satisfies Assumption~\ref{assumption:disaggregation} (e.g., pick an arbitrary customer~$v \in V_+(R)$ and assign~$\mc{Q}(R, v) = \mc{Q}(R)$). More interestingly, the following example shows that, in some cases, the disaggregation can be chosen so that each term~$\mc{Q}(R, v)$ corresponds to the expected cost of executing recourse actions \emph{at customer~$v \in V_+(R)$} while traversing route~$R$.

\begin{example}
\label{example:classic_recourse}
Intuitively, the \emph{classical recourse policy} simply determines that the vehicle executes back-and-forth trips whenever it observes a failure. Since we assume~$\mb{P}$ is given by scenarios, we compute the recourse cost of the directed route~$\oa{R} = (v_1, \ldots, v_\ell)$ under the classical recourse policy using the following formula (see~\citep{dror89}):
\begin{equation}
    \tag{$\mc{Q}_C$}
    \label{eq:formula_dror}
    \mc{Q}_C(\oa{R}) = \sum_{j \in [\ell]} 2 \, c_{0v_j} \sum_{\xi \in [N]} p_\xi \, \sum_{t = 1}^\infty \, \mb{I}\left(\sum_{i \in [j - 1]} d^\xi(v_i) \leq t C < \sum_{i \in [j]} d^\xi(v_i)\right).
\end{equation}
Without loss of generality, assume~$\mc{Q}_C(\oa{R}) \leq \mc{Q}_C(\cev{R})$ and define~$\mc{Q}_C(R) \coloneqq \mc{Q}_C(\oa{R})$. The disaggregation used in~\cite{part1} is such that, for each~$j \in [\ell]$,~$\Qc(R, v_j) = 2 \, c_{0v_j} \sum_{\xi \in [N]} p_\xi \, \sum_{t = 1}^\infty \, \mb{I}\left(\sum_{i \in [j - 1]} d^\xi(v_i) \leq t C < \sum_{i \in [j]} d^\xi(v_i)\right).$~\hfill\Halmos
\end{example}

Having established Assumption~\ref{assumption:disaggregation},~\cite{part1} associate with a recourse function~$\mc{Q}$ (and its disaggregation) the feasible region
\begin{equation}
    \tag{$\mc{F}(\mc{X}, \mc{Q})$}
    \label{set:feasible_region}
    \mc{F}(\mc{X}, \mc{Q}) \coloneqq \left\{ (x, \theta) \in  (\mc{X} \cap \Z^E) \times \R^{V_+}_+ : \theta_v \geq \sum_{R \in \mc{R}(x)} \mc{Q}(R, v),\quad \forall v \in V_+ \right\},
\end{equation}

\noindent
and show that problem~$\vrpsd{\mc{X}, \mc{Q}}$ is equivalent to~$\min \{c^\T x + \allones^\T \theta : (x, \theta) \in \Fq{\mc{X}, \mc{Q}}\}$. The authors then argue that previous approaches~\citep{gendreau95, hoogendoorn2023improved, parada2024disaggregated, legault2025superadditivity} all replace the nonlinear constraints on the~$\theta$-variables with certain (linear) ILS cuts that preserve the optimal value of the problem. We discuss some of these cuts in the next subsection.

\subsection{\mytitle{ILS cuts}}
\label{subsection:ils}

Intuitively, ILS cuts are valid inequalities for~$\Fq{\mc{X}, \mc{Q}}$ with a simple structure: given~$\mc{X}' \subseteq \mc{X} \cap \Z^E$ and a lower bound~$\mc{L} \in \Q_+$, an ILS cut is \emph{active} whenever~$x \in \mc{X}'$, in which case it enforces a lower bound of~$\mc{L}$ on the sum of certain~$\theta$-variables; otherwise, the cut is \emph{inactive} and does not impose any additional restrictions. 

Formally, an affine function~$W(x ; \mc{X}')$ is called an \emph{activation function} (with respect to~$\mc{X}' \subseteq \mc{X} \cap \Z^E$) if, for every~$x \in \mc{X} \cap \Z^E$, we have $W(x ; \mc{X}') = 1$ whenever $x \in \mc{X}'$, and $W(x ; \mc{X}') \leq 0$ otherwise. Given a subset of customers~$\set \subseteq V_+$ and a recourse function~$\mc{Q}$ that satisfies Assumption~\ref{assumption:disaggregation}, we define an ILS cut as a valid inequality for~$\Fq{\mc{X}, \mc{Q}}$ of the form~$\theta(\set) \geq \mc{L} \cdot W(x ; \mc{X}')$,
where~$\mc{L} \in \Q_+$ is called a \emph{recourse lower bound} with respect to~$\mc{X}' \subseteq \mc{X} \cap \Z^E$ and~$\set \subseteq V_+$. By the definition of~$\Fq{\mc{X}, \mc{Q}}$, it follows that~$\sum_{v \in \set} \sum_{R \in \mc{R}(\bar{x})} \mc{Q}(R, v) \geq \mc{L}$, for every~$\bar{x} \in \mc{X}'$.

Next, we describe very briefly two of the main classes of ILS cuts in~\cite{part1}.

\paragraph{\mytitle{\textbf{Partial route cuts.}}}

Partial routes are generalizations of routes that were introduced by~\cite{hjorring1999new} and are commonly used in ILS-based algorithms for VRPSDs. A \emph{partial route} is a tuple~$H = (\set_1, \ldots, \set_\ell)$ of disjoint customer subsets such that there exists no index~$i \in [\ell]$ for which both~$\set_i$ and~$\set_{i + 1}$ are not singletons (for convenience,~$\set_0 = \set_{\ell+1} = \{0\}$). We often interpret~$H$ as a subgraph of~$G$ with vertex set~$V(H) \coloneqq \cup_{i \in [\ell + 1]} S_i$ and edge set~$E(H) \coloneqq \cup_{i \in [\ell]} (E(S_i, S_{i - 1} \cup S_{i + 1}) \cup E(S_i))$. In addition, the customers in~$H$ are represented by~$V_+(H) = V(H) \cap V_+$. We say that a route~$R = (v_1, \ldots, v_h)$ \emph{adheres} to partial route~$H = (\set_1, \ldots, \set_\ell)$ if~$V_+(R) = V_+(H)$ and, for each~$i \in [\ell]$, we can label~$\set_i = \{v^i_1, \ldots, v^i_{t_i}\}$ so that~$(v_1, \ldots, v_h) = (v^1_1, \ldots, v^1_{t_1}, \ldots, v^\ell_1, \ldots, v^\ell_{t_\ell})$. 

To refer to the set of routing plans containing routes (or subroutes) that adhere to~$H$, define
\begin{align}
    & \mc{X}_{=}(H) \coloneqq \left\{ x \in \mc{X} \cap \Z^E :~\text{$\exists R 
    \in \mc{R}(x)$ s.t.~$R$ adheres to~$H$} \right\}, \tag{$\mc{X}_{=}(H)$} 
    \label{set:x_h}\\
    & \mc{X}_{\supseteq}(H) \coloneqq \left\{ x \in \mc{X} \cap \Z^E :~\text{$\exists R \in \mc{R}(x) ~\text{and}~ R' \subseteq R$ s.t.~$R'$ adheres 
    to~$H$} \right\}. \tag{$\mc{X}_{\supseteq}(H)$} \label{set:x_supseteq_h}
\end{align}

\noindent
 \cite{part1} proposed the following activation function with respect to~$\supsetXh{H}$:
\begin{equation}
    \label{eq:wof}
    \tag{$W_{OF}$}
    W_{OF}(x; \supsetXh{H}) = 1 + (x(E(H) \setminus \delta(0)) - |V_+(H)| + 1) + \sum_{i \in \{2,  \ell - 1\} \cap [\ell]}(x(E(\set_i)) - |\set_i| + 1).
\end{equation}
They also have shown that the activation function of~\cite{hoogendoorn2023improved} can be expressed as
\begin{equation}
    \label{eq:whs}
    \tag{$W_{HS}$}
    W_{HS}(x; \Xh{H}) = \Wof(x ; \supsetXh{H}) + \sum_{i \in \{1, \ell\}} (x(\delta(S_i) \cap E(H)) + 2 x(E(S_i)) - 2 |S_i|).
\end{equation}
We refer the reader to~\cite{part1} for more details.

Let~$\mc{L}_=(H)$ and~$\mc{L}_{\supseteq}(H)$ be nonnegative rational values such that
\begin{align}
    & \theta(V_+(H)) \geq \mc{L}_=(H) \cdot \Whs(x ; \Xh{H}) \quad \text{and} \label{ineq:partial_route_cut1}\\
    & \theta(V_+(H)) \geq \mc{L}_{\supseteq}(H) \cdot \Wof(x ; \supsetXh{H}) \label{ineq:partial_route_cut2}
\end{align}
are valid for~$\Fq{\mc{X}, \mc{Q}}$. In other words, inequality~\eqref{ineq:partial_route_cut1} (respectively, inequality~\eqref{ineq:partial_route_cut2}) is an ILS cut, and~$\mc{L}_=(H)$ (respectively,~$\mc{L}_\supseteq(H)$) is a recourse lower bound with respect to~$\Xh{H}$ and~$V_+(H)$ (respectively,~$\supsetXh{H}$ and~$V_+(H)$). We henceforth refer to inequalities~\eqref{ineq:partial_route_cut1} and~\eqref{ineq:partial_route_cut2} as \emph{partial route cuts}.

When~$\mc{Q} = \Qc$, the value of~$\mc{L}_=(H)$ can be computed as described in~\cite{part1}. In the case that~$H$ corresponds to a route~$R$ (i.e., every set in~$H$ is a singleton), denoted as~$H = R$, they set~$\mc{L}_=(H) = \Qc(R)$, which implies that inequalities~\eqref{ineq:partial_route_cut1} ensure that the~$\theta$-variables correctly model the recourse cost of a solution. The computation of~$\mc{L}_{\supseteq}(H)$ is addressed later in Section~\ref{section:projection}.

\paragraph{\mytitle{\textbf{Set cuts.}}}

Set cuts are ILS cuts that enforce a lower bound on the recourse cost incurred when visiting a set of customers using the minimum required number of vehicles. They were introduced by~\cite{parada2024disaggregated} for the VRPSD under the classical recourse policy with specific probability distributions, and later extended to more general recourse functions by~\cite{legault2025superadditivity} and~\cite{part1}$.$

Let~$\emptyset \subsetneq \set \subseteq V_+$ and suppose that~$k' \in \Z_{++}$ is a lower bound on the number of routes required to serve the customers in~$\set$, i.e., inequality~$x(E(\set)) \leq |\set| - k'$ is valid for~$\mc{X} \cap \Z^E$. Define the set
\begin{align}
    & \mc{X}(\set, k') \coloneqq \left\{ x \in \mc{X} \cap \Z^E : x(E(S)) = |S| - k' \right\}. \label{set:x_set} \tag{$\mc{X}(S, k')$}
\end{align}
A \emph{set cut} is an ILS cut of the form
\begin{equation}
    \label{ineq:set_cut1}
    \theta(\set) \geq \mc{L}(\set) \cdot \Wdl(x ; \mc{X}(\set, k')),
\end{equation}
where~$\Wdl(x ; \mc{X}(\set, k')) \coloneqq 1 + (x(E(\set)) - |\set| + k')$ is the activation function employed by the DL-shaped method of~\cite{parada2024disaggregated}, and~$\mc{L}(S) \in \Q_+$ is a recourse lower bound with respect to~$\mc{X}(\set, k')$ and~$S$. We explain the specific recourse lower bound used by~\cite{part1} in Section~\ref{section:projection}.

\section{\mytitle{Recourse policies}}
\label{section:recourse_policies}

As mentioned in Sections~\ref{section:intro} and~\ref{section:problem}, there are multiple different recourse policies (and corresponding recourse functions) that have been studied in the literature. A key difficulty in placing these policies within a common framework is that they are typically defined through sets of rules used to compute the recourse, as illustrated in Example~\ref{example:classic_recourse}. In this section, we take a different view and cast \emph{recourse policies} and \emph{recourse actions} as feasible solutions of a MIP based on a network-flow formulation.

Section~\ref{subsection:recourse_policies_definition} provides the formal definition of recourse policies under this framework. Section~\ref{subsection:convex_hull} then leverages polyhedra and network-flow theory to characterize the convex hull of the set of recourse policies. While these structural results do not seem directly connected with the solution of problem~$\vrpsd{\mc{X}, \mc{Q}}$, subsequent sections will make use of them for that purpose. In particular, Section~\ref{section:scenario_optimal} shows that, by adding an objective function to the MIP formulation studied here, we are able to derive valid inequalities for problem~$\vrpsd{\mc{X}, \mc{Q}}$. %

\subsection{\mytitle{Definition via network flows}}
\label{subsection:recourse_policies_definition}

Henceforth, we reserve the notation~$y$ for a vector in~$\R^{[N] \times V_+}$ with entries~$y^\xi_v$, for each scenario~$\xi \in [N]$ and customer~$v \in V_+$. We often represent~$y$ with the tuple~$(y^1, \ldots, y^N)$, where, for each scenario~$\xi \in [N]$,~$y^\xi \in \R^{V_+}$ is the restriction of~$y$ to the entries~$\{y^\xi_v\}_{v \in V_+}$. 

Recall that~$\D = (V, A)$ is the complete digraph obtained from~$G$. Based on the intuitive description of the VRPSD given in the beginning of Section~\ref{section:problem} (where a vehicle never carries more load than its capacity), we formalize the notions of \emph{recourse actions} and \emph{recourse policies} as follows (see Figure~\ref{figure:recourse_action}).

\begin{definition}
\label{definition:recourse_policy}
Fix a directed route~$\oa{R} = (v_1, \ldots, v_\ell)$ and let~$\xi \in [N]$ (recall that~$v_0 = v_{\ell + 1} = 0$). A \emph{recourse action for~$\oa{R}$ and~$\xi$} is a vector~$y^\xi \in \Z^{V_+}_+$ such that there exist~$f^\xi \in \R^{A}_+$ and~$g^\xi \in \R^{V_+}_+$ satisfying
\begin{subequations}
\label{formulation:flow_feasibility}
\begin{align} 
& f^\xi_{(v_{i - 1}, v_i)} + d^\xi(v_i) = f^\xi_{(v_i, v_{i + 1})} + g^\xi_{v_i}, && \forall i \in [\ell], \label{flow:conservation} \\
& f^\xi_{(v_{i - 1}, v_i)} \leq C, && \forall i \in [\ell + 1], \label{flow:capacity} \\
& g^\xi_{v_i} \leq C \cdot y^\xi_{v_i}, && \forall i \in [\ell]. \label{flow:recourse}
\end{align}
\end{subequations}
The set of all recourse actions for~$\oa{R}$ and~$\xi$ is denoted~$\mc{Y}^\xi(\oa{R})$. A \emph{recourse policy for~$\oa{R}$} is a vector~$y = (y^1, \ldots, y^N) \in \Pi(\oa{R})$, where~$\Pi(\oa{R}) \coloneqq \mc{Y}^1(\oa{R}) \times \ldots \times \mc{Y}^N(\oa{R})$ is the set of all recourse policies for~$\oa{R}$.
\end{definition}

\noindent
Note that, although the constraints in Formulation~\eqref{formulation:flow_feasibility} only involve the vertices in~$V_+(\oa{R})$, the set~$\action^\xi(\oa{R})$ is a subset of~$\R^{V_+}_+$, so the entries of a recourse action~$y^\xi\in \action^\xi(\oa{R})$ that correspond to vertices not in~$\oa{R}$ may take any arbitrary value in~$\Z_+$. We adopted this definition to guarantee that~$y^\xi, f^\xi$ and~$g^\xi$ all lie in a common space across all directed routes, which will be useful later in Section~\ref{section:scenario_optimal}.

In a sense, Definition~\ref{definition:recourse_policy} is natural: for each scenario~$\xi \in [N]$, the flow variables~$f^\xi$ determine the load of the vehicle along the route, the recourse action~$y^\xi$ specifies at which customers and how many times we unload at the depot, and the flow variables~$g^\xi$ determine the corresponding amount that we unload. The next example shows that previously proposed recourse policies are indeed covered by Definition~\ref{definition:recourse_policy}.

\begin{example}
\label{example:recourse_action}
Fix a directed route~$\oa{R} = (v_1, \ldots, v_\ell)$. Suppose that a given recourse policy specifies, for each customer~$v_i$, a threshold value~$\tau_i \in [0, C] \cap \Q_+$ indicating that the vehicle should unload at the depot whenever the accumulated load upon arrival at~$v_i$ (including the demand of~$v_i$) exceeds~$\tau_i$. As observed in~\cite{Salavati2019}, several recourse policies fall within this setup: the classical recourse policy~\citep{gendreau95, laporte2002, jabali2014, christiansen2007, GAUVIN2014141} sets all the threshold values to~$C$; preventive recourse policies~\citep{Yee1980, hoogendoorn2023improved, yang2000stochastic, Salavati2019175, florio2020} sets the threshold values according to a dynamic-programming algorithm; and rule-based recourse policies assign thresholds according to other route measures such as volume, risk and distances (see~\cite{salavati2019trsc, Salavati2019} for more details).

Let~$d^\xi$ be any realization of the random vector~$d$ and let~$\bar{y}^\xi \in \Z^{V_+}_+$ be such that each entry~$\bar{y}^\xi_{v_i}$, with~$i \in [\ell]$, counts how many times the given recourse policy executes a recourse action at customer~$v_i$ (either through back-and-forth or preventive returns) under this realization. One can check that~$\bar{y}^\xi$ is feasible for Formulation~\eqref{formulation:flow_feasibility}. In fact,~$\bar{y}^\xi$ is feasible for a more restricted variant of Formulation~\eqref{formulation:flow_feasibility} where we add the constraints~$f^\xi_{(v_{i - 1}, v_i)} + d^\xi(v_i) \leq \tau_i + (C + d^\xi(v_i) - \tau_i) \, y^\xi_{v_i}$,
for all~$i \in [\ell]$. \hfill\Halmos 
\end{example}

Before continuing, we need to address a technicality concerning the directions of the routes. Lemma~\ref{lemma:symmetric_flow} shows that, whenever we speak of recourse policies (in the sense of Definition~\ref{definition:recourse_policy}), we may ignore route directions. In other words, for every scenario~$\xi \in [N]$ and route~$R$, we may write~$\policy(R) \coloneqq \Pi(\oa{R}) = \Pi(\cev{R})$ and~$\action^\xi(R) \coloneqq \mc{Y}^\xi(\oa{R}) = \mc{Y}^\xi(\cev{R})$. The proof is given in Appendix~\ref{appendix:proof_symmetric_flow} and was inspired by the work of~\cite{hernandez2003one}.

\begin{restatable}{lemma}{lemmasymmetricflow}
\label{lemma:symmetric_flow}
Let~$R$ be a route and~$\xi \in [N]$. For any vector~$y^\xi \in \Z^{V_+}$, we have that~$y^\xi \in \mc{Y}^\xi(\oa{R})$ if and only if there exist~$f^\xi \in \R^{A}_+$ and~$g^\xi \in \R^{V_+}_+$ such that
\begin{subequations}
\label{formulation:flow2_feasibility}
\begin{align} 
& f^\xi_{(v_{i - 1}, v_i)} + f^\xi_{(v_{i + 1}, v_i)} + d^\xi(v_i) = f^\xi_{(v_i, v_{i + 1})} + f^\xi_{(v_i, v_{i - 1})} + g^\xi_{v_i}, && \forall i \in [\ell], \label{flow2:conservation} \\
& f^\xi_{(v_{i - 1}, v_i)} \leq \frac{C}{2}, && \forall i \in [\ell + 1], \label{flow2:capacity1} \\
& f^\xi_{(v_i, v_{i - 1})} \leq \frac{C}{2}, && \forall i \in [\ell + 1], \label{flow2:capacity2} \\
& g^\xi_{v_i} \leq C \cdot y^\xi_{v_i}, && \forall i \in [\ell]. \label{flow2:recourse}
\end{align}
\end{subequations}
In particular, this implies that~$\mc{Y}^\xi(\oa{R}) = \mc{Y}^\xi(\cev{R})$.
\end{restatable}

\begin{figure}[htb]
    \subcaptionsetup[figure]{position=bottom,font=footnotesize}
    \centering
    \subfloat[Recourse action~$y^\xi = (0, 0, 1, 0)^\T$ corresponding to the classical recourse policy.]{\includegraphics[width=0.35\textwidth]{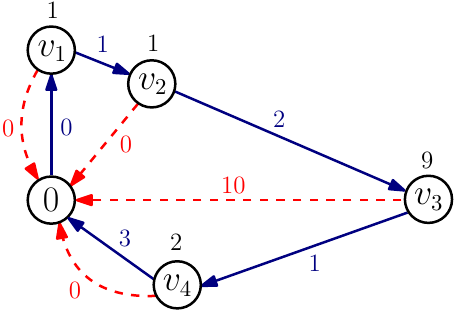} \label{figure:recourse_action_a}}
    \hspace{3cm}
    \subfloat[Recourse action~$y^\xi = (0, 1, 0, 1)^\T$.]{\includegraphics[width=0.35\textwidth]{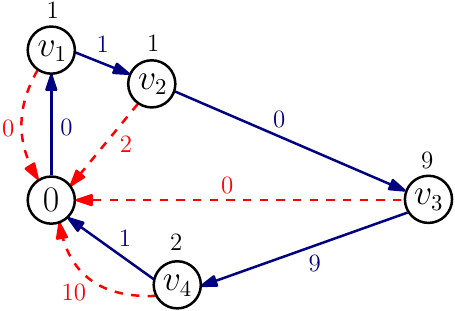}\label{figure:recourse_action_b}}
    \\[0.5cm]
    \caption{Recourse actions for an instance with 4 customers and~$C = 10$. We show recourse actions~$y^\xi = (y^\xi_{v_1}, y^\xi_{v_2}, y^\xi_{v_3}, y^\xi_{v_4})^\T$ for~$\oa{R} = (v_1, v_2, v_3, v_4)$ and a given scenario~$\xi \in [N]$. The black numbers (on top of the vertices) refer to the demands in scenario~$\xi$, the blue numbers (near the solid arcs) indicate the values of~$f^\xi$ and the red numbers (near the dashed arcs) indicate the values of~$g^\xi$.}
    \label{figure:recourse_action}
\end{figure}

\subsection{\mytitle{The convex hull of recourse policies}}
\label{subsection:convex_hull}

The following extension of Hoffman's circulation theorem~\citep{hoffman1958some} (which is a classical application of the max-flow/min-cut theorem, see Section 3.3 of~\cite{cook2011combinatorial}) will be key for our discussion.

\begin{theorem}[Theorem 3.18 of~\cite{cook2011combinatorial}] 
    Let~$D' = (V', A')$ be a directed graph. Let~$h \in \R^{A'}_+$ be a vector of arc capacities and let~$d' \in \R^{V'}$ be a vector of node demands such that~$d'(V') = 0$. There exists a vector~$f \in \R^{A'}_+$ satisfying
    \begin{subequations}
    \label{formulation:hoffman}
    \begin{align*} 
    & f(\delta^-_{D'}(v)) - f(\delta^+_{D'}(v)) = d'_v, & \forall v \in V', \\
    & 0 \leq f_a \leq h_a, & \forall a \in A',
    \end{align*}
    \end{subequations}
    if and only if~$h(\delta^-_{D'}(\set)) \geq d'(\set)$, for all~$\set \subseteq V'$ (or equivalently,~$h(\delta^+_{D'}(\set)) \geq d'(V' \setminus \set)$, for all~$\set \subseteq V'$).
    \label{theorem:hoffman}
\end{theorem}

Let~$R = (v_1, \ldots, v_\ell)$ be a route. From now on, we abbreviate~$d^\xi(V_+(R))$ as~$d^\xi(R)$ and~$y^\xi(V_+(R))$ as~$y^\xi(R)$. Moreover, for every scenario~$\xi \in [N]$ and set~$\set \subseteq V_+$, we define~\hypertarget{definition:k_xi}{$k_\xi(\set) \coloneqq \lceil d^\xi(\set) / C \rceil$}, and we use~$\kscen{R}$ as a shorthand for~$\kscen{V_+(R)}$. In what follows, we recall that a subroute~$R'$ of~$R$, denoted~$R' \subseteq R$, is a route of the form~$R' = (v_i, \ldots, v_j)$, for some~$1 \leq i \leq j \leq \ell$.

The next lemma applies Hoffman's circulation theorem to project out the flow variables in Definition~\ref{definition:recourse_policy}. Note that this result does not fully describe the convex hull of recourse policies, as Definition~\ref{definition:recourse_policy} also requires~$y$ to be an integer vector.

\begin{lemma}
    \label{lemma:projection_recourse_action}
    Let~$R$ be a route,~$\xi \in [N]$ and~$\bar{y}^\xi \in \R_+^{V_+}$. There exist~$(f^\xi, g^\xi)$ feasible for Formulation~\eqref{formulation:flow_feasibility} with respect to~$\bar{y}^\xi$ if and only if~$\bar{y}^\xi(R') \geq (d^\xi(R') / C) - 1$, for every~$R' \subseteq R$.
\end{lemma}
\begin{proof}
    Let~$R = (v_1, \ldots, v_\ell)$ and let~$D'$ be a digraph with vertex set~$V(R) = \{0, v_1, \ldots, v_\ell\}$ and arc set given by the union of the sets~$A_1 = A(R)$ and~$A_2 = \{(v_i, 0)\}_{i \in [\ell]}$. Set~$d'_0 = d^\xi(R)$ and~$d'_{v_i} = - d^\xi(v_i)$, for each~$i \in [\ell]$. Additionally, set arc arc capacities as~$h_a = C$ if~$a \in A_1$, and~$h_a = C \, \bar{y}^\xi_{v_i}$ if~$a = (v_i, 0) \in A_2$. 
    
    Theorem~\ref{theorem:hoffman} implies that Formulation~\eqref{formulation:flow_feasibility} is feasible if and only if~$h(\delta^+_{D'}(\set)) \geq d'(V(R) \setminus \set)$, for every~$\set \subseteq V_+(R)$ (since~$d'_v$ may be positive only for~$v = 0$). Fix such a set~$\set \subseteq V_+(R)$ and decompose it as the disjoint union~$\set = \dot\cup \{V_+(R_q)\}_{q \in [t]}$, where~$R_1, \ldots, R_t$ are maximal disjoint subroutes of~$R$. That is, for each~$R_q = (v_i, \ldots, v_j)$, we have that~$v_{j + 1} \notin S$, so~$(v_j, v_{j + 1}) \in \delta^+_{D'}(V_+(R_q))$ (note that if~$j = \ell$, then~$v_{j + 1} = 0$). 
    
    By the choice of the subroutes, it follows that~$h(\delta^+_{D'}(\set)) = C \sum_{i = 1}^p (\bar{y}^\xi(R_i) + 1)$. In addition,~$d'(V(R) \setminus \set) = d^\xi(\set) = \sum_{i = 1}^p d^\xi(R_i)$, so we conclude that~$h(\delta^+_{D'}(\set)) \geq d'(V(R) \setminus \set)$ is equivalent to~$C \sum_{i = 1}^p (\bar{y}^\xi(R_i) + 1) \geq \sum_{i = 1}^p d^\xi(R_i)$, proving that the conditions in the lemma are sufficient for the existence of the desired~$(f^\xi, g^\xi)$. Necessity follows from the case where~$S$ is such that~$t = 1$.
\end{proof}

Lemma~\ref{lemma:projection_recourse_action} allows us to derive a \emph{perfect formulation} (see Chapter~4 of~\cite{ipbook}) for the set of recourse policies. The key observation that enables this result is that the matrix associated with the constraints in Lemma~\ref{lemma:projection_recourse_action} has the \emph{consecutive ones} property, meaning that it is \emph{totally unimodular} (TU).

\begin{theorem}
\label{theorem:convex_hull}
For every route~$R$,
\begin{equation*}
\conv(\policy(R)) =
\left\{y \in \R^{[N] \times V_+}_+:~y^\xi(R') \geq \kscen{R'} - 1, \quad \forall R' \subseteq R,~\xi \in [N]
\right\}.
\end{equation*}
\end{theorem}
\begin{proof}
Fix a route~$R$ and let~$\tilde{\Pi}$ be the set in the RHS of the statement. We first show that~$\policy(R) = \tilde{\Pi} \cap \Z^{[N] \times V_+}$. If~$\bar{y} \in \tilde{\Pi} \cap \Z^{[N] \times V_+}$, then it follows from Lemma~\ref{lemma:projection_recourse_action} that~$\bar{y} \in \policy(R)$. Conversely, if~$y' \in \policy(R)$, then~$y'$ satisfies the inequalities in Lemma~\ref{lemma:projection_recourse_action}, and as~$y'$ is integer, we can round the RHS of these inequalities to learn that~$y' \in \tilde{\Pi} \cap \Z^{V_+}$. 

Let~$M$ be the matrix given by the inequalities defining~$\tilde{\Pi}$. Note that we may rearrange the columns of~$M$ so that its rows have the consecutive ones property, so we know that~$M$ is TU (see Corollary~2.10 of~\cite{NemhauserW88}). Hence, Theorem~4.4 of~\cite{ipbook} implies that~$\tilde{\Pi}$ is integral, and by the definition of an integral convex set,~$\tilde{\Pi} = \conv(\tilde{\Pi} \cap \Z^{[N] \times V_+}) = \conv(\policy(R))$.
\end{proof}

The same argument can be easily extended to the multi-route setting and to the case where upper bounds are imposed on the recourse action vectors. We leave the proof to Appendix~\ref{appendix:proof_corollary_convex_hull}.
\begin{restatable}{corollary}{corollaryconvexhull}
    \label{corollary:convex_hull2}
    Let~$\mc{R}$ be a collection of disjoint routes and let~$b \in \Z^{V_+}_+$. Then
    \begin{equation*}
    \conv\left( \bigcap_{R \in \mc{R}} (\policy(R) \cap [\mathbf{0}, b]^N) \right) =
    \left\{y \in [\mathbf{0}, b]^N:~y^\xi(R') \geq \kscen{R'} - 1, \quad \forall R \in \mc{R},~R' \subseteq R,~\xi \in [N]
    \right\}.
    \end{equation*}
\end{restatable}

\section{\mytitle{Scenario-optimal recourse policies and scenario-recourse inequalities}}
\label{section:scenario_optimal}

Recall from Sections~\ref{subsection:basic_formulation} and~\ref{subsection:ils} that~\cite{part1} show that existing approaches for the VRPSD replace the nonlinear constraints~$\theta_v \geq \sum_{R \in \mc{R}(x)} \mc{Q}(R, v)$ in~$\Fq{\mc{X}, \mc{Q}}$ with families of valid linear inequalities known as \emph{ILS cuts}. Importantly, these cuts rely on \emph{recourse lower bounds} derived under very specific forms of the recourse function, which makes it difficult to extend their validity to other variants of the problem.

Given that we model recourse policies as solutions to the network-flow Formulation~\eqref{formulation:flow_feasibility}, we now assign weights to the variables~$y^\xi_v$ in order to obtain lower bounds on the terms~$\mc{Q}(R, v)$. This allows us to derive stronger non-ILS cuts that remain valid for~$\Fq{\mc{X}, \mc{Q}}$ across different recourse functions~$\mc{Q}$ associated with VRPSD recourse policies. In addition, we show that, for a \emph{scenario-optimal} recourse policy~$\optQ$, these cuts are sufficient to model problem~$\vrpsd{\mc{X}, \optQ}$ (without the need for ILS cuts).

\subsection{\mytitle{Assigning weights to recourse policies}}
\label{subsection:weights}

To link a generic recourse function~$\mc{Q}$ with a recourse policy, we adopt the following assumption.
\begin{assumption}
\label{assumption:recourse_lower_bound}
We have a disaggregation of~$\mc{Q}$ and vectors~$w \in \Q^{V_+}_+$ and~$b 
\in \Z^{V_+}_+$ such that, for every route~$R$, there exists~$\bar{y}_R \in \policy(R) \cap [\mathbf{0}, b]^N$ satisfying~$\mc{Q}(R, v) \geq \sum_{\xi \in [N]} p_\xi w_v (\bar{y}_R)^\xi_v$, for all~$v \in V_+(R)$.
\end{assumption}

\noindent
For example, in the case of the classical recourse policy --- which was considered in several works~\citep{gendreau95, laporte2002, jabali2014, hoogendoorn2023improved, parada2024disaggregated} and extended to scenarios by~\cite{part1} --- we can set~$w_v = 2c_{0v}$. In fact, as we argue next (recall also Example~\ref{example:recourse_action}), Assumption~\ref{assumption:recourse_lower_bound} captures a structural property shared by several recourse functions used in VRPSDs. 

The strength of the inequalities derived in the next sections depends critically on the lower bounding vector~$w \in \R^{V_+}_+$. For example, one may always set~$w = \mathbf{0}$ and then Assumption~\ref{assumption:recourse_lower_bound} is trivially satisfied, but this would lead to weak inequalities. Intuitively, one should try to set~$w_v$ as a ``tight'' lower bound on the cost of executing recourse actions at customer~$v \in V_+(R)$. For example, for recourse policies that allow preventive return trips between the customers and the depot~\citep{Yee1980, Salavati2019175, salavati2019trsc, laporte2002, hoogendoorn2023improved, jabali2014, GAUVIN2014141, florio2020, christiansen2007}, we may use~$w_v = (\min_{u \in V_+ \setminus \{v\}} \{2 \, c_{0v}, c_{0u} + c_{0v} - c_{uv}\})^+$, for all~$v \in V_+$ (see Remark~\ref{remark:extension_assumption}).

Similarly, one should seek for the lowest possible upper bound vector~$b$ satisfying Assumption~\ref{assumption:recourse_lower_bound}. In fact, for any recourse policy (in the sense of Definition~\ref{definition:recourse_policy}), we may always assume that~$b \leq 2 \cdot \allones$. To see this, fix a directed route~$\oa{R} = (v_1, \ldots, v_\ell)$ and let~$(\bar{f}, \bar{g}, \bar{y})$ be feasible for~\eqref{formulation:flow_feasibility}. By Assumption~\ref{assumption:large_demands}, for any~$i \in [\ell]$ and~$\xi \in [N]$, we know that~$\bar{f}^\xi_{(v_{i - 1}, v_i)} + d^\xi(v_i) \leq 2 C$, which implies that at most two unload trips are needed at~$v_i$, so we may assume that~$\bar{y}^\xi_{v_i} \leq 2$. Actually, since the classical recourse policy never fails at the same customer more than once, we may set~$b = \allones$ whenever~$\mc{Q} = \Qc$.

Having said this, we now fix~$\mc{Q}$ as a recourse function that satisfies Assumption~\ref{assumption:recourse_lower_bound} with parameters~$w$ and~$b$. In view of the definition of recourse actions/policies (Definition~\ref{definition:recourse_policy}), we introduce a \emph{scenario-optimal recourse policy} that, given a route, optimally selects the recourse actions for each scenario.
\begin{definition}
\label{definition:scen_opt_policy}
Let~$w \in \Q^{V_+}_+$ and let~$b \in \Z^{V_+}_+$ be such that, for every route~$R$, the set~$\policy(R) \cap [\mathbf{0}, b]^N$ is nonempty. The \emph{scenario-optimal recourse cost} for a route~$R$ is defined as
\begin{equation}
    \tag{$\mc{Q}^*$}
    \mc{Q}^*(R) \coloneqq \min \left\{ \sum_{\xi \in [N]} \sum_{v \in V_+} p_\xi w_v y^\xi_v : ~ y \in \policy(R) \cap [\mathbf{0}, b]^N \right\}.
\end{equation}
A \emph{scenario-optimal recourse policy} is a vector~$y^* \in \policy(R)$ that attains this minimum.
\end{definition}

Since we fixed~$w$ and~$b$, we omit these parameters from the notation for~$\optQ$, and we sometimes refer to~$\optQ$ as the \emph{scenario-optimal recourse function}. It will also be useful to refer to the scenario-optimal recourse cost of a route~$R$ in a specific scenario. Thus, for each~$\xi \in [N]$, we define~$\mc{Q}^*_\xi(R) \coloneqq \min \left\{ \sum_{v \in V_+} w_v y^\xi_v : y^\xi \in \action^\xi(R) \cap [\mathbf{0}, b] \right\}$, meaning that~$\optQ(R) = \sum_{\xi \in [N]} p_\xi \, \optQscen(R)$.

In the remainder of this section, we leverage the polyhedral results from Section~\ref{section:recourse_policies} to derive a formulation for~$\vrpsd{\mc{X}, \optQ}$ that does not rely on flow variables. This formulation captures the scenario-optimal recourse function through a new class of non-ILS cuts, which we call \emph{scenario-recourse inequalities (SRIs)}. These inequalities not only model~$\optQ$ (Theorem~\ref{theorem:sri_formulation}), but they are also valid for any recourse function satisfying Assumption~\ref{assumption:recourse_lower_bound} (Proposition~\ref{proposition:valid_cuts_recourse}).

\begin{remark}
\label{remark:extension_assumption}
It was recently observed by~\cite{legault2025superadditivity} that, in some cases, the lower bound~$w_v = (\min_{u \in V_+ \setminus \{v\}} \{2 \, c_{0v}, c_{0u} + c_{0v} - c_{uv}\})^+$ can be quite weak. For example, if there exists~$u \in V_+$ such that the depot is located close to the midpoint between~$u$ and~$v$, then~$c_{0u} + c_{0v} - c_{uv}$ is close to zero. 

Let~$\oa{R}$ be a directed route and let~$v$ and~$u$ be distinct customers such that~$c_{0u} + c_{0v} - c_{uv}$ is close to zero. If arc~$(v, u)$ belongs to~$\oa{R}$, then the recourse cost at~$v$ is indeed expected to be small, since, after visiting~$v$, the vehicle can first stop at the depot to unload and then proceed to~$u$, with little additional cost for this detour. However, if~$(v, u) \notin A(\oa{R})$, then the recourse cost at~$v$ might be much larger than~$c_{0u} + c_{0v} - c_{uv}$. To address this issue, we propose the following weakening of Assumption~\ref{assumption:recourse_lower_bound}:
\begin{assumption}
\label{assumption:recourse_lower_bound2}
We have a disaggregation of~$\mc{Q}$ and vectors~$w \in \Q^{A}_+$ and~$b 
\in \Z^{V_+}_+$ such that, for every route~$R = (v_1, \ldots, v_\ell)$, there exists~$\bar{y}_R \in \policy(R) \cap [\mathbf{0}, b]^N$ satisfying
\begin{align*}
    & \mc{Q}(R, v_i) \geq \min\{w_{(v_i, u)} : u \in \{0, v_{i - 1}, v_{i + 1}\} \} \cdot \sum_{\xi \in [N]} p_\xi (\bar{y}_R)^\xi_{v_i}, & \forall i \in [\ell].
\end{align*}
\end{assumption}
Under Assumption~\ref{assumption:recourse_lower_bound2}, for every~$v \in V_+$, we may set~$w_{(v, 0)} = 2 c_{0v}$, and for every distinct~$v, u \in V_+$, we may set~$w_{(v, u)} = c_{0u} + c_{0v} - c_{uv}$. In Appendix~\ref{appendix:extension}, we extend part of the forthcoming results to the case where~$\mc{Q}$ satisfies Assumption~\ref{assumption:recourse_lower_bound2}.~\hfill\Halmos
\end{remark}

\subsection{\mytitle{Formulations under a scenario-optimal recourse policy}}
\label{subsection:scenario_optimal_formulation}

We start by extending the definition of recourse policies from routes to solutions (i.e., routing plans). Thus, for each~$\bar{x} \in \mc{X} \cap \Z^E$, define~$\policy(\bar{x}) \coloneqq \bigcap_{R \in \mc{R}(\bar{x})} \policy(R)$. We naturally map each recourse policy~$\bar{y} \in \policy(\bar{x})$ to its corresponding collection of vectors~$\{\bar{y}_R\}_{R \in \mc{R}(\bar{x})}$ according to the customers in each route:

\begin{fact}
    \label{fact:policy_split}
    Let~$\mc{R}$ be a collection of disjoint routes.
    \begin{itemize}
        \item If~$\bar{y} \in \cap_{R \in \mc{R}} \policy(R)$, for each~$R \in \mc{R}$, define~$\bar{y}_R \in \Z^{[N] \times V_+}_+$ so that~$(\bar{y}_R)^\xi_v = \mb{I}(v \in V_+(R)) \cdot \bar{y}^\xi_v$~for all~$v \in V_+$ and~$\xi \in [N]$. Then, for every~$R \in \mc{R}$,~$\bar{y}_R \in \policy(R)$.
        \item Conversely, if~$\{\bar{y}_R\}_{R \in \mc{R}}$ is such that each~$\bar{y}_R$ belongs to~$\policy(R)$, define~$\bar{y} \in \Z^{[N] \times V_+}$ so that~$\bar{y}^\xi_v = \sum_{R \in \mc{R}} \mb{I}(v \in V_+(R)) \cdot (\bar{y}_R)^\xi_v$~for every~$v \in V_+$ and~$\xi \in [N]$. Then~$\bar{y} \in \cap_{R \in \mc{R}} \policy(R)$.
    \end{itemize}
\end{fact}
\begin{proof} In both cases, for every~$R \in \mc{R}$ and~$\xi \in [N]$, we have~$\bar{y}^{\xi}_v=(\bar{y}_R)^\xi_v$, for all~$v \in V_+(R)$. By Definition~\ref{definition:recourse_policy}, in both cases,~$\bar{y}\in \policy(R)\iff \bar{y}_R\in\policy(R)$, for every~$R \in \mc{R}$.
\end{proof}

Combining Definition~\ref{definition:scen_opt_policy} and Fact~\ref{fact:policy_split} yields the following nonlinear formulation for~$\vrpsd{\mc{X}, \optQ}$:
\begin{subequations}
\label{formulation:scen_opt_basic}
\begin{align} 
\min ~~& c^\T x + \sum_{v \in V_+} \theta_v, & \nonumber \\
\text{s.t.~~} & \theta_v \geq \sum_{\xi \in [N]} p_\xi w_v y^\xi_v, & \forall v \in V_+, \label{scen_opt_basic:theta} \\
& x \in \mc{X} \cap \Z^E, & \label{scen_opt_basic:x} \\
& y \in \policy(x) \cap [\mathbf{0}, b]^N. &  \label{scen_opt_basic:policy}
\end{align}
\end{subequations}

\noindent
It follows from Assumption~\ref{assumption:recourse_lower_bound} that the projection of the feasible region of Formulation~\eqref{formulation:scen_opt_basic} onto the~$(x, \theta)$-space yields a relaxation of~$\Fq{\mc{X}, \mc{Q}}$. 
\begin{proposition}
    \label{proposition:valid_cuts_recourse}
    For any recourse function~$\mc{Q}$ satisfying Assumption~\ref{assumption:recourse_lower_bound},
    $$\Fq{\mc{X}, \mc{Q}} \subseteq \{(x, \theta) : \exists y \in \R^{[N] \times V_+} ~\text{such that}~(x, \theta, y)~\text{is feasible for Formulation~\eqref{formulation:scen_opt_basic}}\}.$$
\end{proposition}
\begin{proof}
    Fix a recourse function~$\mc{Q}$ as in the statement and let~$(\bar{x}, \bar{\theta}) \in \Fq{\mc{X}, \mc{Q}}$. We show that there exists~$\bar{y} \in \policy(\bar{x}) \cap [\mathbf{0}, b]^N$ such that~$(\bar{\theta}, \bar{y})$ satisfies~\eqref{scen_opt_basic:theta}. 
    For each route~$R \in \mc{R}(\bar{x})$, choose~$\bar{y}_R \in \policy(R) \cap [\mathbf{0}, b]^N$ according to Assumption~\ref{assumption:recourse_lower_bound}. 
    Combining these vectors using Fact~\ref{fact:policy_split} we obtain the desired~$\bar{y} \in \policy(\bar{x})$. Indeed, since~$\{\bar{y}_R\}_{R \in \mc{R}(\bar{x})} \subseteq [\mathbf{0}, b]^N$, we have~$\bar{y}\in [\mathbf{0}, b]^N$.
    Moreover, for every~$v \in V_+$, let~$R \in \mc{R}(\bar{x})$ be such that~$v \in V_+(R)$. Then
    \begin{equation*}
        \bar{\theta}_v \geq \sum_{R' \in \mc{R}(\bar{x})} \mc{Q}(R', v) \geq \mc{Q}(R, v)  \overset{\text{Assumption~\ref{assumption:recourse_lower_bound}}}{\geq} \sum_{\xi \in [N]} p_\xi w_v (\bar{y}_R)^\xi_v \overset{\text{Fact~\ref{fact:policy_split}}}{=} \sum_{\xi \in [N]} p_\xi w_v \bar{y}^\xi_v.
    \end{equation*}
\end{proof}

Proposition~\ref{proposition:valid_cuts_recourse} motivates our study of valid inequalities for Formulation~\eqref{formulation:scen_opt_basic} (recall that in Section~\ref{subsection:weights} we fixed~$\mc{Q}$ to be a recourse function that satisfies Assumption~\ref{assumption:recourse_lower_bound}). As a first step in this direction, we build on Lemma~\ref{lemma:symmetric_flow} to reformulate the family of sets~$\{\policy(x)\}_{x \in \mc{X} \cap \Z^E}$ as a single MIP, parameterized by a vector~$\bar{x} \in \mc{X} \cap \Z^E$. In this result, we link the~$x$ variables, defined on the undirected graph~$G = (V, E)$, with the flow variables in Definition~\ref{definition:recourse_policy} (see also Lemma~\ref{lemma:symmetric_flow}), defined on the associated directed graph~$\D = (V, A)$. We defer the proof to Appendix~\ref{appendix:proof_flow_formulation}. 
\begin{restatable}{lemma}{lemmaflowformulation}
    \label{lemma:flow_formulation}
    Fix~$b \in \Z^{V_+}_+$. For every~$\bar{x} \in \mc{X}$, define
    \begin{equation*}
    \textsc{flow}(\bar{x}) \coloneqq
    \left\{ y \in [\mathbf{0}, b]^N :~
    \begin{aligned}
    & f^\xi(\delta^-_{\D}(v)) + d^\xi(v) = f^\xi(\delta^+_{\D}(v)) + g^\xi_v, && \forall v \in V_+, \xi \in [N] \\
    & f^\xi_{(u, v)} \leq \frac{C}{2} \bar{x}_{\{u, v\}}, && \forall uv \in A, \xi \in [N] \\
    & g^\xi_v \leq C \cdot y^\xi_{v}, && \forall v \in V_+, \xi \in [N] \\
    & f \in \R_+^{[N] \times A},~~g \in \R_+^{[N] \times V_+}
    \end{aligned}
    \right\}.
    \end{equation*}
    Then~$\flow(\bar{x}) \cap \Z^{[N] \times V_+} = \policy(\bar{x}) \cap [\mathbf{0}, b]^N$, for every~$\bar{x} \in \mc{X} \cap \Z^E$.
\end{restatable}

Lemma~\ref{lemma:flow_formulation} yields a MIP model for Formulation~\eqref{formulation:scen_opt_basic}, since we can replace constraint~\eqref{scen_opt_basic:policy} with the constraints and variables in the (parametric) MIP formulation for~$\flow(\bar{x}) \cap \Z^{[N] \times V_+}$. However, this approach may be inefficient, as it leads to a formulation for~$\vrpsd{\mc{X}, \optQ}$ that has integrality on the~$y$-variables and uses~$N \cdot (|A| + |V_+|)$ additional flow variables. Similarly to Lemma~\ref{lemma:projection_recourse_action}, we address this issue by applying Hoffman's circulation theorem (Theorem~\ref{theorem:hoffman}) to project out the flow variables. We leave the proof to Appendix~\ref{appendix:proof_benders}.

\begin{restatable}{proposition}{propositionbenders}
    \label{proposition:benders}
    Fix~$b \in \Z^{V_+}_+$.
    For every~$\bar{x} \in \mc{X}$,
    \begin{equation*}
        \flow(\bar{x}) = \left\{ y \in [\mathbf{0}, b]^N : y^\xi(\set) \geq \frac{d^\xi(\set)}{C} + \bar{x}(E(\set)) - |\set|, \quad \forall \emptyset \subsetneq \set \subseteq V_+,~\xi \in [N] \right\}.
    \end{equation*}
\end{restatable}

Now recall from Section~\ref{subsection:convex_hull} that~$\kscen{\set} = \lceil d^\xi(S) / C \rceil$. Since Formulation~\eqref{formulation:scen_opt_basic} enforces integrality on the~$(x, y)$-variables, we may round the RHS of the inequalities in Proposition~\ref{proposition:benders} to obtain
\begin{equation}
    \label{ineq:sri}
    y^\xi(\set) \geq \kscen{\set} + x(E(\set)) - |\set|,
\end{equation}
which we call a \emph{scenario-recourse inequality (SRI)}. (We invite the reader to compare the SRIs with the inequalities in Theorem~\ref{theorem:convex_hull}.) The following example illustrates that SRIs may indeed yield a stronger relaxation than~$\flow(\bar{x})$.

\begin{example}
\label{example:sri}
Consider the setting in Figure~\ref{figure:recourse_action_b}. Suppose that~$\bar{x} \in \mc{X} \cap \Z^E$ is such that~$\mc{R}(\bar{x})$ only contains the route~$R = (v_1, v_2, v_3, v_4)$ and let~$(\bar{y}, \bar{f}, \bar{g})$ be a particular feasible solution for the formulation of~$\flow(\bar{x})$ in Lemma~\ref{lemma:flow_formulation}. We illustrate these vectors in Figure~\ref{figure:example_sri} for a fixed scenario~$\xi \in [N]$. Note that, in this example,~$\sum_{v \in V_+} w_v \bar{y}^\xi_v = 2.4$, which is less than the scenario-optimal recourse cost~$\optQscen(R) = 4$.

\begin{figure}[htbp]
  \centering
  \includegraphics[width=0.5\textwidth]{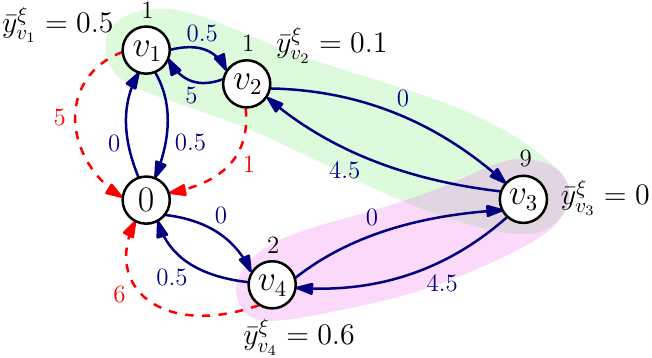}
  \caption{Graphical representation of~$(\bar{x}, \bar{y}^\xi, \bar{f}^\xi, \bar{g}^\xi)$. Here we have~$C = 10$,~$w_{v_3} = 6$ and~$w_{v_1} = w_{v_2} = w_{v_4} = 2$. The black numbers next to each vertex indicate the scenario demand vector~$d^\xi$. The blue and solid arcs represent the vector~$\bar{f}^\xi$, while the red and dashed arcs represent the vector~$\bar{g}^\xi$.}
  \label{figure:example_sri}
\end{figure}

Let~$\set_1 = \{v_1, v_2, v_3\}$ and~$\set_2 = \{v_3, v_4\}$ (green and violet regions in Figure~\ref{figure:example_sri}, respectively). Observe that~$\bar{y}^\xi$ violates the SRIs for~$\set_1$ and~$\set_2$ (with respect to scenario~$\xi$), since, for all~$i \in [2]$,~$$0.6 = \bar{y}^\xi(\set_i) < \lceil d^\xi(\set_i) / C \rceil + \bar{x}(E(\set_i)) - |\set_i| = 1.$$

Interestingly, in this case, the SRIs are sufficient to correctly capture the scenario-optimal recourse cost~$\optQscen(R)$. To see this, consider the vector~$(y')^\xi \in \action^\xi(R)$ defined as~$(y')^\xi_{v_2} = (y')^\xi_{v_4} = 1$ and~$(y')^\xi_{v_1} = (y')^\xi_{v_3} = 0$, and notice that~$\optQscen(R) = \sum_{v \in V_+} w_v (y')^\xi_v = 4$. Summing the SRIs for~$S_1$ and~$S_2$ (and fixing~$x = \bar{x}$) yields~$y^\xi_{v_1} + y^\xi_{v_2} + 2 y^\xi_{v_3} + y^\xi_{v_4} \geq 2$, and multiplying this inequality by 2 we learn that~$$\sum_{v \in V_+} w_v y^\xi_v \geq 2 y^\xi_{v_1} + 2 y^\xi_{v_2} + 4 y^\xi_{v_3} + 2 y^\xi_{v_4} \geq 4.$$ Therefore,~$(\bar{x}, (y')^\xi)$ minimizes~$\sum_{v \in V_+} w_v y^\xi_v$ subject to the constraints that~$y^\xi \in [\mathbf{0}, b]$ and~$(\bar{x}, y^\xi)$ satisfies all the SRIs of the form~$y^\xi(\set) \geq \kscen{\set} + \bar{x}(E(\set)) - |\set|$, for every~$\emptyset \subsetneq \set \subseteq V_+$.\hfill\Halmos
\end{example}

It turns out that a consequence of our polyhedral study in Section~\ref{subsection:convex_hull} is that the situation illustrated in Example~\ref{example:sri} holds more generally, and constraint~\eqref{scen_opt_basic:policy} can be replaced by the SRIs (and~$y \in [\mathbf{0}, b]^N$).

\begin{restatable}{theorem}{theoremsri}
\label{theorem:sri_formulation}
Fix~$b \in \Z^{V_+}_+$. For every~$\bar{x} \in \mc{X}$, define
\begin{equation*}
    \textsc{sri}(\bar{x}) \coloneqq \left\{ y \in [\mathbf{0}, b]^N : y^\xi(\set) \geq \kscen{\set} + \bar{x}(E(\set)) - |\set|, \quad \forall \emptyset \subsetneq \set \subseteq V_+, \xi \in [N] \right\}.
\end{equation*}
Then~$\sri(\bar{x}) = \conv(\policy(\bar{x}) \cap [\mathbf{0}, b]^N)$, for every~$\bar{x} \in \mc{X} \cap \Z^E$.
\end{restatable}
\begin{proof}
Fix~$\bar{x} \in \mc{X} \cap \Z^E$. Since~$\conv(\cap_{R \in \mc{R}(\bar{x})} (\policy(R) \cap [\mathbf{0}, b]^N)) = \conv(\policy(\bar{x}) \cap [\mathbf{0}, b]^N)$, it follows from Corollary~\ref{corollary:convex_hull2} that
\begin{equation}
    \label{eq:sri_proof_convex_hull}
    \conv(\policy(\bar{x}) \cap [\mathbf{0}, b]^N) = \left\{y \in [\mathbf{0}, b]^N:~y^\xi(R') \geq \kscen{R'} - 1, \quad \forall R \in \mc{R}(\bar{x}), R' \subseteq R, \xi \in [N]
\right\}.
\end{equation}

Let~$\bar{y} \in \policy(\bar{x}) \cap [\mathbf{0}, b]^N$. By Lemma~\ref{lemma:flow_formulation},~$\policy(\bar{x}) \cap [\mathbf{0}, b]^N = \flow(\bar{x}) \cap \Z^{[N] \times V_+}$. Since~$\bar{y}$ is integer, it follows from Proposition~\ref{proposition:benders} that~$\bar{y} \in \sri(\bar{x}) $. Convexity of~$\sri(\bar{x})$ then implies that~$\conv(\policy(\bar{x}) \cap [\mathbf{0}, b]^N) \subseteq \sri(\bar{x})$. Conversely, the inequalities describing~$\conv(\policy(\bar{x}) \cap [\mathbf{0}, b]^N)$ in~\eqref{eq:sri_proof_convex_hull} are also included in the formulation defining~$\sri(\bar{x})$ (i.e., they are SRIs with~$x$ fixed to~$\bar{x}$), so~$\sri(\bar{x}) \subseteq \conv(\policy(\bar{x}) \cap [\mathbf{0}, b]^N)$.
\end{proof}

Since optimizing over~$\policy(\bar{x}) \cap [\mathbf{0}, b]^N$ is equivalent to optimizing over its convex hull, we reformulate Formulation~\eqref{formulation:scen_opt_basic} as
\begin{subequations}
\label{formulation:sri}
\begin{align} 
\min ~~& c^\T x + \sum_{v \in V_+} \theta_v, & \nonumber \\
\text{s.t.~~} & \theta_v \geq \sum_{\xi \in [N]} p_\xi w_v y^\xi_v, && \forall v \in V_+, \label{sri_formulation:theta} \\
& y^\xi(\set) \geq \kscen{\set} + x(E(\set)) - |\set|, && \forall \emptyset \subsetneq \set \subseteq V_+, \xi \in [N], \label{sri_formulation:sri} \\
& (x, y) \in (\mc{X} \cap \Z^E) \times [\mathbf{0}, b]^N. &  \label{sri_formulation:domain}
\end{align}
\end{subequations}

Due to the exponential number of SRIs, we describe in Appendix~\ref{appendix:separation} the heuristics that we use to iteratively separate these inequalities as cutting planes. In these heuristics, for each candidate set~$S \subseteq V_+$, we ``aggregate'' the SRIs across the scenarios using the following simple fact.
\begin{fact}
\label{fact:aggregated_sri}
Let~$(\bar{x}, \bar{y}) \in \mc{X} \times [\mathbf{0}, b]^N$,~$\xi' \in [N]$ and~$\set \subseteq V_+$ be such that~$\bar{y}^{\xi'}(\set) < \kscen{\set} + \bar{x}(E(\set)) - |\set|$. Define~$\Xi \coloneqq \{\xi \in [N] : \bar{y}^{\xi}(\set) < k_{\xi}(\set) + \bar{x}(E(\set)) - |\set|\}$, then~$\sum_{\xi \in \Xi} \bar{y}^\xi(\set) < \sum_{\xi \in \Xi} (\kscen{\set} + \bar{x}(E(\set)) - |\set|)$.~\hfill\Halmos
\end{fact}

\noindent
Accordingly, we call inequality~$\sum_{\xi \in \Xi} y^\xi(\set) \geq \sum_{\xi \in \Xi} (\kscen{\set} + x(E(\set)) - |\set|)$ an \emph{aggregated SRI} with respect to~ $\Xi \subseteq [N]$ and~$\set \subseteq V_+$. Preliminary experiments indicate that separating aggregated SRIs, rather than individual SRIs, significantly improves the overall efficiency of the algorithm.

\subsection{\mytitle{LP formulation for the scenario-optimal recourse function}}
\label{subsection:lp_scen_opt}

Besides Theorem~\ref{theorem:sri_formulation}, another consequence of Theorem~\ref{theorem:convex_hull} is that the scenario-optimal recourse cost of a route can be formulated as an LP (note that directly applying Definition~\ref{definition:recourse_policy} to the optimization problem in Definition~\ref{definition:scen_opt_policy} yields a MIP). Recall that~$\optQ(R) = \sum_{\xi \in [N]} p_\xi \, \optQscen(R)$. For a route~$R$ and scenario~$\xi \in [N]$,
\begin{subequations}
\label{formulation:recourse_scen_opt}
\begin{align} 
\optQscen(R) = \min ~~& \sum_{v \in V_+} w_v y^\xi_v, & \nonumber \\
\text{s.t.~~} & y^\xi(R') \geq \kscen{R'} - 1, & \forall R' \subseteq R, \label{ineq:form:recourse_scen_opt}\\
& y^\xi \in [\mathbf{0}, b].
\end{align}
\end{subequations}

Observe that Formulation~\eqref{formulation:recourse_scen_opt} with respect to~$R' \subseteq R$ (i.e., including only constraints~\eqref{ineq:form:recourse_scen_opt} for~$R'' \subseteq R'$) yields a relaxation of Formulation~\eqref{formulation:recourse_scen_opt} with respect to~$R$, which leads to the following property of~$\optQ$.

\begin{corollary}
\label{corollary:monotonicity}
Let~$\optQ$ be a recourse function as in Definition~\ref{definition:scen_opt_policy} with parameters~$w \in \Q^{V_+}_+$ and~$b \in \Z^{V_+}_+$. Choose a disaggregation of~$\optQ$ satisfying Assumption~\ref{assumption:disaggregation} and such that for every route~$R$, there exists~$y^* \in \policy(R)$ with~$\optQ(R, v) = \sum_{\xi \in [N]} p_\xi w_v (y^*)^\xi_v$, for every~$v \in V_+$. Then, for every route~$R$ and subroute~$R' \subseteq R$,~$\sum_{v \in V_+(R')} \optQ(R', v) \leq \sum_{v \in V_+(R')} \optQ(R, v)$.
\end{corollary}
\begin{proof}
Let~$y^* \in \policy(R)$ be a scenario-optimal recourse policy such that, for all~$v \in V_+$,~$\optQ(R, v) = \sum_{\xi \in [N]} p_\xi w_v (y^*)^\xi_v$. For each scenario~$\xi \in [N]$,~$(y^*)^\xi$ is feasible for Formulation~\eqref{formulation:recourse_scen_opt} with respect to~$R'$, which implies that~$\optQscen(R') \leq \sum_{v \in V_+(R')} p_\xi w_v (y^*)^\xi_v$. Summing this inequality over all scenarios~$\xi \in [N]$ yields the desired result.
\end{proof}

Corollary~\ref{corollary:monotonicity} implies that scenario-optimal recourse policies satisfy the \emph{monotonicity/(restricted) superadditivity} properties discussed previously in the context of the \emph{DL-shaped method} \citep{parada2024disaggregated, legault2025superadditivity, part1}. Therefore, rather than using SRIs, one could also solve~$\vrpsd{\mc{X}, \optQ}$ using the \emph{path} and \emph{set} cuts of the DL-shaped method (path cuts are a special case of the partial route cuts~\eqref{ineq:partial_route_cut2} from Section~\ref{subsection:ils}). However, as we demonstrate next, we can instead leverage LP duality and Formulation~\eqref{formulation:recourse_scen_opt} to prove that these ILS cuts are already dominated by the SRIs.

\section{\mytitle{Projecting out the scenario variables and a branch-and-cut algorithm}}
\label{section:projection}

In this section, we study the projection of Formulation~\eqref{formulation:sri} onto the~$(x, \theta)$-space of~$\Fq{\mc{X}, \mc{Q}}$ from both theoretical and computational perspectives. In view of the results in the previous section, it is natural to ask whether it is preferable to formulate~$\vrpsd{\mc{X}, \optQ}$ using ILS cuts, SRIs, or a combination of both. 

From a theoretical standpoint, we show in Section~\ref{subsection:projected_sri} that the projection of Formulation~\eqref{formulation:sri} onto the~$(x, \theta)$-variables can be described by a family of inequalities which we call \emph{projected SRIs}.%
\footnote{In a sense, projected SRIs are similar to the \emph{aggregated cuts} used in single-cut implementations of the Benders' decomposition method~\citep{rahmaniani2017benders}. We use the term ``projected SRI'' because we already adopted the term ``aggregated SRIs'' for the inequalities in Fact~\ref{fact:aggregated_sri}.}
Section~\ref{subsection:comparing_cuts} shows that, under certain assumptions on the \emph{recourse lower bounds} (recall Section~\ref{subsection:ils}), projected SRIs dominate the ILS cuts from~\cite{part1} with respect to the recourse function~$\optQ$.

From a computational perspective, we note that the proofs in Section~\ref{subsection:comparing_cuts} are constructive, and they provide algorithms to generate the dominating projected SRIs. Moreover, in preliminary experiments, we observed that the LP relaxation of Formulation~\eqref{formulation:sri} often provides strong dual bounds, even when the SRIs are separated heuristically (as in Appendix~\ref{appendix:separation}). However, the large number of~$y$-variables can make solving the formulation to integrality computationally expensive. Motivated by this observation, Section~\ref{subsection:lp_bound} shows that solving the LP relaxation of Formulation~\eqref{formulation:sri} also provides a single projected SRI that recovers the LP bound of Formulation~\eqref{formulation:sri} in the~$(x, \theta)$-space. The results of Sections~\ref{subsection:comparing_cuts} and~\ref{subsection:lp_bound} are then combined in Section~\ref{subsection:sri_algorithm} to develop a branch-and-cut algorithm for solving~$\vrpsd{\mc{X}, \optQ}$ using projected SRIs. Finally, Section~\ref{subsection:extension} discusses how we extend this algorithm to different recourse policies.

In the remainder, we fix the parameters~$w \in \Q^{V_+}_+$ and~$b \in \Z^{V_+}_+$ of the scenario-optimal recourse function~$\optQ$ (Definition~\ref{definition:scen_opt_policy}). We also fix a disaggregation of~$\optQ$ as specified in Corollary~\ref{corollary:monotonicity}.

\subsection{\mytitle{Projected SRIs}}
\label{subsection:projected_sri}

In this section, we use the following standard notation for referring to projections. For any set~$\mc{H} \subseteq \R^{n_1 + n_2}$ containing vectors of the form~$(h_1, h_2)$ with~$h_1 \in \R^{n_1}$ and~$h_2 \in \R^{n_2}$, we define~$\proj_{h_1}(\mc{H}) \coloneqq \{h_1 \in \R^{n_1} : \text{~$\exists h_2 \in \R^{n_2}$ such that~}(h_1, h_2) \in \mc{H}\}$. Now, for every~$\bar{x} \in \mc{X}$, we augment the set~$\sri(\bar{x})$ as follows:
{\small
\begin{equation}
\tag{$\widehat{\tsc{sri}}(\bar{x})$}
\label{set:sri_bar}
\widehat{\tsc{sri}}(\bar{x}) \coloneqq 
\left\{(\theta,  y) \in \R^{V_+}_+ \times [\mathbf{0}, b]^N :
\begin{aligned}
    & y^\xi(\set) \geq \kscen{\set} + \bar{x}(E(\set)) - |\set|, & \forall \emptyset \subsetneq \set \subseteq V_+,~\xi \in [N] \\
    & \theta_v \geq \sum_{\xi \in [N]} p_\xi w_v y^\xi_v, & \forall v \in V_+
\end{aligned}
\right\}.
\end{equation}}

\noindent
Therefore, the projection of the feasible region of the LP relaxation of Formulation~\eqref{formulation:sri} is given by
\begin{equation}
\tag{$\mc{P}$}
\label{set:P}
\mc{P} \coloneqq 
    \left\{ (x, \theta) : x \in \mc{X},~\theta \in \proj_{\theta}(\srihat(x)) \right\}.
\end{equation}

To better understand the strength of the LP relaxation of Formulation~\eqref{formulation:sri}, we study the polyhedron~$\P$, which in turn leads us to analyze the set~$\proj_{\theta}(\srihat(x))$. Therefore, one of the goals of this subsection is to explicitly write~$\proj_{\theta}(\srihat(x))$ as the intersection of half-spaces defined by inequalities on the~$\theta$-variables.

To do this, we first consider in Lemma~\ref{lemma:projection_single_inequality} the simpler case of projecting a single inequality defined on the~$y$-variables. This will not only be useful for characterizing~$\proj_{\theta}(\srihat(x))$, but it will also be key in Section~\ref{subsection:comparing_cuts} to construct projected SRIs that dominate certain ILS cuts. The proof is left to Appendix~\ref{appendix:proof_projection_single_inequality}.

\begin{restatable}{lemma}{lemmaprojectionsingleinequality}
\label{lemma:projection_single_inequality}
Let~$a^1, \ldots, a^N$ be vectors in~$\R^{V_+}$ and let~$h \in \R$. Consider the set
\begin{equation*}
\mc{H} = 
\left\{(\theta,  y) \in \R^{V_+}_+ \times \R^{[N] \times V_+}_+ :~
    \sum_{\xi \in [N]} (a^\xi)^\T y^\xi \geq h,
    ~\theta_v \geq \sum_{\xi \in [N]} p_\xi w_v y^\xi_v, \quad \forall v \in V_+
\right\}.
\end{equation*}
The following holds:
\begin{itemize}
    \item if there exists~$v \in V_+$ and~$\xi \in [N]$ such that~$w_v = 0$ and~$a^\xi_v > 0$, then~$\proj_{\theta}(\mc{H}) = \R^{V_+}_+$;
    \item otherwise,~$\proj_{\theta}(\mc{H}) = \left\{\theta \in \R^{V_+}_+ : \sum_{v \in V_+ : w_v > 0} \left(\max_{\xi \in [N]} \left\{ \frac{a^\xi_v}{p_\xi w_v} \right\} \right)^+ \theta_v \geq h \right\}$.
\end{itemize}
\end{restatable}

Next, we apply Lemma~\ref{lemma:projection_single_inequality} to the case where the inequality~$\sum_{\xi \in [N]} (a^\xi)^\T y^\xi \geq h$ is given by a conic combination of SRIs and upper bound constraints on the~$y$-variables. Thus, we reserve the symbol~$\alpha$ for a vector of multipliers associated with the SRIs, with components~$\alpha^\xi_\set \geq 0$, for all~$\set \subseteq V_+$ and~$\xi \in [N]$. Similarly,~$\beta$ is reserved for a vector of multipliers associated with the upper bound constraints~$y^\xi_v \leq b_v$, with components~$\beta^\xi_v \leq 0$, for all~$v \in V_+$ and~$\xi \in [N]$. 

Given~$\alpha \geq 0$ and~$\beta \leq 0$, consider
\begin{align}
    & \sum_{\xi \in [N]} \sum_{\set \subseteq V_+} \alpha^\xi_\set y^\xi(\set) \geq \sum_{\xi \in [N]} \sum_{\set \subseteq V_+} \alpha^\xi_\set \, (\kscen{\set} + x(E(\set)) - |\set|) \quad \text{and} \label{ineq:conic_comb_sri} \\ 
    & \sum_{\xi \in [N]} \sum_{v \in V_+} \beta^\xi_v \, y^\xi_v \geq \sum_{\xi \in [N]} \sum_{v \in V_+} \beta^\xi_v \, b_v. \label{ineq:conic_comb_upper_bound}
\end{align}
Summing~\eqref{ineq:conic_comb_sri} with~\eqref{ineq:conic_comb_upper_bound} and rearranging the terms gives
\begin{equation}
    \label{ineq:conic_comb}
    \sum_{\xi \in [N]} \sum_{v \in V_+} \left(\beta^\xi_v + \sum_{\set \subseteq V_+ : v \in \set} \alpha^\xi_\set \right) y^\xi_v \geq \sum_{\xi \in [N]} \sum_{\set \subseteq V_+} \alpha^\xi_\set x(E(\set)) + \nufunction{\alpha, \beta}, 
\end{equation}
where
\begin{equation}
    \tag{$\nu$}
    \label{eq:nu}
    \nu(\alpha, \beta) \coloneqq  \sum_{\xi \in [N]} \sum_{\set \subseteq V_+} \alpha^\xi_\set ( \kscen{\set} - |\set| ) + \sum_{\xi \in [N]} \sum_{v \in V_+} \beta^\xi_v b_v.
\end{equation}

For each~$\bar{x} \in \mc{X}$, define
\begin{equation}
    \tag{$\Gamma$}
    \label{set:gamma_single_inequality}
    \Gamma(\bar{x}, \alpha, \beta) \coloneqq \left\{(\theta,  y) \in \R^{V_+}_+ \times [\mathbf{0}, b]^N : ~\text{$(\bar{x}, y)$ satisfies~\eqref{ineq:conic_comb}},~~\theta_v \geq \sum_{\xi \in [N]} p_\xi w_v y^\xi_v, \quad \forall v \in V_+\right\}.
\end{equation}

\noindent 
By Lemma~\ref{lemma:projection_single_inequality}, if~$\left(\beta^\xi_v + \sum_{\set \subseteq V_+ : v \in \set} \alpha^\xi_\set \right) \leq 0$, for all~$\xi \in [N]$ and~$v \in V_+$ with~$w_v = 0$, we have that~$\proj_{\theta}(\Gammaset(\bar{x}, \alpha, \beta))$ contains precisely all~$\theta \in \R^{V_+}_+$ such that
\begin{equation}
    \label{ineq:projected_sri}
    \sum_{v \in V_+ : w_v > 0} \phiv(\alpha, \beta) \, \theta_v \geq \sum_{\xi \in [N]} \sum_{S \subseteq V_+} \alpha^\xi_S \, \bar{x}(E(S)) + \nufunction{\alpha, \beta},
\end{equation}
where, for each~$v \in V_+$ with~$w_v > 0$,
\begin{equation}
\tag{$\phi_v$}
\label{eq:phi}
\phi_v(\alpha, \beta) \coloneqq \left( \max_{\xi \in [N]}\left\{ \frac{\beta^\xi_v + \sum_{\set \subseteq V_+ : v \in \set} \alpha^\xi_\set}{p_\xi w_v} \right\} \right)^+.
\end{equation}

We refer to inequality~\eqref{ineq:projected_sri} as the \emph{projected SRI} associated with~$(\alpha, \beta)$ (even when the fixed~$\bar{x}$ is replaced by a free variable~$x$). For convenience, we define~\hypertarget{definition:set_dual_multipliers}{$\mc{A}$} as the set of multipliers~$(\alpha, \beta)$, with~$\alpha \geq 0$ and~$\beta \leq 0$, such that~$\left(\beta^\xi_v + \sum_{\set \subseteq V_+ : v \in \set} \alpha^\xi_\set \right) \leq 0$, for all~$\xi \in [N]$ and~$v \in V_+$ with~$w_v = 0$. The next result builds on Lemma~\ref{lemma:projection_single_inequality} to prove that projected SRIs characterize~$\proj_{\theta} (\srihat(\bar{x}))$. The proof is left to Appendix~\ref{appendix:proof_proj_gamma_characterization}.

\begin{restatable}{proposition}{propositionprojgammacharacterization}
    \label{proposition:proj_gamma_characterization}
    For each~$\bar{x} \in \mc{X}$,~$\proj_{\theta}(\srihat(\bar{x}))$ contains precisely the set of all~$\bar{\theta} \in \R^{V_+}_+$ satisfying the projected SRIs~\eqref{ineq:projected_sri} (with~$x$ fixed to~$\bar{x}$) associated with every~$(\alpha, \beta) \in \dualset$.
    Consequently,
    {\small
    \begin{equation*}
    \P =
    \left\{(x, \theta) \in \mc{X} \times \R^{V_+}_+ :~
    \begin{aligned}
        & \sum_{v \in V_+ : w_v > 0} \phiv(\alpha, \beta) \, \theta_v \geq \sum_{\xi \in [N]} \sum_{\set \subseteq V_+} \alpha^\xi_\set \, x(E(\set)) + \nufunction{\alpha, \beta}, & \forall (\alpha, \beta) \in \dualset
    \end{aligned}
    \right\}.
    \end{equation*}}
\end{restatable}

\begin{example}
\label{example:simple_projected_sri}

Suppose that~$\mc{X} = \Xcvrp$ and
let~$\emptyset \subsetneq \set \subseteq V_+$ be such that~$w_v > 0$, for all~$v \in \set$. Let~$\Xi = \{\xi \in [N] : \kscen{\set} > \kbar{\set}\}$ and assume that~$\Xi$ is nonempty. Set~$\beta = \mathbf{0}$ and~$\alpha^\xi_{\set'} = \mathbb{I}(\set' = \set \text{~and~} \xi \in \Xi)$, for all~$\set' \subseteq V_+$ and~$\xi \in [N]$. The pair~$(\alpha, \beta)$ yields the following aggregated SRI~(recall Fact~\ref{fact:aggregated_sri}):
\begin{equation*}
    \label{ineq:example_simple_sri}
    \sum_{\xi \in \Xi} y^\xi(\set) \geq \sum_{\xi \in \Xi} (\kscen{\set} + x(E(\set)) - |\set|) = |\Xi| \cdot (x(E(\set)) - |\set|) + \sum_{\xi \in \Xi} \kscen{\set}
\end{equation*}

\noindent
By Lemma~\ref{lemma:projection_single_inequality}, for any~$\bar{x} \in \mc{X}$,~$\bar{\theta}$ belongs to~$\proj_{\theta}\Gammaset(\bar{x}, \alpha, \beta)$ if and only if~$\bar{\theta}$ satisfy the projected SRI:
\begin{equation}
    \label{ineq:example_simple_projected_sri}
    \sum_{v \in \set} \max_{\xi \in \Xi}\left\{\frac{1}{p_\xi w_v}\right\} \, \theta_v \geq |\Xi| \cdot (x(E(\set)) - |\set|) + \sum_{\xi \in \Xi} \kscen{\set}.
\end{equation} 
We call inequality~\eqref{ineq:example_simple_projected_sri} a \emph{projected aggregated SRI} associated with~$\Xi$ and~$\set$.
\hfill\Halmos
\end{example}

\subsection{\mytitle{Constructing projected SRIs that dominate ILS cuts}}
\label{subsection:comparing_cuts}

In this subsection, we show that, under some conditions on the recourse lower bounds, the ILS cuts shown in Section~\ref{subsection:ils} (applied to the scenario-optimal recourse function~$\optQ$) are valid for~$\P$. The overall proof strategy consists of two main steps: (1) leverage Formulation~\eqref{formulation:recourse_scen_opt} to express a recourse lower bound as the optimal value of an LP; and (2) use the dual multipliers~$(\alpha, \beta)$ of this LP and Lemma~\ref{lemma:projection_single_inequality} to construct a dominating projected SRI.

\subsubsection{\mytitle{Set cuts}}
\label{subsection:set_cuts_comparison}

The approach of~\cite{part1} to compute a recourse lower bound for the set cut~\eqref{ineq:set_cut} can be interpreted as a simple greedy algorithm. Fix~$\set = \{v_1, \ldots, v_\ell\} \subseteq V_+$ and a scenario~$\xi \in [N]$. Let~$k' \in \Z_{++}$ be such that inequality~$\bar{x}(E(\set)) \leq |\set| - k'$ is valid for~$\mc{X} \cap \Z^E$, and let~$\bar{x} \in \mc{X} \cap \Z^E$ satisfy this inequality at equality. Additionally, suppose that~$0 < \kscen{\set} - k' \leq b(\set)$ (see Remark~\ref{remark:feasibility_set_cut}).

To obtain a lower bound on the recourse cost, we assign the~$\kscen{\set} - k'$ failures observed by~$\bar{x}$ to the customers in~$\set$ in increasing order of their recourse weights. Thus, assume that~$w_{v_1} \leq \ldots \leq w_{v_\ell}$, and let~$j \in [\ell]$ be the smallest index such that~$\kscen{\set} - k' \leq \sum_{i \in [j]} b_{v_i}$. A lower bound on the recourse cost paid in scenario~$\xi \in [N]$ is then given by~$\mc{L}^*_\xi(\set, k') = \sum_{i \in [j - 1]} b_{v_i} w_{v_i} + ((\kscen{\set} - k') - \sum_{i \in [j - 1]} b_{v_i})\, w_{v_j}$. (Recall that~$[a] = \emptyset$, for every integer~$a \leq 0$.)

It is easy to see that this greedy algorithm is equivalent to the LP:
\begin{subequations}
\label{formulation:recourse_lower_bound_set}
\begin{align} 
\mc{L}_\xi^*(\set, k') \coloneqq \min ~~& \sum_{v \in V_+} w_v y^\xi_v, & \nonumber \\
\text{s.t.~~} & y^\xi(\set) \geq \kscen{\set} - k', && \quad\quad\quad (\alpha^\xi_\set) \\
& y^\xi_v \leq b_v, & \forall v \in V_+, & \quad\quad\quad (\beta^\xi_v) \\
& y^\xi \geq 0.
\end{align}
\end{subequations}

\noindent
\cite{part1} set their recourse lower bound as~$\mc{L}^*(\set, k') \coloneqq \sum_{\xi \in [N]} p_\xi \, \mc{L}^*_\xi(\set, k')$ (they use~$\Qc$, so they assume~$b = \allones$). Note that~$\mc{L}^*_\xi(\set, k') = 0$ whenever~$\kscen{\set} \leq k'$.

Replacing this recourse lower bound in the ILS cut~\eqref{ineq:set_cut1} we obtain the set cut
\begin{equation}
    \label{ineq:set_cut}
    \theta(\set) \geq \mc{L}^*(\set, k') \cdot \Wdl(x ; \mc{X}(\set, k')),
\end{equation}
where we recall from Section~\ref{subsection:ils} that~$\Wdl(x ; \mc{X}(\set, k')) = 1 + (x(E(\set)) - |\set| + k')$.

\begin{remark}
    \label{remark:feasibility_set_cut}Formulation~\eqref{formulation:recourse_lower_bound_set} is feasible if and only if~$\kscen{\set} - k' \leq b(\set)$. This condition always holds when~$b \geq \allones$, since, by Assumption~\ref{assumption:large_demands},~$\kscen{\set} = \lceil d^\xi(\set) / C \rceil \leq |\set| \leq b(\set)$. In fact, choosing~$b$ as in Definition~\ref{definition:scen_opt_policy} already guarantees feasibility. To see this, suppose that there exists~$\bar{x} \in \mc{X}(\set, k')$ (otherwise, the set cut~\eqref{ineq:set_cut} is always inactive). For each route~$R \in \mc{R}(\bar{x})$, there exists~$\bar{y}_R \in \policy(R) \cap [\mathbf{0}, b]^N$. Combining these vectors with Fact~\ref{fact:policy_split} we obtain~$\bar{y} \in \policy(\bar{x}) \cap [\mathbf{0}, b]^N$, and since SRIs (with~$x$ fixed to~$\bar{x}$) are valid for~$\policy(\bar{x})$, we get~$b(\set) \geq \bar{y}^\xi(\set) \geq \kscen{\set} + \bar{x}(E(\set)) - |\set| = \kscen{\set} - k'$.
    \hfill\Halmos
\end{remark}

One may note that Formulation~\eqref{formulation:recourse_lower_bound_set} has a similar structure to the inner problems found in robust optimization problems with budget uncertainty sets~\citep{bertsimas2003robust, bertsimas2004price}. Thus, analogously to known results in the robust optimization literature, we can write closed-form expressions for an optimal dual solution for Formulation~\eqref{formulation:recourse_lower_bound_set}. Combining this with the general strategy outlined at the beginning of this subsection yields the desired results. The proofs are deferred to Appendix~\ref{appendix:proof_theorem_dominance_set_cut}, and a simple example of the dominating projected SRI is offered in Appendix~\ref{appendix:example_set_cut}.

\begin{restatable}{lemma}{lemmarobustdual}
    \label{lemma:robust_dual}
    Let~$w \in \Q^{V_+}_+$ and~$b \in \Z^{V_+}_+$. Fix a scenario~$\xi \in [N]$ and let~$\emptyset\subsetneq \set \subseteq V_+$ be such that~$S = \{v_1, \ldots, v_\ell\}$ and~$w_{v_1} \leq \ldots \leq w_{v_\ell}$. Let~$k' \in \Z_{++}$ be such that~$\kscen{\set} - k' \leq b(\set)$ and let~$j \in [\ell]$ be the smallest index such that~$\kscen{\set} - k' \leq \sum_{i \in [j]} b_{v_i}$. Set~$\bar{\alpha}^\xi_\set = \mb{I}(\kscen{\set} > k') \cdot w_{v_j}$ and~$\bar{\beta}^\xi_v = \mb{I}(v \in \{v_i\}_{i \in [j - 1]}) \cdot (w_v - w_{v_j})$, for every~$v \in V_+$.
    Then~$(\bar{\alpha}^\xi_\set, \bar{\beta}^\xi)$ is optimal for the dual of Formulation~\eqref{formulation:recourse_lower_bound_set} and~$\bar{\alpha}^\xi_\set \leq \mc{L}^*_\xi(\set, k')$.
\end{restatable}

\begin{restatable}{theorem}{theoremdominancesetcut}
    \label{theorem:dominance_set_cut}
    Let~$w \in \Q^{V_+}_+$,~$b \in \Z^{V_+}_+$,~$\emptyset \subsetneq \set \subseteq V_+$ and~$k' \in \Z_{++}$. Assume that Formulation~\eqref{formulation:recourse_lower_bound_set} is feasible for every scenario~$\xi \in [N]$ (see Remark~\ref{remark:feasibility_set_cut}). For each~$\xi \in [N]$, let~$(\bar{\alpha}^\xi_\set, \bar{\beta}^\xi)$ be as specified in Lemma~\ref{lemma:robust_dual}, and set~$\hat{\alpha} \geq 0$ and~$\hat{\beta} \leq 0$ as
    \begin{align*}
        \hat{\alpha}^\xi_{\set'} & = \mb{I}(\set' = \set) \cdot (p_\xi \, \bar{\alpha}^\xi_{\set}), && \forall \set' \subseteq V_+, \xi \in [N], \\
        \hat{\beta}^\xi_v & = p_\xi \, \bar{\beta}^\xi_v, && \forall v \in V_+, \xi \in [N].
    \end{align*}
    Then~$(\hat{\alpha}, \hat{\beta}) \in \dualset$, and for every~$\bar{x} \in \{x \in \mc{X} : x(E(\set)) \leq |\set| - k'\}$, the set cut~\eqref{ineq:set_cut} (with~$x$ fixed to~$\bar{x}$) is valid for~$\proj_{\theta} (\Gammaset(\bar{x}, \hat{\alpha}, \hat{\beta}))$. In particular, if~$x(E(\set)) \leq |\set| - k'$ is valid for~$\mc{X}$, then the set cut~\eqref{ineq:set_cut} (with~$x$ free) is dominated by the projected SRI~\eqref{ineq:projected_sri} associated with~$(\hat{\alpha}, \hat{\beta})$, and is thus valid for~$\P$.
\end{restatable}

\subsubsection{\mytitle{Partial route cuts}}
\label{subsection:partial_route_comparison}

To compute recourse lower bounds for the partial route cut~\eqref{ineq:partial_route_cut2}, we propose an approach based on relaxing Formulation~\eqref{formulation:recourse_scen_opt}. Let~$H = (\set_1, \ldots, \set_\ell)$ be a partial route. We write~$H' \subseteq H$ to refer to a partial route of the form~$H' = (\set_i, \ldots, \set_j)$ for some~$1 \leq i \leq j \leq \ell$. Moreover, for every scenario~$\xi \in [N]$, we use the shorthands~$y^\xi(H') = y^\xi(V_+(H'))$ and~$k_
\xi(H') = \kscen{V_+(H')}$. For each scenario~$\xi \in [N]$, let~$\mc{L}_\xi^*(H)$ be the optimal value of the LP:
\begin{subequations}
\label{formulation:recourse_lower_bound}
\begin{align} 
\mc{L}_\xi^*(H) \coloneqq \min ~~& \sum_{v \in V_+} w_v y^\xi_v, & \nonumber \\
\text{s.t.~~} & y^\xi(H') \geq \kscen{H'} - 1, && \forall H' \subseteq H, & ~~~~~~(\alpha^\xi_{H'}) \label{ineq:formulation_recourse_lower_bound1} \\
& y^\xi_v \leq b_v, && \forall v \in V_+, &~~~~~~(\beta^\xi_v) \label{ineq:formulation_recourse_lower_bound2} \\
& y^\xi \geq 0,
\end{align}
\end{subequations}
and define~$\mc{L}^*(H) \coloneqq \sum_{\xi \in [N]} p_\xi  \, \mc{L}^*_\xi(H)$. 

To verify that Formulation~\eqref{formulation:recourse_lower_bound} yields a relaxation of Formulation~\eqref{formulation:recourse_scen_opt}, let~$R$ be a route adhering to~$H$ (recall the definition of adherence in Section~\ref{subsection:ils}). Note that the constraints in Formulation~\eqref{formulation:recourse_lower_bound} (with respect to~$H$) are also included in Formulation~\eqref{formulation:recourse_scen_opt} (with respect to~$R$). In addition, by the choice of~$b$ in Definition~\ref{definition:scen_opt_policy}, we know that~$\policy(R) \cap [\mathbf{0}, b]^N \neq \emptyset$, so Formulation~\eqref{formulation:recourse_lower_bound} is always feasible.

Replacing~$\mc{L}^*(H)$ into the partial route cut~\eqref{ineq:partial_route_cut2} yields
\begin{equation}
    \label{ineq:partial_route_cut}
    \theta(V_+(H)) \geq \mc{L}^*(H) \cdot \Wof(x ; \supsetXh{H}).
\end{equation}
It is not hard to see that inequality~\eqref{ineq:partial_route_cut} is valid for~$\Fq{\mc{X}, \optQ}$. Indeed, this follows from the monotonicity of the disaggregation of~$\optQ$ (Corollary~\ref{corollary:monotonicity}) together with the fact that Formulation~\eqref{formulation:recourse_lower_bound} provides a relaxation of Formulation~\eqref{formulation:recourse_scen_opt} for any route~$R$ that adheres to~$H$. Rather than providing more details on this argument, we establish validity by proving that inequality~\eqref{ineq:partial_route_cut} is valid for~$\P \supseteq \Fq{\mc{X}, \optQ}$. Specifically, using the optimal dual variables~$(\alpha^\xi, \beta^\xi)$ associated with Formulation~\eqref{formulation:recourse_lower_bound}, we prove the following result in Appendix~\ref{appendix:proof_comparison}.
\begin{restatable}{theorem}{theoremdominancesri}
    \label{theorem:dominance_sri}
    Let~$w \in \Q^{V_+}_+$ and suppose that~$b \in \Z^{V_+}_+$ is such that~$\policy(R) \cap [\mathbf{0}, b]^N \neq \emptyset$, for every route~$R$. Let~$H$ be a partial route. Then there exist~$(\alpha, \beta) \in \dualset$ such that the projected SRI~\eqref{ineq:projected_sri} associated with~$(\alpha, \beta)$ dominates the partial route cut~\eqref{ineq:partial_route_cut}.
\end{restatable}

Combining Theorem~\ref{theorem:dominance_sri} with the results in~\cite{part1} yields the following consequences.
\begin{corollary}
    \label{corollary:dominance_sri}
    Let~$w \in \Q^{V_+}_+$ and suppose that~$b \in \Z^{V_+}_+$ is such that~$\policy(R) \cap [\mathbf{0}, b]^N \neq \emptyset$, for every route~$R$.
    The following inequalities are dominated by projected SRIs (and thus valid for~$\P$):
    \begin{enumerate}[(a)]
        \item the partial route cuts of~\cite{part1}:~$\theta(V_+(H)) \geq \mc{L}^*(H) \cdot \Whs(x ; \Xh{H})$, for every partial route~$H$; \label{item:corollary_partial_route}
        \item the path cuts of~\cite{parada2024disaggregated}:~$\theta(V_+(R)) \geq \optQ(R) \cdot \Wof(x ; \supsetXh{H})$, for every route~$R$; \label{item:corollary_path_cuts}
        \item the ILS cuts of~\cite{gendreau201650th} (which assume that~$\mc{X} \subseteq \{x \in \R^E : x(\delta(0)) = 2k\}$):\\
        $\allones^\T \theta \geq \left(\sum_{R \in \mc{R}(\bar{x})} \optQ(R) \right) \cdot \left(1 - |V_+| + k + \sum_{e \in E \setminus \delta(0) : \bar{x}_e = 1} x_e \right)$, for every~$\bar{x} \in \mc{X} \cap \Z^E$. \label{item:corollary_gendreau_cuts}
    \end{enumerate}
\end{corollary}
\begin{proof}
    Item~\ref{item:corollary_partial_route} holds because~$\Wof(\bar{x} ; \supsetXh{H}) \geq \Whs(\bar{x} ; \Xh{H})$ for all~$\bar{x} \in \mc{X}$. For item~\ref{item:corollary_path_cuts}, take a route~$R = (v_1, \ldots, v_\ell)$ and note that~$\mc{L}^*(H) = \optQ(R)$, so we are done by Theorem~\ref{theorem:dominance_sri}. Finally, Claim~2 of~\cite{part1} shows that the path cuts dominate those of~\cite{gendreau201650th}, proving item~\ref{item:corollary_gendreau_cuts}.
\end{proof}

\subsection{\mytitle{Recovering the higher-dimensional LP relaxation bound with a single projected SRI}}
\label{subsection:lp_bound}

As discussed in the beginning of Section~\ref{section:projection} (and suggested by the theoretical results from the previous section), the LP relaxation bound of Formulation~\eqref{formulation:sri}, here denoted as~$z^*$, can be fairly strong, but solving the formulation to integrality is challenging due to the large number of~$y$-variables. In this subsection, we show how to recover~$z^*$ in the~$(x,\theta)$-space using a single projected SRI, thus obtaining strong root-node bounds without having the~$y$-variables in the model.

The approach consists of first solving the LP relaxation of Formulation~\eqref{formulation:sri}, and then using the optimal dual variables associated with the SRIs and upper bound constraints on the~$y$-variables to construct the desired projected SRI. We remark that the ideas here were inspired by recent advances on recovering the Dantzig–Wolfe bound via cutting planes~\citep{chen2024recovering, corner}.

We start by replacing every occurrence of~$\theta_v$ in Formulation~\eqref{formulation:sri} with the expression~$\sum_{\xi \in [N]} p_\xi w_v y^\xi_v$. We then dualize the SRIs and the upper bound constraints on the~$y$-variables into the objective function of the resulting formulation. This yields the Lagrangian dual:
{\small
\begin{equation}
\label{problem:lagrangian_dual}
z^* = \max_{\alpha \geq 0, \beta \leq 0} \left\{ \min_{x \in \mc{X}, y \geq 0} \left\{
\begin{aligned} 
& c^\T x + \sum_{v \in V_+} \sum_{\xi \in [N]} p_\xi w_v y^\xi_v + \sum_{\xi \in [N]} \sum_{\set \subseteq V_+} \alpha^\xi_\set (\kscen{\set} + x(E(\set)) - |\set| - y^\xi(\set)) \\
& + \sum_{\xi \in [N]} \sum_{v \in V_+} \beta^\xi_v (b_v - y^\xi_v)
\end{aligned} \right\} \right\}.
\end{equation}}
Since problem~\eqref{problem:lagrangian_dual} is separable, we write~$\lpopt = \max_{\alpha \geq 0, \beta \leq 0} \{\sigmax{\alpha} + \sigmay{\alpha, \beta} + \nufunction{\alpha, \beta}\}$, where
\begin{align}
    & \sigma_x(\alpha) \coloneqq \min_{x \in \mc{X}} \left\{ \sum_{\xi \in [N]} \sum_{uv \in E} \left( c_{uv} + \sum_{\set \subseteq V_+ : u,v \in \set} \alpha^\xi_\set \right) x_{uv} \right\}, \tag{$\sigma_x$} \label{definition:sigma_x} \\
    & \sigma_y(\alpha, \beta) \coloneqq \min_{y \geq 0} \left\{ \sum_{\xi \in [N]} \sum_{v \in V_+} \left(p_\xi w_v - \beta^\xi_v - \sum_{\set \subseteq V_+ : v \in \set} \alpha^\xi_\set \right) y^\xi_v \right\}, \tag{$\sigma_y$} \label{definition:sigma_y}
\end{align}
and recall that~$\nu(\alpha, \beta) = \sum_{\xi \in [N]} \sum_{\set \subseteq V_+} \alpha^\xi_\set ( \kscen{\set} - |\set| ) + \sum_{\xi \in [N]} \sum_{v \in V_+} \beta^\xi_v b_v$, as defined in page~\pageref{eq:nu}.

The following theorem is proved in Appendix~\ref{appendix:proof_theorem_single_inequality}.
\begin{restatable}{theorem}{theoremsingleinequality}
    \label{theorem:single_inequality}
    For every~$\alpha \geq 0$ and~$\beta \leq 0$, 
    $$\min\{c^\T x + \allones^\T \theta : x \in \mc{X},~\theta \in \proj_{\theta}(\Gammaset(x, \alpha, \beta))\} \geq \sigmax{\alpha} + \sigmay{\alpha, \beta} + \nufunction{\alpha, \beta}.$$
\end{restatable}

\begin{corollary}
    \label{corollary:single_inequality_optimal}
    Suppose that we solve the LP relaxation of Formulation~\eqref{formulation:sri} to optimality. Let~$\alpha^* \geq 0$ and~$\beta^* \leq 0$ be optimal dual solutions  associated with the SRIs and upper bound constraints on the~$y$-variables, respectively. Then~$\lpopt = \min\{c^\T x + \allones^\T \theta : x \in \mc{X},~\theta \in \proj_{\theta}(\Gammaset(x, \alpha^*, \beta^*))\}$.
\end{corollary}
\begin{proof}
    Let~$z' = \min\{c^\T x + \allones^\T \theta : x \in \mc{X},\ \theta \in \proj_{\theta}(\Gammaset(x, \alpha^*, \beta^*))\}$. By Proposition~\ref{proposition:proj_gamma_characterization},~$\P \subseteq \{(x, \theta) : x \in \mc{X},~\theta \in \proj_{\theta}(\Gammaset(x, \alpha^*, \beta^*))\}$, so~$\lpopt \geq z'$. For the converse, apply Theorem~\ref{theorem:single_inequality} to obtain~$z' \geq \sigmax{\alpha^*} + \sigmay{\alpha^*, \beta^*} + \nufunction{\alpha^*, \beta^*}$. We are now done by the fact that~$(\alpha^*, \beta^*)$ is optimal for the outer problem in the Lagrangian dual~\eqref{problem:lagrangian_dual} (\cite{frangioni2005lagrangian} has a nice proof via ``partial dualization'').
\end{proof}

\subsection{\mytitle{Branch-and-cut algorithm for the VRPSD under a scenario-optimal recourse policy}}
\label{subsection:sri_algorithm}

We now propose a two-phase branch-and-cut algorithm for solving problem~$\vrpsd{\mc{X}, \optQ}$. First, we solve the LP relaxation of Formulation~\eqref{formulation:sri}  and construct projected SRIs using the results from Sections~\ref{subsection:projected_sri} and~\ref{subsection:lp_bound}. This allows us to recover strong bounds in the~$(x, \theta)$-space. In the second step, we solve the problem to integrality by separating the projected SRIs developed in Sections~\ref{subsection:projected_sri} and~\ref{subsection:comparing_cuts}. Since these projected SRIs dominate the path cuts of~\cite{parada2024disaggregated} (Corollary~\ref{corollary:dominance_sri}), it follows from Theorem~1 of~\cite{part1} that this procedure yields an exact algorithm for solving~$\vrpsd{\mc{X}, \optQ}$. We describe the algorithm in more detail next.

\noindent
\textbf{\textit{\mytitle{Step 1: Solving the high-dimensional LP to obtain a strong root node relaxation.}}} \label{step1} We use a cutting-plane procedure to (approximately) solve the LP relaxation of Formulation~\eqref{formulation:sri}. We start the LP with the variable bounds and the degree constraints in~$\mc{X}$. Given a candidate solution~$(\bar{x}, \bar{\theta})$, we first try to separate RCIs (or SECs) using the CVRPSEP package of~\cite{Lysgaard2004}, obtaining a collection~$\mc{S}$ of subsets~$S \subseteq V_+$. For each~$S \in \mc{S}$, we try to separate an aggregated SRI using Fact~\ref{fact:aggregated_sri}. If~$\mc{S} = \emptyset$, we instead call the (aggregated) SRI separation heuristics described in Appendix~\ref{appendix:separation}. We run this procedure until no violated cut is found or 60 seconds have elapsed.

Let~$\tilde{z} \leq \lpopt$ be the LP bound obtained with the previous procedure and let~$(\tilde{\alpha}, \tilde{\beta})$ be the corresponding optimal dual multipliers associated with the SRIs and the upper bound constraints on the~$y$-variables. Recall that aggregated SRIs have the form~$\sum_{\xi \in \Xi} y^\xi(\set) \geq \sum_{\xi \in \Xi} (\kscen{\set} + x(E(\set)) - |\set|)$, for~$\Xi \subseteq [N]$ and~$\set \subseteq V_+$. Thus, rather than having duals for each SRI, we have dual multipliers~$\Lambda^\Xi_\set$ for the aggregated SRIs. To build~$\tilde{\alpha}$ from~$\Lambda$, we may set~$\tilde{\alpha}^\xi_\set = \sum_{\Xi \subseteq [N] : \xi \in \Xi} \Lambda^\Xi_\set$, for each~$\xi \in [N]$ and~$\set \subseteq V_+$. We include this construction here only to write the formulations and inequalities in terms of~$\tilde{\alpha}$ instead of~$\Lambda$. Our implementation works directly with~$\Lambda$ and builds the projected SRI associated with~$(\tilde{\alpha}, \tilde{\beta})$ without explicitly constructing~$\tilde{\alpha}$.

Let~$\tilde{\mc{X}}$ be the relaxation of~$\mc{X}$ obtained by considering only the RCIs (or SECs, if~$\mc{X} = \Xsub$) that were separated while solving the LP relaxation of Formulation~\eqref{formulation:sri}. Moreover, let~$\mc{U} \coloneqq \{(\set,\Xi) : \Lambda^\Xi_\set > 0~\text{and~$w_v > 0$ for all~$v \in \set$}\}$. We then solve the following relaxation:
\begin{equation}
    \label{formulation:root_node}
    \min \left\{c^\T x + \allones^\T \theta : ~\text{$(x, \theta) \in \tilde{\mc{X}} \times \R^{V_+}_+$ and~$(x, \theta)$ satisfies~\eqref{ineq:example_simple_projected_sri}, for all~$(S, \Xi) \in \mc{U}$}\right\}.
\end{equation}

\noindent
We note in passing that preliminary experiments showed that it is better to use the projected aggregated SRIs~\eqref{ineq:example_simple_projected_sri} than the inequalities developed in Theorem~\ref{theorem:dominance_set_cut}.

If the optimal value of Formulation~\eqref{formulation:root_node} is at least~$\tilde{z}$, we proceed to~\hyperref[step2]{Step 2}. Otherwise, we add the projected SRI~\,$\sum_{v \in V_+ : w_v > 0} \phiv(\tilde{\alpha}, \tilde{\beta}) \, \theta_v \geq \sum_{\xi \in [N]} \sum_{\set \subseteq V_+} \tilde{\alpha}^\xi_\set \, x(E(\set)) + \nufunction{\tilde{\alpha}, \tilde{\beta}}$. Note that~$\tilde{z}$ is just a lower bound on~$\lpopt$. Nevertheless, Theorem~\ref{theorem:single_inequality} still implies that the projected SRI associated with~$(\tilde{\alpha}, \tilde{\beta})$ recovers the bound~$\tilde{z}$ in the~$(x, \theta)$-space (even if we replace~$\mc{X}$ by~$\tilde{\mc{X}}$).

One may wonder why we add the projected aggregated SRIs~\eqref{ineq:example_simple_projected_sri} if Theorem~\ref{theorem:single_inequality} already guarantees that one projected SRI is enough to recover the bound~$\tilde{z}$. Our reasoning here is the same as highlighted in~\citep{chen2024recovering, gamrath2010experiments, corner}: using a single inequality to recover a strong bound can be detrimental to the MIP solver performance, since it frequently leads to a high-dimensional optimal face. By adding several projected SRIs, we try to mitigate this issue.

\noindent
\textbf{\textit{\mytitle{Step 2: Branch-and-cut algorithm based on the separation of projected SRIs.}}} \label{step2} After constructing a strong root-node relaxation in~\hyperref[step1]{Step 1}, we proceed similarly to the ILS-based algorithm of~\cite{part1}, which separates partial route and set cuts. However, whenever a violated partial route cut is found, instead of adding the cut, we use the results from Section~\ref{subsection:comparing_cuts} to separate a dominating projected SRI.

Given a candidate solution~$(\bar{x}, \bar{\theta})$, we use the separation algorithm of~\cite{part1} to (possibly) find a partial route~$H = (\set_1, \ldots, \set_\ell)$ such that~$\Wof(\bar{x}; \supsetXh{H}) > 0$. We then construct a single LP with~$N \cdot |V_+(H)|$ variables to simultaneously solve the  LPs~\eqref{formulation:recourse_lower_bound} for every scenario~$\xi \in [N]$. To improve the efficiency, for each scenario~$\xi \in [N]$, we only consider constraints of type~\eqref{ineq:formulation_recourse_lower_bound1} corresponding to \emph{minimal} partial routes~$H' \subseteq H$, meaning that, for every~$H'' \subseteq H'$ with~$H'' \neq H'$,~$\kscen{H''} < \kscen{H'}$. In this way, for each scenario~$\xi \in [N]$, our formulation only uses~$\ell(\kscen{H} - 1)$ constraints of type~\eqref{ineq:formulation_recourse_lower_bound1}, rather than~$\ell^2$.

By solving the previous LP, we obtain dual multipliers~$(\alpha, \beta)$ and an associated projected SRI that we add to the formulation. Theorem~\ref{theorem:dominance_sri} implies that this projected SRI dominates the partial route cut~$\theta(V_+(H)) \geq \mc{L}^*(H) \cdot \Wof(x; \supsetXh{H})$. Therefore, by Theorem~1 of~\cite{part1} this procedure yields a correct branch-and-cut algorithm for~$\vrpsd{\mc{X}, \optQ}$, since if~$\bar{x}$ is integer, then the separation of~\cite{part1} is guaranteed to examine the partial routes~$H = R$, with~$R \in \mc{R}(\bar{x})$.

To improve performance, the algorithm of~\cite{part1} also separates set cuts. Whenever a violated set cut is detected, we instead add a corresponding projected aggregated SRI~\eqref{ineq:example_simple_projected_sri}. Preliminary experiments indicate that these inequalities are more effective in practice than the dominating projected SRI described in Theorem~\ref{theorem:dominance_set_cut}. Further details on our separation procedures are provided in Appendix~\ref{appendix:separation_algorithm}.

\subsection{
\mytitle{Extension to other recourse functions}}
\label{subsection:extension}

Suppose now that~$\mc{Q} \geq \optQ$ is a recourse function satisfying Assumption~\ref{assumption:recourse_lower_bound} (with the same parameters~$w \in \Q^{V_+}_+$ and~$b \in \Z^{V_+}_{++}$ used for defining~$\optQ$). By Propositions~\ref{proposition:valid_cuts_recourse} and~\ref{proposition:proj_gamma_characterization}, we know that projected SRIs are valid for~$\P \supseteq \Fq{\mc{X}, \mc{Q}}$, so we can solve~$\vrpsd{\mc{X}, \mc{Q}}$ by separating additional inequalities on top of the algorithm described in Section~\ref{subsection:sri_algorithm}. 

More precisely, for any candidate solution~$(\bar{x}, \bar{\theta}) \in (\mc{X} \cap \Z^E) \times \R^{V_+}_+$, we can verify if every route~$R \in \mc{R}(\bar{x})$ satisfies~$\sum_{v \in V_+(R)} \bar{\theta}_v \geq \mc{Q}(R)$. If not, we add the partial route cut~$\theta(V_+(R)) \geq \mc{Q}(R) \cdot \Whs(x ; \Xh{R})$ to the formulation. As shown by~\cite{part1}, these cuts suffice to recover the recourse cost of a solution. In our implementation, we also incorporate additional partial route cuts developed in~\cite{part1} for the recourse function~$\Qc$ (see Appendix~\ref{appendix:separation_algorithm}).

\section{\mytitle{Computational experiments}}
\label{section:computation}

We conducted experiments on the VRPSD with scenarios under both the scenario-optimal and classical recourse policies, with and without the assumption of CVRP feasibility. In other words, we considered problems~$\vrpsd{\mc{X}, \mc{Q}}$, for~$\mc{X} \in \{\Xcvrp, \Xsub\}$ and~$\mc{Q} \in \{\optQ, \Qc\}$. Our goal here is twofold. First, we complement the theoretical findings in Section~\ref{subsection:comparing_cuts} by empirically comparing the performance of projected SRIs and ILS cuts with respect to~$\optQ$. Second, since Proposition~\ref{proposition:valid_cuts_recourse} ensures that inequalities valid for~$\optQ$ are also valid for~$\Qc$, we also investigate whether incorporating projected SRIs can improve the performance of the ILS-based algorithm for~$\vrpsd{\mc{X}, \Qc}$ proposed by~\cite{part1}.

\subsection{\mytitle{Experimental setup}}

Our approaches were evaluated on the instances developed by~\cite{part1} for the VRPSD with scenarios, which in turn are based on the 270 instances of~\cite{jabali2014} and the 20 instances of~\cite{Dinh2018}. We implemented three algorithmic approaches:
\begin{itemize}[leftmargin=*]
    \item \tsc{ils:} the ILS-based method proposed by~\cite{part1}, which we applied to both~$\optQ$ and~$\Qc$;
    \item \tsc{sri:} the SRI-based algorithm discussed in Section~\ref{subsection:sri_algorithm}, which only applies to~$\optQ$;
    \item \tsc{ils+sri:} the extension combining projected SRIs with ILS cuts described in Section~\ref{subsection:extension}, which we applied only to~$\Qc$.
\end{itemize}
All algorithms were implemented in C++ using Gurobi~12 as the LP/MIP solver and the Lemon library~\citep{lemon} for basic graph operations. Experiments were executed in single-thread mode with a time limit of 1800 seconds per instance on a machine equipped with an \mbox{Intel(R) Xeon(R) Gold 6142 CPU @ 2.60 GHz.}

\subsection{\mytitle{Numerical results}}

We first consider in Table~\ref{table:results_scenopt} the scenario-optimal recourse function~$\optQ$. Each row represents a different combination of first-stage feasibility region ($\mc{X} \in \{\Xcvrp, \Xsub\}$) and instance set (``Dinh'' or ``Jabali''). The ``Total'' column indicates the total number of instances in each set. For each algorithm, we report the number of instances solved to optimality within the time limit (``Solved''), the average solution time in seconds (``T(s)''), the average optimality gap for unsolved instances (``G(\%)''), and the average root node gap (``RG(\%)''). The average solution times were computed considering only the instances that were solved to optimality by both algorithms. Moreover, the optimality and root gaps were measured relative to the best primal bound found across the algorithms.

{\renewcommand{\arraystretch}{1.2}
\begin{table}[htb!]
\captionsetup{subrefformat=parens,font=normalsize}
\centering
\small
\begin{tabular}{l@{\hskip 1em}l@{\hskip 1em}r@{\hskip 2em}rrrrcrrrr}
\toprule
& & & \multicolumn{4}{c}{\tsc{ils}} & & \multicolumn{4}{c}{\tsc{sri}} \\
\cmidrule(r){5-7} \cmidrule(r){9-12}
$\mc{X}$ & Instance Set & Total & Solved & T(s) & G(\%) & RG(\%) & & Solved & T(s) & G(\%) & RG(\%) \\
\midrule
$\Xcvrp$ & Dinh & 20 & 9 & 204.08 & 8.15 & 7.75 & & 9 & 117.53 & 7.71 & 6.65 \\
$\Xcvrp$ & Jabali & 270 & 267 & 54.02 & 1.56 & 1.73 & & 267 & 92.34 & 0.70 & 1.23 \\
\hline
$\Xsub$ & Dinh & 20 & 0 & -- & 29.00 & 34.39 & & 4 & -- & 9.39 & 12.44 \\
$\Xsub$ & Jabali & 270 & 77 & 211.58 & 6.74 & 8.00 & & 249 & 67.99 & 2.40 & 2.41 \\
\bottomrule
\end{tabular}
\caption{Comparison of ILS and SRI algorithms on VRPSD instances with~$\mc{Q} = \optQ$.}
\label{table:results_scenopt}
\end{table}
}

We see that algorithm~\tsc{sri} consistently achieved smaller root gaps than~\tsc{ils}, which validates our theoretical findings from Section~\ref{subsection:comparing_cuts}, where we proved that the inequalities used by algorithm~\tsc{sri} dominate the ILS cuts used by~\tsc{ils} (since~$\mc{Q} = \optQ$). Although this improved relaxation does not translate into an overall better performance when~$\mc{X} = \Xcvrp$, it yields a substantial improvement when~$\mc{X} = \Xsub$, with~\tsc{sri} solving~176 more instances than~\tsc{ils} in total.

These different behaviors can be explained by the fact that enforcing CVRP feasibility (i.e.,~$x \in \Xcvrp \cap \Z^E$) may significantly restrict the feasible region, strengthening the ILS set cuts and leaving little room for improvement via better recourse approximation. On the other hand, when~$\mc{X}=\Xsub$, the feasible region is typically much larger, which makes the quality of the recourse approximation critical. As pointed out by~\cite{hoogendoorn2025evaluation}, while the CVRP-feasibility assumption is often reasonable, it can be arbitrary and undesirable in certain settings. In such situations, our experiments highlight that projected SRIs provide a tighter recourse approximation without artificially restricting the problem formulation. %

Table~\ref{table:results_classical} shows that a similar situation occurs for the classical recourse policy ($\mc{Q} = \Qc$). The hybrid approach~\tsc{ils+sri} consistently achieved stronger root gaps than~\tsc{ils}, and this improvement translated into an overall better performance of the algorithm. As with the scenario-optimal recourse policy, the gains were much more pronounced for~$\mc{X} = \Xsub$, but algorithm~\tsc{ils+sri} still solved~5 more instances than~\tsc{ils} when~$\mc{X} = \Xcvrp$. These results highlight the value of our generic treatment and Proposition~\ref{proposition:valid_cuts_recourse}, which enables the use of SRIs to recourse functions besides the scenario-optimal one. In this case, we successfully applied (projected) SRIs to the recourse function~$\Qc$.

{\renewcommand{\arraystretch}{1.2}
\begin{table}[htb!]
\captionsetup{subrefformat=parens,font=normalsize}
\centering
\small
\begin{tabular}{l@{\hskip 1em}l@{\hskip 1em}r@{\hskip 2em}rrrrcrrrr}
\toprule
& & & \multicolumn{4}{c}{\tsc{ils}} & & \multicolumn{4}{c}{\tsc{ils+sri}} \\
\cmidrule(r){5-7} \cmidrule(r){9-12}
$\mc{X}$ & Instance Set & Total & Solved & T(s) & G(\%) & RG(\%) & & Solved & T(s) & G(\%) & RG(\%) \\
\midrule
$\Xcvrp$ & Dinh & 20 & 8 & 246.56 & 7.34 & 7.48 & & 11 & 125.14 & 8.67 & 6.38 \\
$\Xcvrp$ & Jabali & 270 & 261 & 47.72 & 1.13 & 1.81 & & 263 & 90.26 & 0.96 & 1.32 \\
\hline
$\Xsub$ & Dinh & 20 & 0 & -- & 30.59 & 35.80 & & 4 & -- & 12.56 & 14.89 \\
$\Xsub$ & Jabali & 270 & 73 & 203.90 & 7.21 & 8.42 & & 217 & 90.24 & 2.87 & 2.83 \\
\bottomrule
\end{tabular}
\caption{Comparison of ILS and SRI algorithms on VRPSD instances with~$\mc{Q} = \Qc$.}
\label{table:results_classical}
\end{table}
}

Lastly, to complement the results from Tables~\ref{table:results_scenopt} and~\ref{table:results_classical}, we show in Figure~\ref{figure:experiments_time} the empirical cumulative distribution of execution times for the algorithms. Specifically, consider Figure~\ref{figure:time_cvrp} and take as an example a point~$p = (p_1, p_2) \in \R^2$ on the line of algorithm~\tsc{(ils+)sri}. Let~$\mc{I}$ be the union of the instance sets of~\cite{jabali2014} and~\cite{Dinh2018}, so~$|\mc{I}| = 290$. Let~$\hat{\mc{I}}(\optQ) \subseteq \mc{I}$ (respectively,~$\hat{\mc{I}}(\Qc) \subseteq \mc{I}$) be the subset of instances that algorithm~\tsc{sri} (respectively,~\tsc{ils+sri}) solved within~$p_1$ seconds using first-stage feasible region~$\Xcvrp$. Point~$p_2$ is then given by the ratio~$p_2 = \frac{|\hat{\mc{I}}(\optQ)| + |\hat{\mc{I}}(\Qc)|}{2|\mc{I}|}$. Figure~\ref{figure:time_basic} was constructed similarly for~$\mc{X} = \Xsub$. 

These plots confirm that projected SRIs are particularly effective when no CVRP-feasibility assumption is imposed. We also point out that the ``\tsc{(ils+)sri}'' curve exhibits slow progress in the first seconds. This is due to the procedure described in Step~1 of Section~\ref{subsection:sri_algorithm}, where we solve a high-dimensional formulation in the~$(x, y)$-space for 60 seconds. Overall, after~200 seconds and with~$\mc{X} = \Xcvrp$, we see no significant difference between the performance of ILS and SRI-based methods. More detailed computational results can be found in Appendix~\ref{appendix:experiments}.

\begin{figure}[htb]
    \centering
    \subfloat[Execution time for~$\Xcvrp$.]{\includegraphics[width=0.44\textwidth]{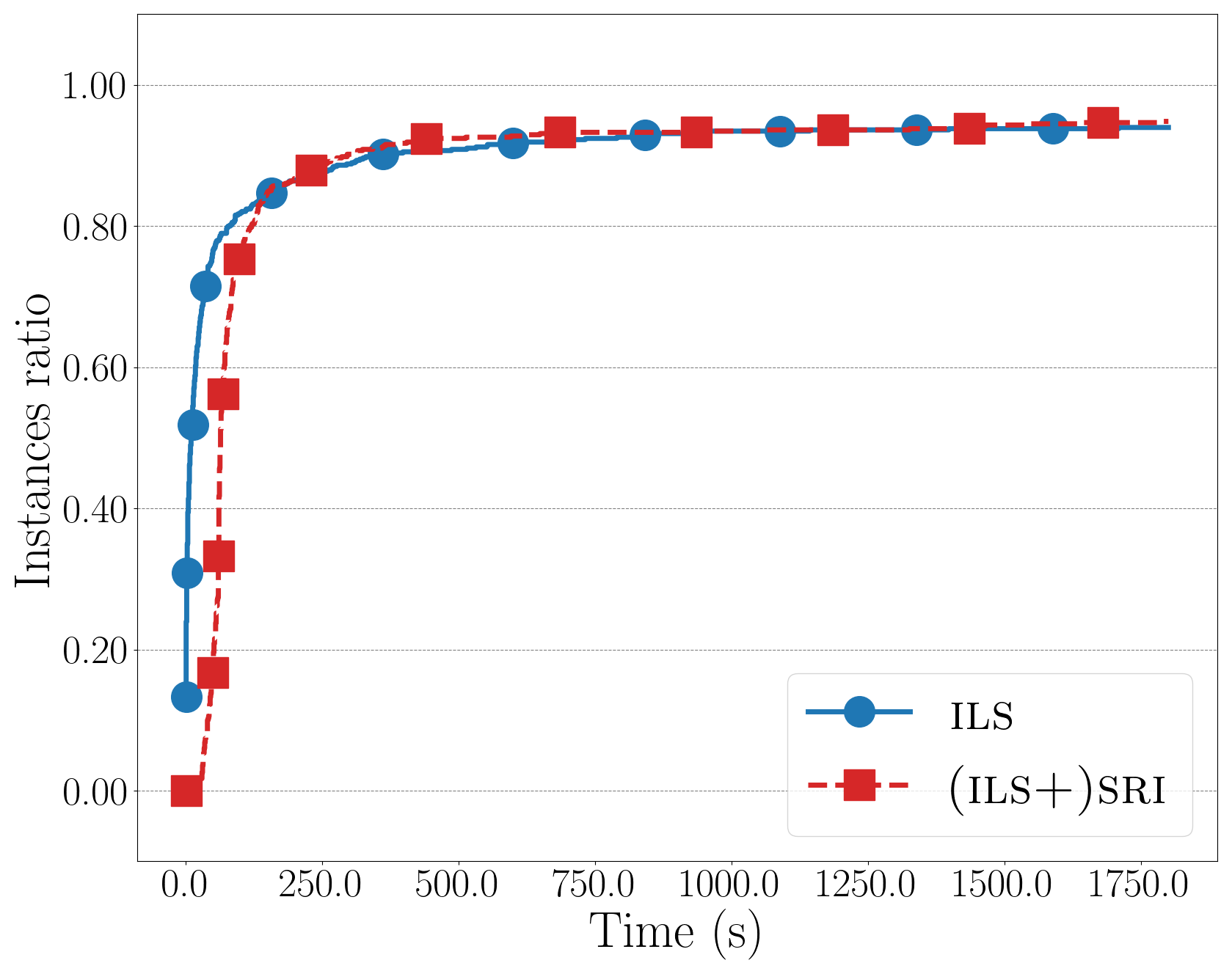} \label{figure:time_cvrp}}
    \hfill
    \subfloat[Execution time for~$\Xsub$.]{\includegraphics[width=0.44\textwidth]{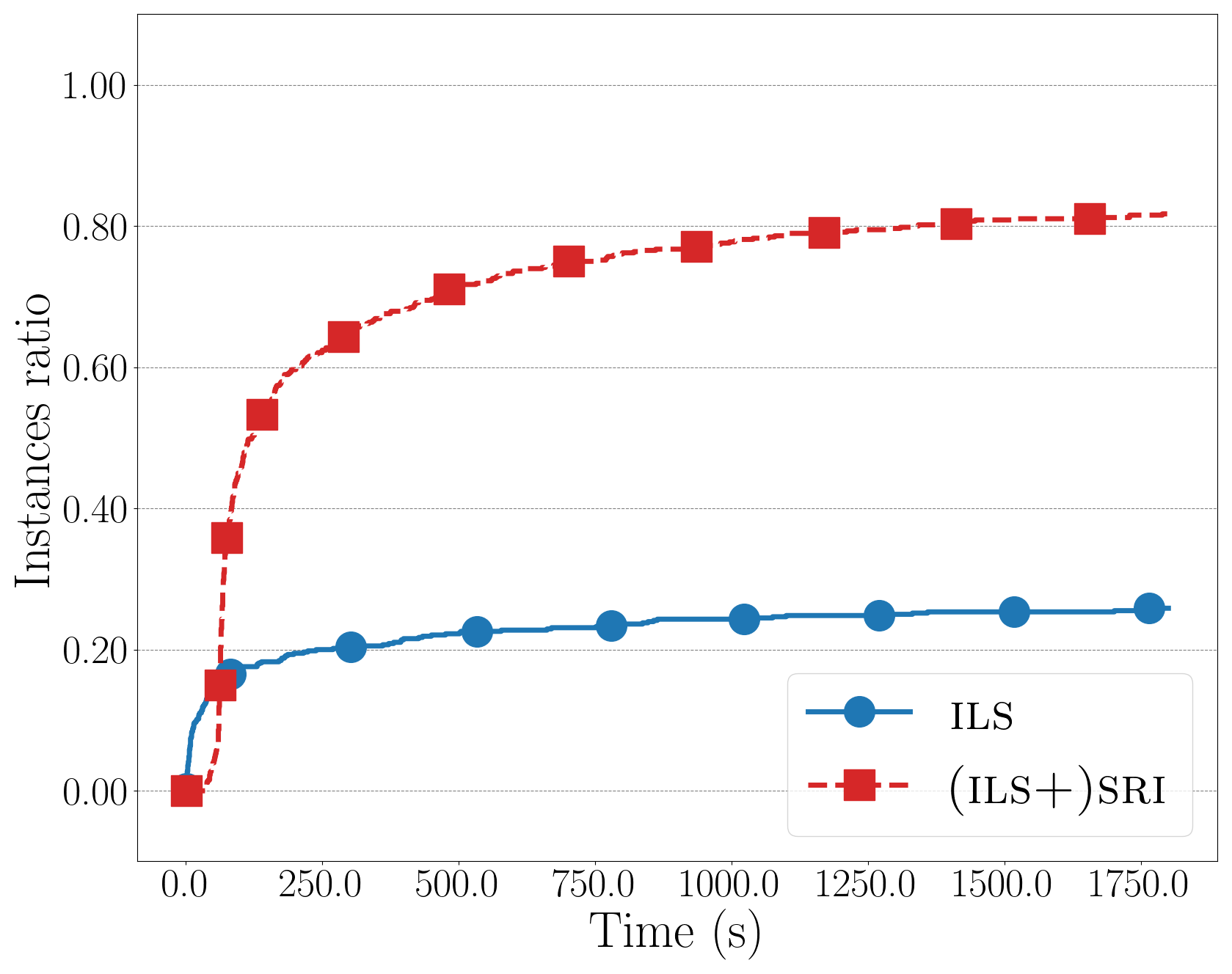}\label{figure:time_basic}}
    \\[0.5cm]
    \caption{Empirical cumulative distribution of the execution times. The legend~\tsc{(ils+)~sri} refers to algorithm~\tsc{sri} when~$\mc{Q} = \optQ$, and to algorithm~\tsc{ils+sri} when~$\mc{Q} = \Qc$.}
    \label{figure:experiments_time}
\end{figure}

\section{\mytitle{Conclusion and future work}}
\label{section:conclusion}

In this paper, we introduced a framework for the VRPSD with scenarios in which recourse policies are represented as feasible solutions of a flow-based MIP. We characterized the convex hull of such recourse policies and introduced a new class of inequalities called scenario recourse inequalities (SRIs). We showed that SRIs yield a formulation for the VRPSD under a scenario-optimal recourse policy, and we further demonstrated that they are valid for other recourse policies and dominate certain ILS cuts.

Our computational results show that, compared to ILS cuts, SRIs provide a better approximation of the recourse function and are particularly useful when the assumption of CVRP feasibility is not imposed. Since we can now better approximate the recourse function, there is less need to rely on CVRP feasibility, which imposes a somewhat arbitrary hard capacity constraint on the sum of expected demands (see~\cite{hoogendoorn2025evaluation}). Using our results, one could instead consider more flexible approaches, such as reducing the occurrence of failures by assigning larger costs to the recourse actions. We leave this direction as future work.

Overall, we have shown that modeling recourse decisions via a network flow formulation can be useful and yields a polyhedral structure that can be exploited. For the future, it would be interesting to study the separation of projected SRIs and to explore similar ideas in other stochastic routing problems where the recourse policy involves other types of resource collection/consumption along routes.

\bibliographystyle{plainnat}
\bibliography{bibliography}

\begin{APPENDICES}

\section{Proof of Lemma~\ref{lemma:symmetric_flow}}
\label{appendix:proof_symmetric_flow}

\lemmasymmetricflow*
\begin{proof}
Let~$\bar{y}^\xi \in \action^\xi(\oa{R})$. We first show that~$\bar{y}^\xi$ also belongs to~$\action^\xi(\cev{R})$. Using Definition~\ref{definition:recourse_policy}, we know that there exist~$\overrarrow{f}^\xi$ and~$\bar{g}^\xi$ satisfying Formulation~\eqref{formulation:flow_feasibility}. Construct vector~$\cev{f}^\xi \in \R^{A}_+$ by setting~$\cev{f}^\xi_{a} = C - \overrarrow{f}^\xi_{(v_i, v_{i + 1})}$ if~$a = (v_{i + 1}, v_i)$ (for some~$i \in \{0, \ldots, \ell\}$) and~$\cev{f}^\xi_a = 0$ otherwise. Clearly,~$\cev{f}^\xi_a \in [0, C]$, for all~$a \in A$. Moreover, for each~$i \in [\ell]$,
\begin{align}
    & \overrarrow{f}^\xi_{(v_{i - 1}, v_i)} + d^\xi(v_i) = \overrarrow{f}^\xi_{(v_i, v_{i + 1})} + g^\xi_{v_i} \nonumber \\
    \iff & C - \overrarrow{f}^\xi_{(v_i, v_{i + 1})} + d^\xi(v_i) = C - \overrarrow{f}^\xi_{(v_{i - 1}, v_i)} + g^\xi_{v_i} \nonumber \\
    \iff & \cev{f}^\xi_{(v_{i + 1}, v_i)} + d^\xi(v_i) = \cev{f}^\xi_{(v_i, v_{i - 1})} + g^\xi_{v_i}, \nonumber
\end{align}
so the pair~$(\cev{f}^\xi, \bar{g}^\xi)$ verifies that~$\bar{y}^\xi \in \action^\xi(\cev{R})$. Now take the convex combination~$\bar{f}^\xi = (\overrarrow{f}^\xi + \cev{f}^\xi) / 2$ and observe that~$(\bar{f}^\xi, \bar{g}^\xi)$ is feasible for Formulation~\eqref{formulation:flow2_feasibility}, proving one direction of the statement.

For the converse, suppose now that~$(\bar{f}^\xi, \bar{g}^\xi)$ is feasible for Formulation~\eqref{formulation:flow2_feasibility} with~$y^\xi = \bar{y}^\xi$ integer. Construct a vector~$\overrarrow{f}^\xi \in \R^A_+$ by setting~$\overrarrow{f}^\xi_{a} = \bar{f}^\xi_{(v_i, v_{i + 1})} + \left(\frac{C}{2} -  \bar{f}^\xi_{(v_{i + 1}, v_i)} \right)$ if~$a = (v_i, v_{i + 1})$ (for some~$i \in \{0, \ldots, \ell\}$) and~$\overrarrow{f}^\xi_a = 0$ otherwise. Note that~$\overrarrow{f}^\xi_a \in [0, C]$, for every~$a \in A$. Additionally, for each~$i \in [\ell]$,
\begin{align}
    & \bar{f}^\xi_{(v_{i - 1}, v_i)} + \bar{f}^\xi_{(v_{i + 1}, v_i)} + d^\xi(v_i) = \bar{f}^\xi_{(v_i, v_{i + 1})} + \bar{f}^\xi_{(v_i, v_{i - 1})} + \bar{g}^\xi_{v_i} \nonumber \\
    \iff & \bar{f}^\xi_{(v_{i - 1}, v_i)} + \left(\frac{C}{2} - \bar{f}^\xi_{(v_i, v_{i - 1})} \right) + d^\xi(v_i) = f^\xi_{(v_i, v_{i + 1})} + \left(\frac{C}{2} - \bar{f}^\xi_{(v_{i + 1}, v_i)} \right) + \bar{g}^\xi_{v_i} \nonumber \\
    \iff & \overrarrow{f}^\xi_{(v_{i - 1}, v_i)} + d^\xi(v_i) = \overrarrow{f}^\xi_{(v_i, v_{i + 1})} + \bar{g}^\xi_{v_i}, \nonumber
\end{align}
so~$(\overrarrow{f}^\xi, \bar{g}^\xi)$ is feasible for Formulation~\eqref{formulation:flow_feasibility} and~$\bar{y}^\xi \in \action^\xi(\oa{R})$.
\end{proof}

\section{Proof of Corollary~\ref{corollary:convex_hull2}}
\label{appendix:proof_corollary_convex_hull}

\corollaryconvexhull*
\begin{proof}
Let~$\tilde{\Pi}$ be the polytope in the RHS of the statement. By Theorem~\ref{theorem:convex_hull}, for every route~$R \in \mc{R}$, we know that~$y \in \policy(R)$ satisfies~$y^\xi(R') \geq \kscen{R'} - 1$, for every~$R' \subseteq R$ and~$\xi \in [N]$. Hence,~$\cap_{R \in \mc{R}} (\policy(R) \cap [\mathbf{0}, b]^N)$ lies in~$\tilde{\Pi}$. By convexity of~$\tilde{\Pi}$, this gives~$\conv\left(\bigcap_{R \in \mc{R}} (\policy(R) \cap [\mathbf{0}, b]^N) \right) \subseteq \tilde{\Pi}$.

To prove the converse, 
let~$\bar{y} \in [\mathbf{0}, b]^N$ be an extreme point of~$\tilde{\Pi}$. 
Note that if~$\tilde{\Pi} = \emptyset$ (i.e. $\bar{y}$ does not exist), we are done. 
It suffices to show~$\bar{y} \in \cap_{R \in \mc{R}} \policy(R)$. Just as in the proof of Theorem~\ref{theorem:convex_hull}, we note that the matrix corresponding to the inequalities defining~$\tilde{\Pi}$ is TU. Therefore, by Theorem 4.5 of~\cite{ipbook}, we have~$\bar{y} \in \Z^{[N] \times V_+}_+$. Lemma~\ref{lemma:projection_recourse_action} then implies that, for every route~$R \in \mc{R}$,~$\bar{y} \in \policy(R)$.
\end{proof}

\section{Proof of Lemma~\ref{lemma:flow_formulation}}
\label{appendix:proof_flow_formulation}

Before proving this result, we note that, in the case where~$R = (v_1, \ldots v_\ell) = (v_1)$, we have that~$v_0 = v_{2} = 0$, and thus the flow conservation constraints in Lemma~\ref{lemma:symmetric_flow} reduce to~$2 f^\xi_{(0, v_1)} + d^\xi(v_1) = 2 f^\xi_{(v_1, 0)} + g^\xi_{v_1}$, meaning that the flow variables on~$A(\oa{R})$ are counted twice. This is why routes with a single customer are treated differently in the following proof.

\lemmaflowformulation*
\begin{proof}
Fix~$\bar{x} \in \mc{X} \cap \Z^E$.

\noindent\underline{$\flow(\bar{x})\cap\Z^{[N]\times V_+}\subseteq \policy(\bar{x})\cap [\mathbf{0},b]^N$}: Clearly~$\flow(\bar{x})\cap\Z^{[N]\times V_+} \subseteq [\mathbf{0},b]^N$, so we just need to show~$\flow(\bar{x})\cap\Z^{[N]\times V_+} \subseteq \policy(\bar{x})$. We may assume~$\flow(\bar{x})\neq \emptyset$. Note that given a feasible solution to $\flow(\bar{x})$ we can round up its components and remain feasible. Since we assume~$\flow(\bar{x})\neq \emptyset$, we have~$\flow(\bar{x})\cap \Z^{[N] \times V_+}\neq \emptyset$. 
So let $\bar{y}\in \flow(\bar{x}) \cap \Z^{[N] \times V_+}$ and set~$\{\bar{y}_R\}_{R\in R(\bar{x})}$ as in Fact~\ref{fact:policy_split}. To show~$\bar{y} \in \policy(\bar{x})$, it suffices to prove that~$\bar{y}_R \in \policy(R)$, for each~$R\in R(\bar{x})$. 

Fix a route~$R \in \mc{R}(\bar{x})$ and a scenario~$\xi \in [N]$. In what follows, we construct~$f^\xi$ and~$g^\xi$ such that~$(\bar{y}^\xi,\bar{f}^\xi,\bar{g}^\xi)$ satisfies the constraints in Formulation~\eqref{formulation:flow2_feasibility} (with respect to route~$R$ and~$\xi$). Any entry of~$\bar{f}^\xi$ and~$\bar{g}^\xi$ not explicitly specified is assumed to be zero.

Suppose first that~$R$ has a single customer~$v$. Set~$\bar{f}^\xi_{0v} = 0$,~$\bar{f}^\xi_{v0}= d^\xi(v) / 2$, and~$g^\xi_v = 0$, for all~$v \in V_+$. By Assumption~\ref{assumption:large_demands},~$d^\xi(v) \leq C$, which implies that~$(\bar{y}^\xi,\bar{f}^\xi,\bar{g}^\xi)$ is feasible for Formulation~\eqref{formulation:flow2_feasibility}. Assume then that~$|V_+(R)| \geq 2$. Let $(\hat{f},\hat{g})$ be such that $(\bar{y},\hat{f},\hat{g})$ is feasible for $\flow(\bar{x})$. Construct~$\bar{f}^\xi$ and~$\bar{g}^\xi$ as follows:
\begin{align*}
    & \bar{f}^\xi_a = {\hat{f}}^\xi_a, & \forall a \in A(\oa{R}) \cup A(\cev{R}), \\
    & \bar{g}^\xi_v = \hat{g}^\xi_v, & \forall v \in V_+(R).
\end{align*}

\noindent
Feasibility of~$(\bar{y},\hat{f},\hat{g})$ to the formulation of~$\flow(\bar{x})$ implies feasibility of~$(\bar{y}^\xi,\bar{f}^\xi,\bar{g}^\xi)$ to Formulation~\eqref{formulation:flow2_feasibility}.

\noindent\underline{$\flow(\bar{x})\cap\Z^{[N]\times V_+}\supseteq \policy(\bar{x})\cap [\mathbf{0},b]^N$}:

A similar reasoning can be applied for the reverse direction. Let~$\hat{y} \in \policy(\bar{x}) \cap [\mathbf{0}, b]^N$ (if no such $\hat{y}$ exists, the result is trivial). We construct~$(\hat{f},\hat{g})$ so that~$(\hat{y},\hat{f},\hat{g})$ is feasible for the formulation of~$\flow(\bar{x})$. First, apply Fact~\ref{fact:policy_split} to decompose~$\hat{y}$ into vectors~$\{\hat{y}_R\}_{R\in \mc{R}(\bar{x})}$. Since~$R(\bar{x})$ is a routing plan, the graphs defined by each route~$R \in R(\bar{x})$ intersect only at the depot, and therefore, one can select the entries of~$(\hat{f},\hat{g})$ for each graph induced by each route in~$R(\bar{x})$ independently. Again, any entry of~$\hat{f}$ and~$\hat{g}$ not explicitly specified is assumed to be zero.

For each route~$R \in \mc{R}(\bar{x})$, we set entries of~$\hat{f}$ and~$\hat{g}$ as follows. If~$V_+(R) = \{v\}$, for every~$\xi \in [N]$, we set~$\hat{g}^\xi_v = 0$,~$\hat{f}^\xi_{(v,0)} = d^\xi(v)$ and~$\hat{f}^\xi_{(0,v)} = 0$. Note that~$\bar{x}_{0v} = 2$, so~$\hat{f}^\xi_{(v, 0)} \leq (C / 2) \cdot \bar{x}_{0v}$ by Assumption~\ref{assumption:large_demands}. If~$R$ has at least two customers, we apply Lemma~\ref{lemma:symmetric_flow} and let~$(\bar{f}, \bar{g})$ be such that $(\hat{y},\bar{f},\bar{g})$ satisfies the constraints of  Formulation~\eqref{formulation:flow2_feasibility} for $R$ and $\xi\in [N]$. Set~$\hat{f}^\xi_a = \bar{f}^\xi_a$, for every~$a \in A(\oa{R}) \cup A(\cev{R})$, and~$\hat{g}^\xi_v = \bar{g}^\xi_v$, for every~$v \in V_+(R)$. In either case, $(\hat{y}, \hat{f}, \hat{g})$ satisfies the constraints of $\flow(\bar{x})$ for $R$.
\end{proof}

\section{Proof of Proposition~\ref{proposition:benders}}
\label{appendix:proof_benders}

\propositionbenders*
\begin{proof}
Let~$\bar{x} \in \mc{X}$ and~$\bar{y} \in [\mathbf{0}, b]^N$. Fix a scenario~$\xi \in [N]$ and create a digraph~$D' = (V, A')$ obtained from~$\D$ by adding, for each~$v \in V_+$, a duplicate of arc~$(v, 0)$. Next, construct a vector of arc capacities~$h \in \R^{A'}$ by setting, for each~$(u, v) \in A$,~$h_{(u, v)} = (C / 2) \, \bar{x}_{\{u, v\}}$; and, for each duplicate arc~$(v, 0) \in A' \setminus A$,~$h_{(v, 0)} = C \bar{y}^\xi_v$. Additionally, set~$d'_0 = d^\xi(V_+)$ and~$d'_v = -d^\xi(v)$, for all~$v \in V_+$. 

From this construction, apply~Theorem~\ref{theorem:hoffman} to learn that a tuple~$(\bar{y}^\xi, f^\xi, g^\xi)$ satisfying the constraints in~$\flow(\bar{x})$ with respect to scenario~$\xi \in [N]$ exists if and only if
\begin{align*}
    & h(\delta^+_{D'}(\set)) \geq d'(V \setminus \set), && \forall \emptyset \subsetneq \set \subseteq V, & \\
    \iff & h(\delta^+_{D'}(\set)) \geq d'(V \setminus \set), && \forall \emptyset \subsetneq \set \subseteq V_+, &~~~~\text{(as~$d'_v \leq 0$, for all~$v \in V_+$)} \\
    \iff & C \, \left(\bar{y}^\xi(\set) + \frac{\bar{x}(\delta_G(\set))}{2} \right) \geq d^\xi(S) && \forall \emptyset \subsetneq \set \subseteq V_+.
\end{align*}
Since~$\bar{x} \in \mc{X}$, we close the proof by applying~$\bar{x}(\delta_G(\set)) = 2 |\set| - 2 \bar{x}(E_G(\set))$ to the last inequality.
\end{proof}

\section{\mytitle{Heuristics for separating SRIs}}
\label{appendix:separation}

Before discussing our separation routines, we remark that the separation of SRIs is strongly NP-hard in general. Indeed, suppose that~$(\bar{x}, \bar{y}) \in \mc{X} \times [\mathbf{0}, b]^N$ is a candidate solution with~$\bar{y}^\xi = \mathbf{0}$, for some~$\xi \in [N]$. Since for every~$\set \subseteq V_+$, inequality~$\bar{y}^\xi(\set) \geq \kscen{\set} + \bar{x}(E(\set)) - |\set|$ is equivalent to~$\bar{x}(E(\set)) \leq |\set| - \lceil d^\xi(\set) / C \rceil$,
it follows that separating a violated SRI reduces to finding a violated \emph{rounded-capacity inequality} (RCI). As the separation of RCIs is strongly NP-hard~\citep{DIARRASSOUBA201786}, we do not expect to separate the SRIs exactly in pseudo-polynomial time.

We thus propose two separation heuristics: one based on the CVRPSEP package of~\cite{Lysgaard2004}, and another one based on a MIP formulation. In our implementation, we first call the CVRPSEP-based heuristic. If this heuristic fails to separate (aggregated) SRIs, we then call the MIP-based heuristic.

By ordering the scenarios by their total demands, we assume henceforth that~$d^1(V_+) \geq \ldots \geq d^N(V_+)$.

\subsection{\mytitle{Heuristic based on the CVRPSEP package}}
\label{subsection:cvrpsep}

This separation heuristic simply calls CVRPSEP, for each scenario~$\xi = 1, \ldots, N$, and checks whether the returned candidate sets correspond to violated SRIs. The overall procedure is summarized in Algorithm~\ref{algorithm:cvrpsep_heuristic}. We make a few brief observations on this algorithm. First, we limit CVRPSEP to return at most 10 sets, so~$|\mc{\set}| \leq 10$ in line 4. Second, line 6 effectively constructs the aggregated SRIs discussed in Fact~\ref{fact:aggregated_sri}. Finally, we stop the algorithm earlier in line 10 to avoid calling CVRPSEP for every scenario.

\begin{algorithm}[htb]
  \hspace*{\algorithmicindent}
  ~\textbf{Input:}~$(\bar{x}, \bar{y}) \in \mc{X} \times [\mathbf{0}, b]^N$.\\
  \hspace*{\algorithmicindent}
  \textbf{Output:} A set of pairs~$(\set, \Xi)$, with~$\emptyset \subsetneq \set \subseteq V_+$ and~$\Xi \subseteq [N]$, such that~$\sum_{\xi \in \Xi} \bar{y}^\xi(\set) < \sum_{\xi \in \Xi} (\kscen{\set} + \bar{x}(E(\set)) - |\set|)$. Each of these violated inequalities will be later added to a relaxation of Formulation~\eqref{formulation:sri}.
\begin{algorithmic}[1]
    \Procedure {\textsc{SRI-CVRPSEP}}{$\bar{x}, \bar{y}$}
    \State {$\mc{C} \gets \emptyset$}
    \For{$\xi = 1, \ldots, N$}
        \State {Call CVRPSEP with demand vector~$d^\xi$ to get a family of sets~$\mc{\set} \subseteq 2^{V_+}$ such that~$\bar{x}(E(\set)) > |\set| - \kscen{\set}$, for every~$\set \in \mc{\set}$.}
        \For {$\set \in \mc{\set}$}
            \State {$\Xi \gets \{\xi' \in [N] : \bar{y}^{\xi'}(\set) < k_{\xi'}(\set) + \bar{x}(E(\set)) - |\set|\}$}
            \If {$\Xi \neq \emptyset$}
                \State {$\mc{C} \gets \mc{C} \cup \{(\set, \Xi)\}$}
            \EndIf
        \EndFor

        \If {$\mc{C} \neq \emptyset$}
            \State {\textbf{break}}
        \EndIf
    \EndFor
    \State {\textbf{return}~$\mc{C}$}
\EndProcedure
\end{algorithmic}
\caption{\textsc{SRI-CVRPSEP}}
\label{algorithm:cvrpsep_heuristic}
\end{algorithm}

\subsection{\mytitle{MIP formulation}}
\label{subsection:exact_separation}

To separate SRIs, we also adapt the MIP model of~\cite{pavlikov2024exact}, which was originally proposed for the exact separation of RCIs. To explain this approach, let~$(\bar{x}, \bar{y}) \in \mc{X} \times [\mathbf{0}, b]^N$ be a candidate solution. This solution violates an SRI if and only if there exists a scenario~$\xi \in [N]$ and a set~$\set \subseteq V_+$ such that
\begin{equation}
    \left\lceil \frac{d^\xi(\set)}{C} \right\rceil + \bar{x}(E(\set)) - |\set| - \bar{y}^\xi(\set) > 0.
    \label{ineq:violated_sri}
\end{equation}

By scaling, we assume without loss of generality that the vehicle capacity and the demands in scenario~$\xi$ are all integers, i.e.,~$C \in \Z_{++}$ and~$d^\xi \in \Z^{V_+}_+$. Fix a parameter~$\varepsilon \in (0, 1)$. The problem of maximizing the LHS of~\eqref{ineq:violated_sri} can be solved with the following formulation:
\begin{subequations}
\label{formulation:separation}
\begin{align} 
\max ~~& (\gamma + 1) - \sum_{v \in V_+} (\bar{y}^\xi_v + 1) q_v + \sum_{e \in E \setminus \delta(0)} \bar{x}_e h_e, & \nonumber \\
\text{s.t.~~} & \gamma C \leq \sum_{v \in V_+} d^\xi_v q_v - \varepsilon, & \label{separation:t1} \\
& h_{uv} \leq q_u, & \forall uv \in E \setminus \delta(0), \label{separation:z1} \\
& h_{uv} \leq q_v, & \forall uv \in E \setminus \delta(0), \label{separation:z2} \\
& h_{uv} \geq q_u + q_v - 1, & \forall uv \in E \setminus \delta(0), \label{separation:z3} \\
& \gamma \in \Z,~q \in \{0, 1\}^{V_+},~h \in \R^{E \setminus \delta(0)}_+.
\end{align}
\end{subequations}

\noindent
The binary variables~$q_v$, for each~$v \in V_+$, indicate whether customer~$v$ belongs to the set~$\set$. \mbox{Constraints~\eqref{separation:z1}-\eqref{separation:z3}} enforce that~$h_{uv} = 1$ if and only if~$q_u = q_v = 1$, meaning that~$h_{uv}$ indicates whether~$uv \in E$ belongs to~$E(\set)$. 

Finally, let~$(\bar{\gamma}, \bar{q}, \bar{h})$ be optimal for Formulation~\eqref{formulation:separation} and let~$\set = \{v \in V_+ : \bar{q}_v = 1\}$. We now use the same argument as in~\cite{pavlikov2024exact} to show that inequality~\eqref{separation:t1} guarantees that~$\bar{\gamma} = \lceil d^\xi(\set) / C \rceil - 1$. Let us write~$d^\xi(\set) = tC + r$, with~$t \in \Z_+$ and~$r \in [0, C)$, and note that~$d^\xi(\set) \in \Z_+$. If~$r$ is a positive integer, then~$\varepsilon < r$, which implies that~$tC < d^\xi(\set) - \varepsilon < (t + 1)C$. Since~$\bar{\gamma}$ is integral and we maximize~$\gamma$ in the objective function, it follows that~$\bar{\gamma} = t = \lceil d^\xi(\set) / C \rceil - 1$. If~$r = 0$, we have~$(t - 1) C < d^\xi(\set) - \varepsilon < tC$, so~$\bar{\gamma} = t - 1 = \lceil d^\xi(\set) / C \rceil - 1$.

Although Formulation~\eqref{formulation:separation} exactly separates SRIs, our implementation yields a heuristic for the following reasons. First, rather than maximizing the original objective function, we change the objective function to zero and enforce the constraint~$(\gamma + 1) - \sum_{v \in V_+} (\bar{y}_v + 1) q_v + \sum_{e \in E \setminus \delta(0)} \bar{x}_e h_e \geq 0.01$. Second, for each scenario~$\xi = 1, \ldots, N$, we solve Formulation~\eqref{formulation:separation} with a time limit of 30 seconds. 

Whenever we find a violated SRI of the form~$y^{\xi}(\set) < \kscen{\set} + \bar{x}(E(\set)) - |\set|$, we apply Fact~\ref{fact:aggregated_sri} and add the corresponding aggregated SRI. As in Algorithm~\ref{algorithm:cvrpsep_heuristic}, we stop the separation procedure whenever a violated (aggregated) SRI is found.

\section{Proof of Lemma~\ref{lemma:projection_single_inequality}}
\label{appendix:proof_projection_single_inequality}

\lemmaprojectionsingleinequality*
\begin{proof}
It is clear that~$\proj_{\theta}(\mc{H}) \subseteq \R^{V_+}_+$. Now take any~$\theta' \in \R^{V_+}_+$ and suppose that~$v \in V_+$ and~$\xi \in [N]$ are such that~$w_v = 0$ and~$a^\xi_v > 0$. Choose~$\bar{y} \in \R^{[N] \times V_+}_+$ with~$\bar{y}^\xi_v$ arbitrarily large and the other entries set to zero. Then~$\sum_{\xi' \in [N]} (a^{\xi'})^\T \bar{y}^{\xi'} \geq h$ and, for every~$u \in V_+$,~$\theta'_{u} \geq \sum_{\xi' \in [N]} p_{\xi'} w_u \bar{y}^{ \xi'}_u$, which shows that~$\theta' \in \proj_{\theta}(\mc{H})$. For the rest of the proof, we assume that~$a^\xi_v \leq 0$, for every~$\xi \in [N]$ and~$v \in V_+$ with~$w_v = 0$.

Fix~$\bar{\theta}$ to be an arbitrary vector in~$\R^{V_+}_+$. Then~$\bar{\theta}$ belongs to~$\proj_{\theta}(\mc{H})$ if and only if the LP below is feasible and has optimal value zero:
\begin{subequations}
\begin{align*} 
\min ~~& ~0 && \nonumber \\
\text{s.t.~~} & \sum_{\xi \in [N]} p_\xi w_v y^\xi_v \leq \bar{\theta}_v, & \forall v \in V_+, && (\gamma_v) \\
& \sum_{\xi \in [N]} (a^\xi)^\T y^\xi \geq h, &&& (\eta) \\
& y^\xi_v \geq 0, & \forall v \in V_+, \xi \in [N].
\end{align*}
\end{subequations}
Taking the dual:
\begin{subequations}
\label{formulation:dual_proof_single_inequality}
\begin{align} 
\max ~~& \sum_{v \in V_+} \gamma_v \bar{\theta}_v + \eta h, & \\
\text{s.t.~~} & p_\xi w_v \gamma_v + \eta \, a^\xi_v \leq 0, & \forall v \in V_+, \xi \in [N], \label{ineq:form:dual_proof_single_inequality}\\
& \eta \geq 0 \\
& \gamma_v \leq 0, & \forall v \in V_+.
\end{align}
\end{subequations}

\noindent
Clearly~$(\gamma, \eta) = (\mathbf{0}, 0)$ is feasible for the above formulation, so~$\bar{\theta}$ belongs to~$\proj_{\theta}(\mc{H})$ if and only if every feasible solution~$(\bar{\gamma}, \bar{\eta})$ for Formulation~\eqref{formulation:dual_proof_single_inequality} satisfies~$\sum_{v \in V_+} \bar{\gamma}_v \bar{\theta}_v + \bar{\eta} h \leq 0$. 

Without loss of generality, we may assume that such a feasible solution takes the following specific form. If~$w_v = 0$, we know that~$a^\xi_v \leq 0$ for every~$\xi \in [N]$, so the corresponding inequality~\eqref{ineq:form:dual_proof_single_inequality} is redundant. Since~$\bar{\theta}_v \geq 0$, we may then assume that~$\bar{\gamma}_v = 0$. If~$w_v > 0$, we instead assume that
$$\bar{\gamma}_v = \min_{\xi \in [N]} \left\{- \bar{\eta} \, \left(\frac{a^\xi_v}{ p_\xi w_v} \right), 0 \right\} = - \bar{\eta} \cdot \left(\max_{\xi \in [N]} \left\{\frac{a^\xi_v}{p_\xi w_v} \right\} \right)^+.$$ 

\noindent
Therefore,~$\bar{\theta} \in \R^{V_+}_+$ belongs to~$\proj_{\theta}(\mc{H})$ if and only if~$$\max_{\eta \geq 0} \left\{ \eta \left( - \sum_{v \in V_+ : w_v > 0} \left(\max_{\xi \in [N]} \left\{ \frac{a^\xi_v}{p_\xi w_v} \right\} \right)^+ \bar{\theta}_v + h \right) \right\} \leq 0,$$ 
which is equivalent to~$\left( - \sum_{v \in V_+ : w_v > 0} \left(\max_{\xi \in [N]} \left\{ \frac{a^\xi_v}{p_\xi w_v} \right\} \right)^+ \bar{\theta}_v + h \right) \leq 0$.
\end{proof}

\section{Proof of Proposition~\ref{proposition:proj_gamma_characterization}}
\label{appendix:proof_proj_gamma_characterization}

Before presenting the proof, we note that Proposition~\ref{proposition:proj_gamma_characterization} is shown via Lagrangian duality, rather than by repeating the ``Benders-type'' argument in the proof of Lemma~\ref{lemma:projection_single_inequality}. The latter approach has the advantage of showing that $\proj_{\theta}(\srihat(\bar{x}))$ can be represented with a finite number of projected SRIs. We used the Lagrangian approach to avoid repeating essentially the same argument as in Lemma~\ref{lemma:projection_single_inequality}, and we remind the reader that isolating the argument in Lemma~\ref{lemma:projection_single_inequality} is crucial for the results in Section~\ref{subsection:comparing_cuts}.

\propositionprojgammacharacterization*
\begin{proof}
Fix~$\bar{x} \in \mc{X}$ and let~$\bar{\theta} \in \proj_{\theta}(\srihat(\bar{x}))$. Then the optimal value of the LP below is zero:
\begin{subequations}
\begin{align*} 
\min ~~& 0 && \nonumber \\
\text{s.t.~~} & \sum_{\xi \in [N]} p_\xi w_v y^\xi_v \leq \bar{\theta}_v, & \forall v \in V_+ & \\
& y^\xi(\set) \geq \kscen{\set} + \bar{x}(E(\set)) - |\set|, & \forall \set \subseteq V_+, \xi \in [N], &&& (\alpha^\xi_\set) \\
& y^\xi_v \leq b_v, & \forall v \in V_+, \xi \in [N], &&& (\beta^\xi_v) \\
& y^\xi_v \geq 0, & \forall v \in V_+, \xi \in [N].
\end{align*}
\end{subequations}

\noindent
By Lagrangian duality, this holds if and only if, for every~$\alpha \geq 0$ and~$\beta \leq 0$,
\begin{subequations}
\label{formulation:proof-proj-single-lag-dual}
\begin{align} 
0~~\geq~~\min ~~& 
\begin{aligned}[t]
\nufunction{\alpha, \beta}
 - \sum_{\xi \in [N]} \sum_{v \in V_+}
   \left(\beta^\xi_v + \sum_{\set \subseteq V_+ : v \in \set}
   \alpha^\xi_\set \right) y^\xi_v + \sum_{\xi \in [N]} \sum_{\set \subseteq V_+}
   \alpha^\xi_\set \, \bar{x}(E(\set))
\end{aligned}
\\
\text{s.t.~~} & \sum_{\xi \in [N]} p_\xi w_v y^\xi_v \leq \bar{\theta}_v, \quad \quad \forall v \in V_+ \\
& y^\xi_v \geq 0, \quad\quad\quad\quad\quad\quad~~ \forall v \in V_+, \xi \in [N].
\end{align}
\end{subequations}
In particular, for any~$(\alpha, \beta) \in \dualset$, there exists~$y$ feasible for Formulation~\eqref{formulation:proof-proj-single-lag-dual} and such that
$$\sum_{\xi \in [N]} \sum_{v \in V_+}
\left(\beta^\xi_v + \sum_{S \subseteq V_+ : v \in S} \alpha^\xi_S \right) y^\xi_v - \sum_{\xi \in [N]} \sum_{S \subseteq V_+} \bar{\alpha}^\xi_S \, \bar{x}(E(S)) \geq \nufunction{\bar{\alpha}, \bar{\beta}}.$$
In other words,~$\bar{\theta} \in \proj_{\theta}(\Gammaset(\bar{x}, \alpha, \beta))$. Choose~$a^1, \ldots, a^N$ and~$h$ in Lemma~\ref{lemma:projection_single_inequality} so that~$\sum_{\xi \in [N]} (a^\xi)^\T y^\xi \geq h$ is equivalent to the above inequality. By the same lemma,~$\bar{\theta}$ satisfies the projected SRI for~$(\alpha, \beta)$, showing one side of the inclusion.

For the other side, suppose that~$\theta' \in \R^{V_+}_+$ satisfies the projected SRIs associated with every~$(\alpha, \beta) \in \dualset$. By Lemma~\ref{lemma:projection_single_inequality}, this implies that~$\theta' \in \proj_{\theta}(\Gammaset(\bar{x}, \alpha, \beta))$, for all~$\alpha \geq 0$ and~$\beta \leq 0$ (even if~$(\alpha, \beta) \notin \dualset$). Therefore, for every~$\alpha \geq 0$ and~$\beta \leq 0$, the optimal value of Formulation~\eqref{formulation:proof-proj-single-lag-dual} is at most zero, which implies that~$\theta'$ belongs to~$\proj_{\theta}(\srihat(\bar{x}))$.
\end{proof}

\section{Proofs of Lemma~\ref{lemma:robust_dual} and Theorem~\ref{theorem:dominance_set_cut}}
\label{appendix:proof_theorem_dominance_set_cut}

\lemmarobustdual*
\begin{proof}
    The condition~$\kscen{\set} - k' \leq b(\set)$ guarantees that Formulation~\eqref{formulation:recourse_lower_bound_set} is feasible, so the following dual formulation is bounded and has optimal objective value of~$\mc{L}^*_\xi(\set, k')$.
    \begin{subequations}
    \label{formulation:dual_recourse_set}
    \begin{align}
    \max ~~& \alpha^\xi_\set \, (\kscen{\set} - k') + \sum_{v \in V_+} \beta^\xi_v b_v \\ 
    \text{s.t.:~~} & \alpha^\xi_\set + \beta^\xi_v \leq w_v, & \forall v \in \set,\\
    & \beta^\xi_v \leq w_v, & \forall v \in V_+ \setminus \set,\\
    & \alpha^\xi_\set \geq 0, \\
    & \beta^\xi_v \leq 0, & \forall v \in V_+.
    \end{align}
    \end{subequations}
    \noindent
    If~$\kscen{\set} \leq k'$, then~$j = 1$, ~$(\bar{\alpha}^\xi_\set, \bar{\beta}^\xi) = (0, \mathbf{0})$ and~$\mc{L}^*_\xi(\set, k') = 0$, so we can safely assume that~$\kscen{\set} > k'$. 
    
    Let us verify that~$(\bar{\alpha}^\xi_\set, \bar{\beta}^\xi)$ is feasible for Formulation~\eqref{formulation:dual_recourse_set}. Take a customer~$v \in V_+$. If~$v \notin S$, then~$\bar{\beta}^\xi_v = 0 \leq w_v$. If~$v \in \{v_i\}_{i \in [j - 1]}$, then~$\bar{\alpha}^\xi_\set + \bar{\beta}^\xi_v = w_v$. Lastly, if~$v = v_i$, for some~$i \geq j$, then~$\bar{\alpha}^\xi_\set + \bar{\beta}^\xi_v = w_{v_j} \leq w_{v_i}$. 
    
    Now, the solution~$(\bar{\alpha}^\xi_\set, \bar{\beta}^\xi)$ has the same objective value as the greedy algorithm:
    \begin{equation}
        \label{eq:proof_set_cut_dual}
        \bar{\alpha}^\xi_\set (\kscen{\set} - k') + \sum_{v \in V_+} \bar{\beta}^\xi_v b_v = \sum_{i \in [j - 1]} b_{v_i} w_{v_i} + \left(\kscen{\set} - k' - \sum_{i \in [j - 1]} b_{v_i}\right) w_{v_j} = \mc{L}^*_\xi(\set, k'),
    \end{equation}
    which implies that~$(\bar{\alpha}^\xi_\set, \bar{\beta}^\xi)$ is optimal for Formulation~\eqref{formulation:dual_recourse_set}. 
    
    To show that~$w_{v_j} = \bar{\alpha}^\xi_\set \leq \mc{L}^*_\xi(\set, k')$, suppose first that~$j = 1$. By equation~\eqref{eq:proof_set_cut_dual},~$\mc{L}^*_\xi(\set, k') = (\kscen{\set} - k') \, \bar{\alpha}^\xi_\set$ and we are done. On the other hand, if~$j \geq 2$, the choice of~$j$ implies~$\kscen{\set} - k' > \sum_{i \in [j - 1]} b_{v_i}$, and since the terms are integers,~$\kscen{\set} - k' - \sum_{i \in [j - 1]} b_{v_i} \geq 1$. Equation~\eqref{eq:proof_set_cut_dual} then gives~$\mc{L}^*_\xi(\set, k') \geq \sum_{i \in [j - 1]} b_{v_i} w_{v_i} + w_{v_j} \geq \bar{\alpha}^\xi_\set$.
\end{proof}

\theoremdominancesetcut*
\begin{proof}
    Fix~$\bar{x} \in \{x \in \mc{X} : x(E(\set)) \leq |\set| - k'\}$ and recall that~$\nufunction{\hat{\alpha}, \hat{\beta}} = \sum_{\xi \in [N]} \sum_{\set \subseteq V_+} \hat{\alpha}^\xi_\set ( \kscen{\set} - |\set| ) + \sum_{\xi \in [N]} \sum_{v \in V_+} \hat{\beta}^\xi_v b_v$. We rewrite the inequality~\eqref{ineq:conic_comb} defining~$\Gammaset(\bar{x}, \hat{\alpha}, \hat{\beta})$ as follows:
    \begin{align*}
        & \sum_{\xi \in [N]} \sum_{v \in V_+} \left(\hat{\beta}^\xi_v + \sum_{\set' \subseteq V_+ : v \in \set'} \hat{\alpha}^\xi_{\set'} \right) y^\xi_v \geq \sum_{\xi \in [N]} \sum_{\set' \subseteq V_+} \hat{\alpha}^\xi_{\set'} \, \bar{x}(E(\set)) + \nufunction{\hat{\alpha}, \hat{\beta}} \nonumber \\
        \iff & \sum_{\xi \in [N]} p_\xi \sum_{v \in \set} \left(\bar{\beta}^\xi_v + \bar{\alpha}^\xi_\set \right) y^\xi_v \geq \sum_{\xi \in [N]} p_\xi \left( \bar{\alpha}^\xi_\set \, (\kscen{\set} + \bar{x}(E(\set)) - |\set|) + \sum_{v \in V_+} \bar{\beta}^\xi_v b_v\right).
    \end{align*}
    where we used that $\hat{\beta}^\xi_v = 0$ if $v\notin S$, and $\hat{\alpha}^\xi_{\set'}=0$ if $S'\neq S$.

    Fix a scenario~$\xi \in [N]$. It follows from Lemma~\ref{lemma:robust_dual} and the choice of~$(\bar{\alpha}^\xi_\set, \bar{\beta}^\xi)$ that~$\bar{\beta}^\xi_v + \bar{\alpha}^\xi_\set \leq w_v$, for all~$v \in \set$. In particular, whenever~$w_v = 0$, we have that~$\hat{\beta}^\xi_v + \hat{\alpha}^\xi_\set \leq 0$, so~$(\hat{\alpha}, \hat{\beta}) \in \dualset$. In addition, since~$\bar{\alpha}^\xi_\set \leq \mc{L}^*_\xi(\set, k')$ and~$\bar{x}(E(\set)) - |\set| + k' \leq 0$,
    \begin{equation*}
        \bar{\alpha}^\xi_\set \, (\kscen{\set} + \bar{x}(E(\set)) - |\set|) + \sum_{v \in V_+} \bar{\beta}^\xi_v b_v = \underbrace{\bar{\alpha}^\xi_\set \, (\kscen{\set} - k') + \sum_{v \in V_+} \bar{\beta}^\xi_v \, b_v}_{= \mc{L}^*_\xi(\set, k')} + \underbrace{\bar{\alpha}^\xi_\set ( \bar{x}(E(\set)) - |\set| + k')}_{\geq \mc{L}^*_\xi(\set, k') (\bar{x}(E(\set)) - |\set| + k')}.
    \end{equation*}

    \noindent
    Summing over all scenarios~$\xi \in [N]$, we learn that inequality~\eqref{ineq:set_cut} is valid for~$\proj_{\theta} (\Gammaset(\bar{x}, \hat{\alpha}, \hat{\beta}))$:
    $$\theta(\set) \geq \sum_{\xi \in [N]} \sum_{v \in \set} p_\xi w_v y^\xi_v \geq \sum_{\xi \in [N]} p_\xi \sum_{v \in \set} \left(\bar{\beta}^\xi_v + \bar{\alpha}^\xi_\set \right) y^\xi_v \geq \mc{L}^*(\set, k') \, (1 + \bar{x}(E(\set)) - |\set| + k').$$ 
    
    To show the second part of the statement, we recall that, by Lemma~\ref{lemma:projection_single_inequality}, the projected SRI~\eqref{ineq:projected_sri} associated with~$(\hat{\alpha}, \hat{\beta})$ describes~$\proj_{\theta} (\Gammaset(\bar{x}, \hat{\alpha}, \hat{\beta}))$:
    $$\proj_{\theta} (\Gammaset(\bar{x}, \hat{\alpha}, \hat{\beta})) = \left\{ \theta \in \R^{V_+}_+ : \sum_{v \in V_+ : w_v > 0} \phiv(\alpha, \beta) \, \theta_v \geq \sum_{\xi \in [N]} \sum_{S \subseteq V_+} \hat{\alpha}^\xi_S \, \bar{x}(E(S)) + \nufunction{\hat{\alpha}, \hat{\beta}} \right\}.$$
    Thus, any~$(\bar{x}, \bar{\theta}) \in \mc{X} \cap \R^{V_+}_+$ satisfying the projected SRI~\eqref{ineq:projected_sri} associated with~$(\hat{\alpha}, \hat{\beta})$ also satisfies the set cut~\eqref{ineq:set_cut}.
\end{proof}

\section{\mytitle{Example of a projected SRI dominating a set cut}}
\label{appendix:example_set_cut}

Let~$\set = \{v_1, v_2, v_3\} \subseteq V_+$, and suppose that~$w_{v_1} = 2$, $w_{v_2} = 3$, $w_{v_3} = 4$, and~$b = \allones$. We set~$k' = 1$ and note that~$x(E(\set)) \leq |\set| - 1$ is valid for~$\mc{X}$. Suppose that~$\xi \in [N]$ is such that~$\kscen{\set} = 3$, while for every other~$\xi' \in [N] \setminus \{\xi\}$ we have~$k_{\xi'}(\set) \leq 1$. Since~$\mc{L}^*(\set, 1) = 5 p_\xi$, the set cut~\eqref{ineq:set_cut} becomes
\begin{equation}
    \label{ineq:example_set_cut}
    \theta_{v_1} + \theta_{v_2} + \theta_{v_3} \geq 5 p_\xi + 5 p_\xi \cdot (x(E(\set)) - |\set| + 1).
\end{equation}

Applying Lemma~\ref{lemma:robust_dual}, we obtain a dual solution~$(\bar{\alpha}^\xi_\set, \bar{\beta}^\xi)$ with~$\bar{\alpha}^\xi_\set = 3$,~$\bar{\beta}^\xi_{v_1} = -1$, and~$\bar{\beta}^\xi_v = 0$ for all~$v \in V_+ \setminus \{v_1\}$. Note that~$\mc{L}^*_\xi(S, 1) = \bar{\alpha}^\xi_\set (\kscen{\set} - k') + \sum_{v \in V_+} \bar{\beta}^\xi_v b_v = 3 \cdot 2 - 1 = 5$.

Following the construction in the proof of Theorem~\ref{theorem:dominance_set_cut}, we sum the following two inequalities:
\begin{align*}
    (3 p_\xi) \cdot y^\xi(\set) & \geq (3 p_\xi) \cdot \left(\kscen{\set} + x(E(\set)) - |\set| \right) = (3 p_\xi) \cdot \left(2 + x(E(\set)) - |\set| + 1\right),\\
    (-1 p_\xi) \cdot y^\xi_{v_1} & \geq (-1 p_\xi) \cdot 1,
\end{align*}
which yields
$$(2 p_\xi) y^\xi_{v_1} + (3 p_\xi) y^\xi_{v_2} + (3 p_\xi) y^\xi_{v_3} \geq 5 p_\xi + 3 p_\xi \cdot (x(E(\set)) - |\set| + 1).$$
By Lemma~\ref{lemma:projection_single_inequality}, the projection of the above inequality onto the~$(x, \theta)$-space is given by the projected SRI
$$\frac{2}{2} \theta_{v_1} + \frac{3}{3} \theta_{v_2} + \frac{3}{4} \theta_{v_3} \geq 5 p_\xi + 3 p_\xi \cdot (x(E(\set)) - |\set| + 1),$$
which indeed dominates the set cut~\eqref{ineq:example_set_cut}.~\hfill\Halmos

\section{Proof of Theorem~\ref{theorem:dominance_sri}}
\label{appendix:proof_comparison}

We first need a few technical lemmas.
\begin{lemma}
    \label{lemma:existence_dual}
    Let~$w \in \Q^{V_+}_+$ and~$\xi \in [N]$. Let~$H$ be a partial route and suppose that Formulation~\eqref{formulation:recourse_lower_bound} is feasible (with respect to~$H$ and~$\xi$). Then there exists an optimal solution~$(\bar{\alpha}^\xi, \bar{\beta}^\xi)$ for the dual of Formulation~\eqref{formulation:recourse_lower_bound} such that
    \begin{enumerate}[(i)]
        \item $\bar{\beta}^\xi_v = 0$, for every~$v \in V_+ \setminus V_+(H)$;~ and \label{item:lemma_existence_dual1}
        \item $\sum_{H' \subseteq H} \bar{\alpha}^\xi_{H'} \leq \mc{L}^*_\xi(H)$. \label{item:lemma_existence_dual2}
    \end{enumerate}
\end{lemma}
\begin{proof}
    Consider the dual of Formulation~\eqref{formulation:recourse_lower_bound}:
    \begin{subequations}
    \label{formulation:dual_recourse}
    \begin{align}
    \max & \sum_{H' \subseteq H} \alpha^\xi_{H'} \, (\kscen{H'} - 1) + \sum_{v \in V_+} \beta^\xi_v b_v \\ 
    \text{s.t.:} & \sum_{H' \subseteq H : v \in V_+(H')} \alpha^\xi_{H'} + \beta^\xi_v \leq w_v, & \forall v \in V_+, \label{ineq:formulation:dual_recourse} \\
    & \alpha^\xi \geq 0, \\
    & \beta^\xi \leq 0.
    \end{align}
    \end{subequations}
    
    Let~$(\bar{\alpha}^\xi, \bar{\beta}^\xi)$ be an optimal solution for Formulation~\eqref{formulation:dual_recourse} and let~$S = \{v \in V_+ : \bar{\beta}^\xi_v < 0\}$. Suppose that we choose~$(\bar{\alpha}^\xi, \bar{\beta}^\xi)$ so that the size of its support (i.e.,~$|\{H' \subseteq H : \bar{\alpha}^\xi_{H'} > 0\}| + |S|$) is minimum. Since~$b \geq 0$, it follows from the optimality of~$(\bar{\alpha}^\xi, \bar{\beta}^\xi)$ that we may assume~$\bar{\beta}^\xi_v = \min\{w_v - \sum_{H' \subseteq H : v \in V_+(H')} \bar{\alpha}^\xi_{H'}, 0\}$, for all~$v \in V_+$. This already gives item~\ref{item:lemma_existence_dual1}, so it remains to show item~\ref{item:lemma_existence_dual2}.

    Rewriting the objective value of~$(\bar{\alpha}^\xi, \bar{\beta}^\xi)$ yields
    \begin{align*}
        \mc{L}^*_{\xi}(H) & = \sum_{H' \subseteq H} \bar{\alpha}^\xi_{H'} \, ( \kscen{H'} - 1 ) + \sum_{v \in V_+} \bar{\beta}^\xi_v b_v \\
        & = \sum_{H' \subseteq H} \bar{\alpha}^\xi_{H'} \, ( \kscen{H'} - 1 ) + \sum_{v \in S} \left( w_v - \sum_{H' \subseteq H : v \in V_+(H')} \bar{\alpha}^\xi_{H'} \right) b_v \\
        & = \sum_{v \in S} w_v b_v + \sum_{H' \subseteq H} \bar{\alpha}^\xi_{H'} \, ( \kscen{H'} - b(V_+(H') \cap S) - 1) \\
        & \geq \sum_{H' \subseteq H} \bar{\alpha}^\xi_{H'},
    \end{align*}
    where the last inequality follows because minimality of the support implies that~$$\kscen{H'} - b(V_+(H') \cap S) \geq 2,$$ for every~$H' \subseteq H$ with~$\bar{\alpha}^\xi_{H'} > 0$.

    To see this last claim, suppose by contradiction that~$H' \subseteq H$ is such that~$\bar{\alpha}^\xi_{H'} > 0$ and~$\kscen{H'} - b(V_+(H') \cap S) \leq 1$. Consider the direction~$r = (-\mathbbm{1}_{H'}, \mathbbm{1}_{V_+(H') \cap S})$, where~$\mathbbm{1}_{H'}$ is the incident vector of~$H'$ and~$\mathbbm{1}_{V_+(H') \cap S}$ is the incident vector of~$V_+(H') \cap S$. Take~$\varepsilon = \min\left\{\bar{\alpha}^\xi_{H'}, \min_{v \in V_+(H') \cap S} \{-\bar{\beta}^\xi_v\}\right\} > 0$ and construct~$(\hat{\alpha}^\xi, \hat{\beta}^\xi) = (\bar{\alpha}^\xi, \bar{\beta}^\xi) + \varepsilon r$. In other words, we obtain~$(\hat{\alpha}^\xi, \hat{\beta}^\xi)$ by decreasing~$\bar{\alpha}^\xi_{H'}$ by~$\varepsilon$ and increasing~$\bar{\beta}^\xi_v$ by~$\varepsilon$ at every~$v \in V_+(H') \cap S$. 
    
    By the choice of~$\varepsilon$,~$(\hat{\alpha}^\xi, \hat{\beta}^\xi)$ is feasible for Formulation~\eqref{formulation:dual_recourse}, has smaller support than~$(\bar{\alpha}^\xi, \bar{\beta}^\xi)$, and has objective value~$$\mc{L}_\xi^*(H) + \varepsilon \, (-\kscen{H'} + 1 + b(V_+(H') \cap S)) \geq \mc{L}_\xi^*(H).$$
    Therefore,~$(\hat{\alpha}^\xi, \hat{\beta}^\xi)$ is optimal for Formulation~\eqref{formulation:dual_recourse}, contradicting the minimality of the support of~$(\bar{\alpha}^\xi, \bar{\beta}^\xi)$.
\end{proof}

Let~$H = (\set_1, \ldots, \set_\ell)$ be a partial route. To simplify the notation, we abbreviate~$x(E(H) \setminus \delta(0))$ as~$x(H)$. Recall the activation function from Section~\ref{subsection:ils}:
\begin{equation*}
    \Wof(x; \supsetXh{H}) = 1 + (x(H) - |V_+(H)| + 1) + \sum_{i \in \{2,  \ell - 1\} \cap [\ell]}(x(E(\set_i)) - |\set_i| + 1).
\end{equation*}
To prove Theorem~\ref{theorem:dominance_sri}, we use Lemma~\ref{lemma:sri_activation} below. In what follows, for any disjoint~$S, T \subseteq V_+$, we repeatedly use~$x(E(S \cup T)) = x(E(S)) + x(E(T)) + x(E(S, T))$ and that~$\bar{x}\in \mc{X}$ satisfies~$\bar{x}(E(S))\leq |S| - 1$, for every~$\emptyset \subsetneq S \subseteq V_+$.

\begin{lemma}
    \label{lemma:sri_activation}
    For every~$H' \subseteq H$ and~$\bar{x} \in \mc{X}$,~$\bar{x}(H') - |V_+(H')| + 1 \geq \Wof(\bar{x}; \supsetXh{H}) - 1$.
\end{lemma}
\begin{proof}
    Suppose by contradiction that~$H' = (\set_i, \ldots, \set_j) \subseteq H$ and~$\bar{x} \in \mc{X}$ is such that~$\bar{x}(H') - |V_+(H')| + 1 < \Wof(\bar{x}; \supsetXh{H}) - 1$. Furthermore, assume that~$H'$ is a maximal counterexample, meaning that, for every~$H'' \subseteq H$ with~$H' \subseteq H''$ and~$H'' \neq H'$, we have that~$\bar{x}(H'') - |V_+(H'')| + 1 \geq \Wof(\bar{x}; \supsetXh{H}) - 1$.

    \begin{quote}
    \begin{claim}
    \label{claim:proof_activation_function}
    $j \geq \ell - 1$ and~$i \leq 2$.
    \end{claim}
    \begin{proof}
    By symmetry, we only show that~$j \geq \ell - 1$. Suppose by contradiction that~$j \leq \ell - 2$ (so~$\ell \geq 3$). We first consider the case~$|\set_{j + 1}| > 1$. By the definition of partial routes, we know that~$|\set_j| = 1$, so~$\bar{x}(E(\set_j, \set_{j + 1})) + \bar{x}(E(\set_{j + 1})) = \bar{x}(E(\set_j \cup \set_{j + 1})) \leq |\set_j \cup \set_{j + 1}| - 1 = |\set_{j + 1}|$. Hence, for~$H'' = (S_i, \ldots, S_{j + 1})$,
    \begin{align*}
        \bar{x}(H'') - |V_+(H'')| + 1 & = (\bar{x}(H') + \bar{x}(E(\set_j, \set_{j + 1})) + \bar{x}(E(\set_{j + 1}))) - (|V_+(H')| + |\set_{j + 1}|) + 1 \\
        & = (\bar{x}(H') - |V_+(H')| + 1) + (\bar{x}(E(\set_j \cup \set_{j + 1})) - |\set_{j + 1}|) \\
        & \leq \bar{x}(H') - |V_+(H')| + 1,
    \end{align*}
    contradicting the maximality of~$H'$. If~$\set_{j + 1} = \{v\}$, we set~$H'' = (\set_i, \ldots, \set_{j + 2})$ and note that~$\bar{x}(E(v, \set_j \cup \set_{j + 2})) \leq 2$. This implies that
    \begin{align*}
        & \bar{x}(H'') - |V_+(H'')| + 1 \\
        & \quad = (\bar{x}(H') + \bar{x}(E(v, \set_j \cup \set_{j + 2})) + \bar{x}(E(\set_{j + 2})) - (|V_+(H')| + |\set_{j + 2}| + 1) + 1 \\
        & \quad = (\bar{x}(H') - |V_+(H')| + 1) + (\bar{x}(E(v, \set_j \cup \set_{j + 2})) - 2) + (\bar{x}(E(\set_{j + 2})) - |\set_{j + 2}| + 1) \\
        & \quad \leq \bar{x}(H') - |V_+(H')| + 1,
    \end{align*}
    again contradicting the maximality of~$H'$.
    \end{proof}
    \end{quote}
    
    It remains to check the cases where~$i \in \{1, 2\}$ and~$j \in \{\ell - 1, \ell\}$. Note that~$\bar{x}(H) - |V_+(H)| + 1 \geq \Wof(\bar{x}; \supsetXh{H}) - 1$, so we cannot have~$i = 1$ and~$j = \ell$, which implies that~$\ell \geq 2$. 
    
    Suppose first that~$j = \ell - 1$ and~$i = 1$. Using the formula for~$\Wof$, 
    \begin{align*}
        \Wof(\bar{x}; \supsetXh{H}) - 1 &
        = (\bar{x}(H) - |V_+(H)| + 1) + \sum_{i \in \{2,  \ell - 1\} \cap [\ell]}(x(E(\set_i)) - |\set_i| + 1) \\
        & \leq (\bar{x}(H) - |V_+(H)| + 1) + (\bar{x}(E(\set_{\ell - 1})) - |\set_{\ell - 1}| + 1) \\
        & = (\bar{x}(H') + \bar{x}(E(\set_{\ell - 1}, \set_\ell)) + \bar{x}(E(\set_\ell)) - |V_+(H')| - |\set_\ell| + 1) \\
        & \qquad + (\bar{x}(E(\set_{\ell - 1})) - |\set_{\ell - 1}| + 1) \\
        & = (\bar{x}(H') - |V_+(H')| + 1) + (\bar{x}(E(\set_{\ell - 1} \cup \set_\ell)) - |\set_{\ell - 1} \cup \set_\ell| + 1) \\
        & \leq \bar{x}(H') - |V_+(H')| + 1.
    \end{align*}
    By symmetry, this also handles the case~$i = 2$ and~$j = \ell$. 
    
    Suppose now that~$i = 2$ and~$j = \ell - 1$. Since~$i \leq j$, we have~$\ell \geq 3$. If~$\ell = 3$, we have that~$H' = (\set_2)$, so
    \begin{align*}
        \Wof(\bar{x}; \supsetXh{H}) - 1 & = (\bar{x}(H) - |V_+(H)| + 1) + (x(E(\set_2)) - |\set_2| + 1) \\
        &=  (\bar{x}(H) - |V_+(H)| + 1) + (\bar{x}(H') - |V_+(H')| + 1) \leq \bar{x}(H') - |V_+(H')| + 1. 
    \end{align*}
    On the other hand, if~$\ell \geq 4$,
    \begin{align*}
        \Wof(\bar{x}; \supsetXh{H}) - 1
        & = (\bar{x}(H) - |V_+(H)| + 1)
        + \sum_{i \in \{2, \ell - 1\}} (\bar{x}(E(\set_i)) - |\set_i| + 1) \\
        & = \bar{x}(H') - |V_+(H')| + 1 \\
        & \qquad + (\bar{x}(E(\set_1)) + \bar{x}(E(\set_1, \set_2)) - |\set_1|) + (\bar{x}(E(\set_\ell)) + \bar{x}(E(\set_{\ell - 1}, \set_\ell)) - |\set_\ell|) \\
        & \qquad + (\bar{x}(E(\set_2)) - |\set_2| + 1) + (\bar{x}(E(\set_{\ell - 1})) - |\set_{\ell - 1}| + 1) \\
        & = \bar{x}(H') - |V_+(H')| + 1 \\
        & \qquad (\bar{x}(E(\set_1 \cup \set_2)) - |\set_1 \cup \set_2| + 1) + (\bar{x}(E(\set_{\ell - 1} \cup \set_\ell)) - |\set_{\ell -1} \cup \set_\ell| + 1)\\
        & \leq \bar{x}(H') - |V_+(H')| + 1.
    \end{align*}
\end{proof}

\theoremdominancesri*
\begin{proof}
Take an arbitrary scenario~$\xi \in [N]$. We first show that inequality 
\begin{equation}
    \label{ineq:proof_dominance_sri}
    \sum_{v \in V_+(H)} w_v y^\xi_v \geq \mc{L}_\xi^*(H) \cdot \Wof(x; \supsetXh{H})
\end{equation}
is valid for the LP relaxation of Formulation~\eqref{formulation:sri}.

Let~$R$ be a route adhering to~$H$, since~$\policy(R) \cap [\mathbf{0}, b]^N \neq \emptyset$, we know that Formulation~\eqref{formulation:recourse_lower_bound} is feasible.
Let~$(\bar{\alpha}^\xi, \bar{\beta}^\xi)$ be an optimal dual solution chosen according to Lemma~\ref{lemma:existence_dual}. We construct an inequality from~$(\bar{\alpha}^\xi, \bar{\beta}^\xi)$ as follows. For each~$H' \subseteq H$, we multiply the SRI~$y^\xi(H') \geq \kscen{H'} + x(E(V_+(H'))) - |H'|$ by~$\bar{\alpha}^\xi_{H'} \geq 0$, and for each~$v \in V_+(H)$, we multiply~$y^\xi_v \leq b_v$ by~$\bar{\beta}^\xi_v \leq 0$. Summing these inequalities we get an inequality (valid for the LP relaxation of~\eqref{formulation:sri}) of the form~$A^\xi(y) \geq B^\xi(x)$, where
\begin{align*}
    A^\xi(y) & = \sum_{v \in V_+(H)} \left( \bar{\beta}^\xi_v + \sum_{H' \subseteq H : v \in V_+(H')} \bar{\alpha}^\xi_{H'} \right) y^\xi_v \qquad \text{ and} \\
    B^\xi(x) &= \sum_{H' \subseteq H} \bar{\alpha}^\xi_{H'} 
    \left[(\kscen{H'} - 1) + x(E(V_+(H'))) - |H'| + 1 \right] + \sum_{v \in V_+(H)} \bar{\beta}^\xi_v b_v.
\end{align*}

\noindent
Item~\ref{item:lemma_existence_dual1} of Lemma~\ref{lemma:existence_dual} implies that~$\sum_{v \in V_+(H)} \bar{\beta}^\xi_v b_v = \sum_{v \in V_+} \bar{\beta}^\xi_v b_v$, so by optimality of~$(\bar{\alpha}^\xi, \bar{\beta}^\xi)$ to Formulation~\eqref{formulation:dual_recourse} (which is the dual of Formulation~\eqref{formulation:recourse_lower_bound}),~$$B^\xi(x) = \mc{L}_\xi^*(H) + \sum_{H' \subseteq H} \bar{\alpha}_{H'} \cdot (\bar{x}(E(V_+(H'))) - |H'| + 1).$$ Consequently, for every~$\bar{x} \in \mc{X}$,
\begin{align*}
    B^\xi(\bar{x}) & = \mc{L}_\xi^*(H) + \sum_{H' \subseteq H} \bar{\alpha}_{H'} \cdot (\bar{x}(E(V_+(H'))) - |H'| + 1) & \\
    & \geq \mc{L}_\xi^*(H) + \left(\sum_{H' \subseteq H} \bar{\alpha}^\xi_{H'} \right) \cdot (\Wof(\bar{x} ; \supsetXh{H}) - 1) & \text{(by Lemma~\ref{lemma:sri_activation})}\\
    & \geq \mc{L}_\xi^*(H) \cdot \Wof(\bar{x} ; \supsetXh{H}). & \text{(by item~\ref{item:lemma_existence_dual2} of Lemma~\ref{lemma:existence_dual})}
\end{align*}

\noindent
Now, feasibility of~$(\bar{\alpha}^\xi, \bar{\beta}^\xi)$ to Formulation~\eqref{formulation:dual_recourse} implies~$A^\xi(y) \leq \sum_{v \in V_+(H)} w_v y^\xi_v$, so we conclude that~$A^\xi(y) \geq B^\xi(x)$ dominates inequality~\eqref{ineq:proof_dominance_sri}, as desired.

To obtain the desired projected SRI, construct~$(\hat{\alpha}, \hat{\beta})$ as follows. For each~$\xi \in [N]$ and~$v \in V_+$, set~$\hat{\beta}^\xi_v = p_\xi \, \bar{\beta}^\xi_v$; and for each~$\xi \in [N]$ and~$S \subseteq V_+$, set 
\begin{equation*}
    \hat{\alpha}^\xi_S =
    \begin{cases}
    p_\xi \, \bar{\alpha}^\xi_{H'}, & \text{if~$S = V_+(H')$ for some~$H' \subseteq H$,}\\
    0, & \text{otherwise.}
    \end{cases}
\end{equation*}
Observe that, for every scenario~$\xi \in [N]$,~$(\bar{\alpha}^\xi, \bar{\beta}^\xi)$ is feasible for Formulation~\eqref{formulation:dual_recourse}, so~$(\hat{\alpha}, \hat{\beta}) \in \dualset$.

Summing inequality~\eqref{ineq:proof_dominance_sri} over all scenarios, we learn that, for every~$\bar{x} \in \mc{X}$, inequality $$\theta(V_+(H)) \geq \sum_{\xi \in [N]} \sum_{v \in V_+(H)} p_\xi w_v y^\xi_v \geq \sum_{\xi \in [N]} p_\xi \, A^\xi(y) \geq \sum_{\xi \in [N]} p_\xi \, B^\xi(x) \geq \mc{L}^*(H) \cdot \Wof(\bar{x} ; \supsetXh{H})$$ is valid for~$\proj_{\theta}(\Gammaset(\bar{x}, \hat{\alpha}, \hat{\beta}))$. By Lemma~\ref{lemma:projection_single_inequality}, the projected SRI~\eqref{ineq:projected_sri} associated with~$(\hat{\alpha}, \hat{\beta})$ dominates the partial route cut~\eqref{ineq:partial_route_cut}.
\end{proof}

\section{Proof of Theorem~\ref{theorem:single_inequality}}
\label{appendix:proof_theorem_single_inequality}

\theoremsingleinequality*
\begin{proof}
    Fix~$\alpha \geq 0$ and~$\beta \leq 0$ and write
    \begin{align*}
    z' 
    &\coloneqq \min\{c^\T x + \allones^\T \theta : x \in \mc{X},~\theta \in \proj_{\theta}(\Gammaset(x, \alpha, \beta))\}\\
    & = \min\left\{c^\T x + \allones^\T \theta : x \in \mc{X},~(\theta, y) \in \Gammaset(x, \alpha, \beta)\right\} \\
    &= \min~c^\T x + \sum_{v \in V_+} \sum_{\xi \in [N]} p_\xi w_v y^\xi_v \\
    &\quad~\text{s.t.~~}
    \sum_{\xi \in [N]} \sum_{v \in V_+} \left(\beta^\xi_v + \sum_{S \subseteq V_+ : v \in S} \alpha^\xi_S \right) y^\xi_v - \sum_{\xi \in [N]} \sum_{S \subseteq V_+} \alpha^\xi_S x(E(S)) \geq \nufunction{\alpha, \beta}, \\
    &\qquad\quad~~ x \in \mc{X}, ~~y \geq 0.
    \end{align*}
    Dualizing the inequality with~$(\alpha, \beta)$ into the objective function gives the Lagrangian dual: 
    \begin{equation*}
    z' = \max_{\eta \geq 0} \left\{ \min_{x \in \mc{X}, y \geq 0} \left\{
    \begin{aligned} 
    & c^\T x + \sum_{v \in V_+} \sum_{\xi \in [N]} p_\xi w_v y^\xi_v - \eta \sum_{\xi \in [N]} \sum_{v \in V_+} \left(\beta^\xi_v + \sum_{S \subseteq V_+ : v \in S} \alpha^\xi_S \right) y^\xi_v \\
    & \quad + \eta \sum_{\xi \in [N]} \sum_{S \subseteq V_+} \alpha^\xi_S x(E(S)) + \eta \, \nufunction{\alpha, \beta}
    \end{aligned} \right\} \right\}.
    \end{equation*}
    Taking~$\eta = 1$ we conclude that~$z' \geq \sigmax{\alpha} + \sigmay{\alpha, \beta} + \nufunction{\alpha, \beta}$.
\end{proof}

\section{Extension to Assumption~\ref{assumption:recourse_lower_bound2}}
\label{appendix:extension}

Fix~$w \in \Q^{A}_+$ and~$b 
\in \Z^{V_+}_+$ according to Assumption~\ref{assumption:recourse_lower_bound2}. Let~$R = (v_1, \ldots, v_\ell)$ and consider the recourse function
\begin{equation}
    \tag{$\tilde{\mc{Q}}^*$}
    \tilde{\mc{Q}}^*(R) \coloneqq \min \left\{ \sum_{\xi \in [N]} \sum_{i \in [\ell]} p_\xi \min\{w_{(v_i, u)} : u \in \{0, v_{i - 1}, v_{i + 1}\} \} \cdot (\bar{y}_R)^\xi_{v_i} : ~ y \in \policy(R) \cap [\mathbf{0}, b]^N \right\}.
\end{equation}
\noindent
The SRI-based formulation developed in Section~\ref{section:scenario_optimal} can be adapted to the recourse function~$\tilde{\mc{Q}}^*$ as follows. 

We introduce variables~$s_v$, for every~$v \in V_+$, and~$r_a$, for every~$a \in A$. To gain an intuition for these variables, suppose that~$\bar{x} \in \mc{X} \cap \Z^E$ is a first-stage solution such that~$R = (v_1, \ldots, v_\ell) \in \mc{R}(\bar{x})$ and let~$i \in [\ell]$ (recall that~$v_{0} = v_{\ell + 1} = 0$). For this solution, we set variable~$s_{v_i}$ so that it represents the expected number of failures that route~$R$ observes at customer~$v_i$. This value is then ``distributed'' along the variables~$r$ via the equation~$s_v = r_{(v_i, 0)} + r_{(v_i, v_{i - 1})} + r_{(v_i, v_{i + 1})}$. The expected recourse cost at a customer~$v \in V_+$ is then given by~$w_{(v_i, 0)}\, r_{(v_i, 0)} + w_{(v_i, v_{i - 1})}\, r_{(v_i, v_{i - 1})} + w_{(v_i, v_{i + 1})}\, r_{(v_i, v_{i + 1})}$. 

Our formulation is detailed in Corollary~\ref{corollary:sri_extension}. Note that variables~$s_v$ could be replaced by the expression~$\sum_{\xi \in [N]} p_\xi y^\xi_v$, but keeping these variables now will be useful later on.

\begin{corollary}
    \label{corollary:sri_extension}
    Problem~$\vrpsd{\mc{X}, \tilde{\mc{Q}}^*}$ can be formulated as
    \begin{subequations}
    \label{formulation:sri_extension}
    \begin{align} 
    \min ~~& c^\T x + \sum_{v \in V_+} \theta_v, & \nonumber \\
    \text{s.t.~~} & y^\xi(\set) \geq \kscen{\set} + x(E(\set)) - |\set|, & \forall \emptyset \subsetneq \set \subseteq V_+, \xi \in [N], \\
    & \theta_v \geq \sum_{u \in V \setminus \{v\}} w_{(v, u)} r_{(v, u)}, & \forall v \in V_+, \\
    & s_v \geq \sum_{\xi \in [N]} p_\xi y^\xi_v, & \forall v \in V_+, \\
    & s_v = \sum_{u \in V \setminus \{v\}} r_{(v, u)}, & \forall v \in V_+, \\
    & r_{(v, u)} \leq b_v \cdot x_{\{v, u\}}, &  \forall (v, u) \in V_+ \times V_+, u \neq v, \\
    & (x, \theta, y, r) \in (\mc{X} \cap \Z^E) \times \R^{V_+}_+ \times [\mathbf{0}, b]^N \times \R^{A}_+.
    \end{align}
    \end{subequations}
\end{corollary}
\begin{proof}
    We first eliminate variables~$s$ to simplify the formulation:
    \begin{subequations}
    \label{formulation:sri_extension2}
    \begin{align} 
    \min ~~& c^\T x + \sum_{v \in V_+} \theta_v, & \nonumber \\
    \text{s.t.~~} & y^\xi(\set) \geq \kscen{\set} + x(E(\set)) - |\set|, & \forall \emptyset \subsetneq \set \subseteq V_+, \xi \in [N], \\
    & \theta_v \geq \sum_{u \in V \setminus \{v\}} w_{(v, u)} r_{(v, u)}, & \forall v \in V_+, \\
    & \sum_{u \in V \setminus \{v\}} r_{(v, u)} \geq \sum_{\xi \in [N]} p_\xi y^\xi_v, & \forall v \in V_+, \\
    & r_{(v, u)} \leq b_v \cdot x_{\{v, u\}}, & \forall (v, u) \in V_+ \times V_+, u \neq v, \\
    & (x, \theta, y, r) \in (\mc{X} \cap \Z^E) \times \R^{V_+}_+ \times [\mathbf{0}, b]^N \times \R^{A}_+.
    \end{align}
    \end{subequations}
    
    Now fix~$x$ to some~$\bar{x} \in \mc{X} \cap \Z^E$ in the above formulation. For each~$R = (v_1, \ldots, v_\ell) \in \mc{R}(\bar{x})$ and~$i \in [\ell]$, set~$\tilde{w}_{v_i} = \min\{w_{(v_i, 0)}, w_{(v_i, v_{i - 1})}, w_{(v_i, v_{i + 1})} \}$. Note that we may assume that if variable~$r_{(v_i, u)}$ attains a positive value, then~$w_{(v_i, u)} = \tilde{w}_{v_i}$. Moreover, since~$\sum_{\xi \in [N]} p_\xi = 1$, we know that~$\sum_{\xi \in [N]} p_\xi y^\xi_v \leq b_v$, for all~$v \in V_+$. Hence, fixing~$x = \bar{x}$ we rewrite Formulation~\eqref{formulation:sri_extension2} as:
    \begin{align*} 
    \min ~~& c^\T \bar{x} + \sum_{v \in V_+} \theta_v, & \nonumber \\
    \text{s.t.~~} & y^\xi(\set) \geq \kscen{\set} + \bar{x}(E(\set)) - |\set|, & \forall \emptyset \subsetneq \set \subseteq V_+, \xi \in [N], \\
    & \theta_v \geq \tilde{w}_v \sum_{\xi \in [N]} p_\xi y^\xi_v, & \forall v \in V_+, \\
    & y \in [\mathbf{0}, b]^N.
    \end{align*}

    \noindent
    By Theorem~\ref{theorem:sri_formulation}, this is equivalent to
    \begin{align*}
        &~~c^\T \bar{x} + \min \left\{\sum_{\xi \in [N]} \sum_{v \in V_+} p_\xi \tilde{w}_v y^\xi_v : y \in \conv(\policy(\bar{x}) \cap [\mathbf{0}, b]^N) \right\} \\
        = &~~c^\T \bar{x} + \min \left\{ \sum_{\xi \in [N]} \sum_{v \in V_+} p_\xi \tilde{w}_v y^\xi_v : y \in \policy(\bar{x}) \cap [\mathbf{0}, b]^N \right\},
    \end{align*}
    so we are done by Fact~\ref{fact:policy_split}.
\end{proof}

Using the same argument as in Proposition~\ref{proposition:valid_cuts_recourse}, we thus obtain the following result.
\begin{corollary}
    \label{corollary:valid_cuts_recourse_extension}
    For any recourse function~$\mc{Q}$ that satisfies Assumption~\ref{assumption:recourse_lower_bound2},
    $$\Fq{\mc{X}, \mc{Q}} \subseteq \tsc{proj}_{(x, \theta)} \left(\{(x, \theta, y, r, s) : (x, \theta, y, r, s) \text{ is feasible for Formulation~\eqref{formulation:sri_extension}}\}\right).$$
\end{corollary}

We can also extend some of our projection results for this setting. Let~$\bar{x} \in \mc{X} \cap \Z^E$ and define
\begin{equation}
\label{set:sri_bar2}
\widetilde{\tsc{sri}}(\bar{x}) \coloneqq 
\left\{(s,  y) \in \R^{V_+}_+ \times [\mathbf{0}, b]^N :~
\begin{aligned}
    & y^\xi(\set) \geq \kscen{\set} + \bar{x}(E(\set)) - |\set|, & \forall \emptyset \subsetneq \set \subseteq V_+,~\xi \in [N] \\
    & s_v \geq \sum_{\xi \in [N]} p_\xi y^\xi_v, & \forall v \in V_+
\end{aligned}
\right\}.
\end{equation}

\noindent
That is,~$\widetilde{\tsc{sri}}(\bar{x})$ is equivalent to~$\srihat(\bar{x})$ in the special case that~$w_v = 1$, for every customer~$v \in V_+$. 

For any~$\alpha \geq 0$,~$\beta \leq 0$ and~$v \in V_+$, define
\begin{equation}
\label{eq:phi_tilde}
\tilde{\phi}_v(\alpha, \beta) \coloneqq \left( \max_{\xi \in [N]}\left\{ \frac{\beta^\xi_v + \sum_{\set \subseteq V_+ : v \in \set} \alpha^\xi_\set}{p_\xi} \right\} \right)^+.
\end{equation}
By the same argument as in Proposition~\ref{proposition:proj_gamma_characterization}, we obtain the following characterization of~$\proj_{s} (\widetilde{\tsc{sri}}(\bar{x}))$:
\begin{corollary}
    For every~$\bar{x} \in \mc{X}$,
    {\small
    \begin{equation*}
    \proj_{s} (\widetilde{\tsc{sri}}(\bar{x})) =
    \left\{ s\in \R^{V_+}_+ : \sum_{v \in V_+} \tilde{\phi}_v(\alpha, \beta) \, s_v \geq \sum_{\xi \in [N]} \sum_{\set \subseteq V_+} \alpha^\xi_\set \, \bar{x}(E(\set)) + \nufunction{\alpha, \beta}, \quad \forall \alpha \geq 0, \beta \leq 0
    \right\}.
    \end{equation*}}
\end{corollary}

\noindent
Consequently, Formulation~\eqref{formulation:sri_extension} can be reformulated as
\begin{subequations}
\label{formulation:sri_extension3}
\begin{align} 
\min ~~& c^\T x + \sum_{v \in V_+} \theta_v, & \nonumber \\
\text{s.t.~~} & \tilde{\phi}_v(\alpha, \beta) \, s_v \geq \sum_{\xi \in [N]} \sum_{\set \subseteq V_+} \alpha^\xi_\set \, x(E(\set)) + \nufunction{\alpha, \beta}, & \forall \alpha \geq 0, \beta \leq 0, \\
& \theta_v \geq \sum_{u \in V \setminus \{v\}} w_{(v, u)} r_{(v, u)}, & \forall v \in V_+, \\
& s_v = \sum_{u \in V \setminus \{v\}} r_{(v, u)}, & \forall v \in V_+, \label{eq:form:sri3:svar} \\
& r_{(v, u)} \leq b_v \cdot x_{\{v, u\}}, & \forall (v, u) \in V_+ \times V_+, u \neq v, \\
& (x, \theta, r) \in (\mc{X} \cap \Z^E) \times \R^{V_+}_+ \times \R^A_+,
\end{align}
\end{subequations}
and we may eliminate variables~$s$ by using Equations~\eqref{eq:form:sri3:svar}.

\section{\mytitle{Separation algorithms}}
\label{appendix:separation_algorithm}

The following are the pseudocodes for our separation of ILS cuts and SRIs. In these algorithms, we implicitly assume that we have flags indicating whether~$\mc{X} = \Xcvrp$ or~$\mc{X} = \Xsub$, and whether~$\mc{Q} = \optQ$ or $\mc{Q} = \Qc$. Algorithm~\tsc{GetPartialRoutes} corresponds to the routine described in~\cite{part1} to find partial routes, and the value~$\mc{L}_C(H)$ refers to their proposed recourse lower bound. When~$H$ corresponds to a route~$R$, we have~$\mc{L}_C(H) = \Qc(R)$. We refer the reader to~\cite{part1} for more details on the computation of~$\mc{L}_C(H)$.

\begin{algorithm}[H]
\begin{algorithmic}[1]
    \Procedure {\textsc{AddSetCutOrSRI}}{$\bar{x}, \bar{\theta}, \mathtt{use\_sri}, \set, k'$}        
    \State{Compute~$\mc{L}^*(\set, k')$ using the greedy algorithm mentioned in Section~\ref{subsection:set_cuts_comparison}.}

    \If{$\mathtt{use\_sri} = \mathtt{true}$}
        \State{$\Xi \gets \{\xi \in [N] : \kscen{\set} + \bar{x}(E(\set)) - |\set| > 0 \}$}
        \If{$\mc{X} = \Xcvrp$}
            \State{$\Xi \gets \{\xi \in \Xi : \kscen{\set} > \kbar{\set}\}$}
        \EndIf
        \If{$\sum_{v \in \set} \max_{\xi \in \Xi} \left\{\frac{1}{p_\xi w_v}\right\} \, \bar{\theta}_v < |\Xi| \cdot (\bar{x}(E(\set)) - |\set|) + \sum_{\xi \in \Xi} \kscen{\set}$}
            \State{Add projected SRI~$\sum_{v \in \set} \max_{\xi \in [\Xi]} \left\{\frac{1}{p_\xi w_v}\right\} \, \theta_v \geq |\Xi| \cdot (x(E(\set)) - |\set|) + \sum_{\xi \in \Xi} \kscen{\set}$.}
            \State{\textbf{return}~$\mathtt{true}$}
        \EndIf
    \Else
        \If{$\bar{\theta}(\set) < \mc{L}^*(\set, k') \cdot (\bar{x}(E(\set)) - |\set| + k')$}
            \State{Add set cut~$\theta(\set) \geq \mc{L}^*(\set, k') \cdot (x(E(\set)) - |\set| + k')$.}
            \State{\textbf{return}~$\mathtt{true}$}
        \EndIf
    \EndIf
    \State{\textbf{return}~$\mathtt{false}$}
\EndProcedure
\end{algorithmic}
\caption{\textsc{AddSetCutOrSRI}}
\label{algorithm:add_set_cut}
\end{algorithm}

\begin{algorithm}[H]
\begin{algorithmic}[1]
    \Procedure {\textsc{AddPartialRouteCutOrSRI}}{$\bar{x}, \bar{\theta}, \mathtt{use\_sri}, H$}        
    \State{Compute~$\mc{L}^*(H)$ using Formulation~\eqref{formulation:recourse_lower_bound}.}

    \If{$\mathtt{use\_sri} = \mathtt{true}$}
        \State{Compute~$(\alpha, \beta)$ using Lemma~\ref{lemma:existence_dual} and Theorem~\ref{theorem:dominance_sri}.}
        \If{$\sum_{v \in V_+} \phiv(\alpha, \beta) \, \bar{\theta}_v < \sum_{\xi \in [N]} \sum_{\set \subseteq V_+} \alpha^\xi_\set \, \bar{x}(E(\set)) + \nufunction{\alpha, \beta}$}
            \State{Add projected SRI~$\sum_{v \in V_+} \phiv(\alpha, \beta) \, \theta_v \geq \sum_{\xi \in [N]} \sum_{\set \subseteq V_+} \alpha^\xi_\set \, x(E(\set)) + \nufunction{\alpha, \beta}$.}
            \State{\textbf{return}~$\mathtt{true}$}
        \EndIf
    \Else
        \If{$\bar{\theta}(V_+(H)) < \mc{L}^*(H) \cdot \Wof(\bar{x} ; \supsetXh{H})$}
            \State{Add partial route cut~$\theta(V_+(H)) \geq \mc{L}^*(H) \cdot \Wof(x ; \supsetXh{H})$.}
            \State{\textbf{return}~$\mathtt{true}$}
        \EndIf
    \EndIf
    \State{\textbf{return}~$\mathtt{false}$}
\EndProcedure
\end{algorithmic}
\caption{\textsc{AddPartialRouteCutOrSRI}}
\label{algorithm:add_partial_route_cut}
\end{algorithm}

\begin{algorithm}[H]
\textbf{Input:} A candidate solution~$(\bar{x}, \bar{\theta}) \in \R^E_+ \times \Q^{V_+}_+$ and a boolean flag~$\mathtt{use\_sri}$, which indicates if we use projected SRIs or not.

\begin{algorithmic}[1]
    \Procedure {\textsc{SeparationVRPSD}}{$\bar{x}, \bar{\theta}, \mathtt{use\_sri}$}
    \State {Call CVRPSEP to get a family of customer sets~$\mc{\set} \subseteq 2^{V_+}$ (if~$\mc{X} = \Xsub$, we use~$\bar{d}(V_+) + 1$ as the capacity in CVRPSEP).}
    \For {$\set \in \mc{\set}$}
        \State{$k' \gets 1 + (\kbar{\set} - 1) \cdot \mb{I}(\mc{X} = \Xcvrp)$}
        \State{Add inequality~$x(E(\set)) \leq |\set| - k'$.}
        \State{$\tsc{AddSetCutOrSRI}(\bar{x}, \bar{\theta}, \mathtt{use\_sri}, \set, k')$}
    \EndFor
    \If {$\mc{\set} \neq \emptyset$}
        \State{\textbf{return}}
    \EndIf
    \State {$\mc{H} \gets \tsc{GetPartialRoutes}(\bar{x}, \bar{\theta})$}
    \For {$H \in \mc{H}$}
        \State{$k' \gets 1 + (\kbar{\set} - 1) \cdot \mb{I}(\mc{X} = \Xcvrp)$}
        \If{$\tsc{AddSetCutOrSRI}(\bar{x}, \bar{\theta}, \mathtt{use\_sri}, V_+(H), k')$}
            \State{\textbf{continue}}
        \EndIf
        \If{$\tsc{AddPartialRouteCutOrSRI}(\bar{x}, \bar{\theta}, \mathtt{use\_sri}, H)$}
            \State{\textbf{continue}}
        \EndIf
        \If{$\mc{Q} = \Qc$~\textbf{and}~$\bar{\theta}(V_+(H)) < \mc{L}_C(H) \cdot \Wof(\bar{x} ; \supsetXh{H})$}
            \State{Add inequality~$\theta(V_+(H)) \geq \mc{L}_C(H) \cdot \Wof(x ; \supsetXh{H})$.}
        \EndIf
    \EndFor
\EndProcedure
\end{algorithmic}
\caption{\textsc{SeparationVRPSD}}
\label{algorithm:vrpsd_separation}
\end{algorithm}

\section{\mytitle{Detailed computational experiments}}
\label{appendix:experiments}

The following figures were generated analogously to Figure~\ref{figure:experiments_time}, but the considered instances were split according to the first-stage feasibility region~$\mc{X}$, the recourse function~$\mc{Q}$, and the instance set. For example, in Figure~\ref{figure:jabali_cvrp_scenopt_time}, a point~$p = (p_1, p_2) \in \R^2$ (with~$p_1 \leq 1800$) on the curve of~\tsc{sri} indicates that algorithm~\tsc{sri} solved a fraction~$p_2$ of the instances from~\cite{jabali2014} with~$\mc{X} = \Xcvrp$ and~$\mc{Q} = \optQ$ within~$p_1$ seconds. For~$p_1 > 1800$, the value of~$p_2$ is defined relative to the final optimality gap.

\begin{figure}
    \centering
    \subfloat[Execution time and final optimality gaps.]{\includegraphics[scale=0.15]{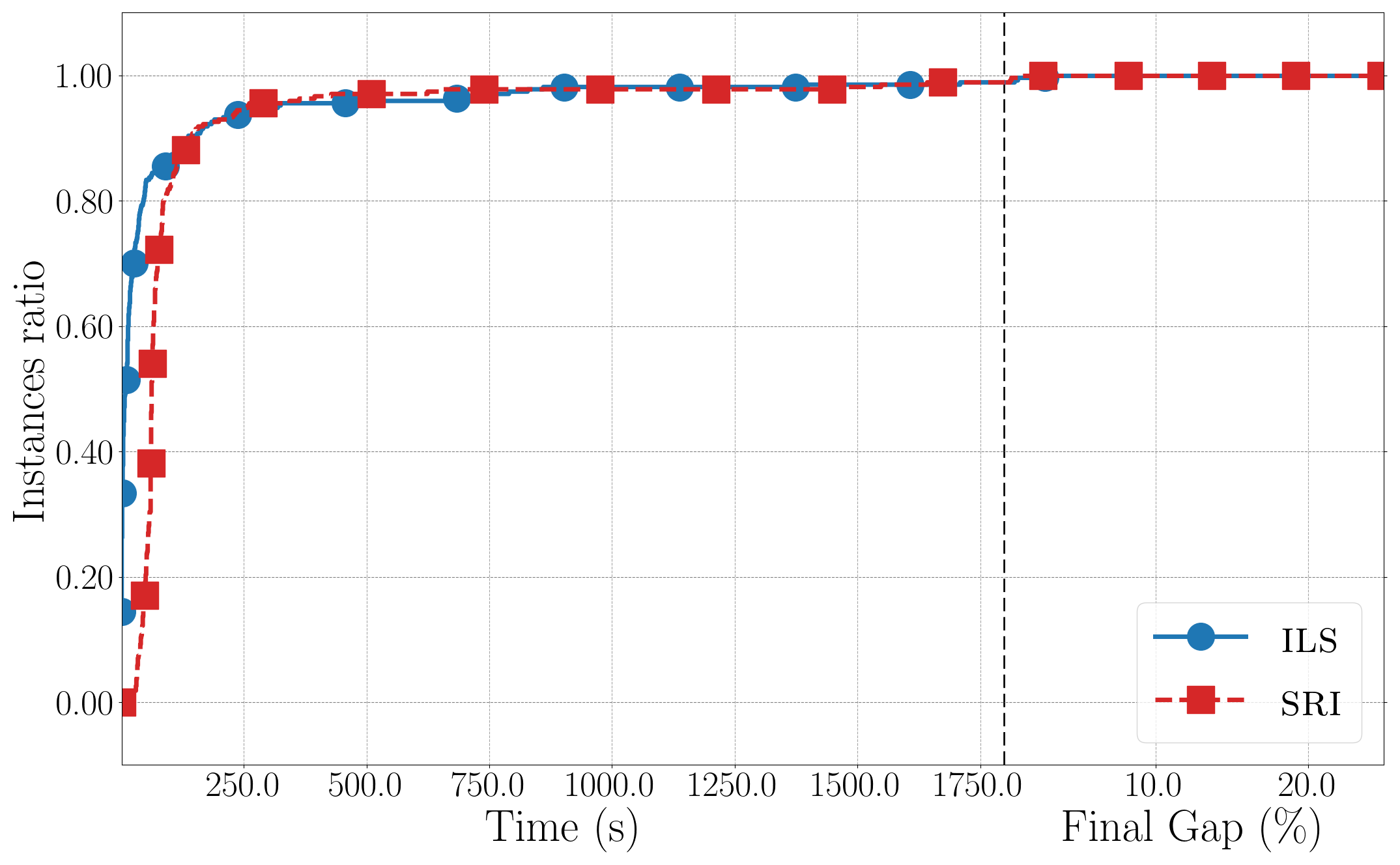} \label{figure:jabali_cvrp_scenopt_time}}
    \hfill
    \subfloat[Root gaps.]{\includegraphics[scale=0.15]{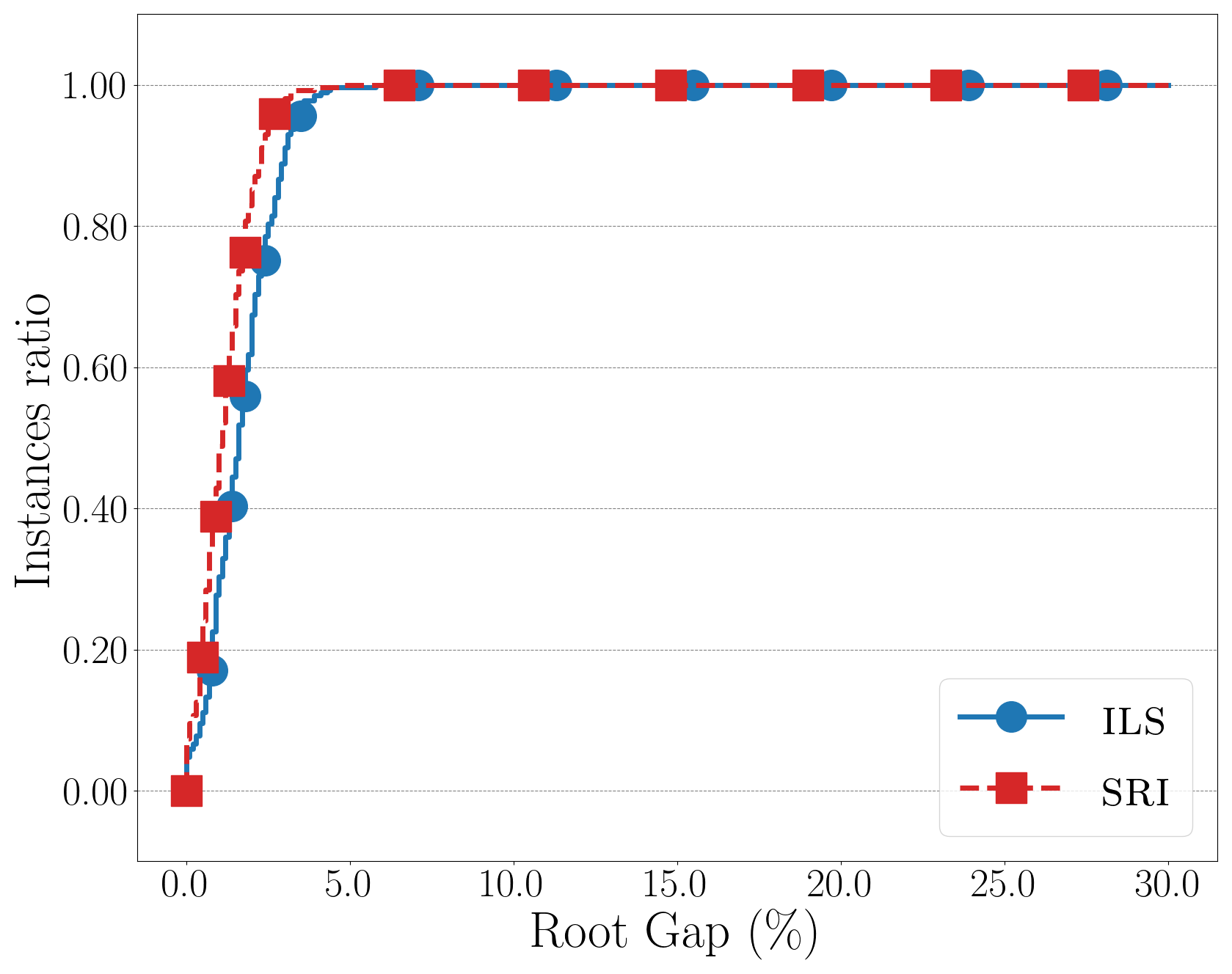}\label{figure:jabali_cvrp_scenopt_gap}}
    \\[0.5cm]
    \caption{Empirical cumulative distribution of the execution times and root gaps for~\cite{jabali2014} instances, with~$\mc{X} = \Xcvrp$ and~$\mc{Q} = \optQ$.}
    \label{figure:jabali_cvrp_scenopt}
\end{figure}

\begin{figure}
    \centering
    \subfloat[Execution time and final optimality gaps.]{\includegraphics[scale=0.15]{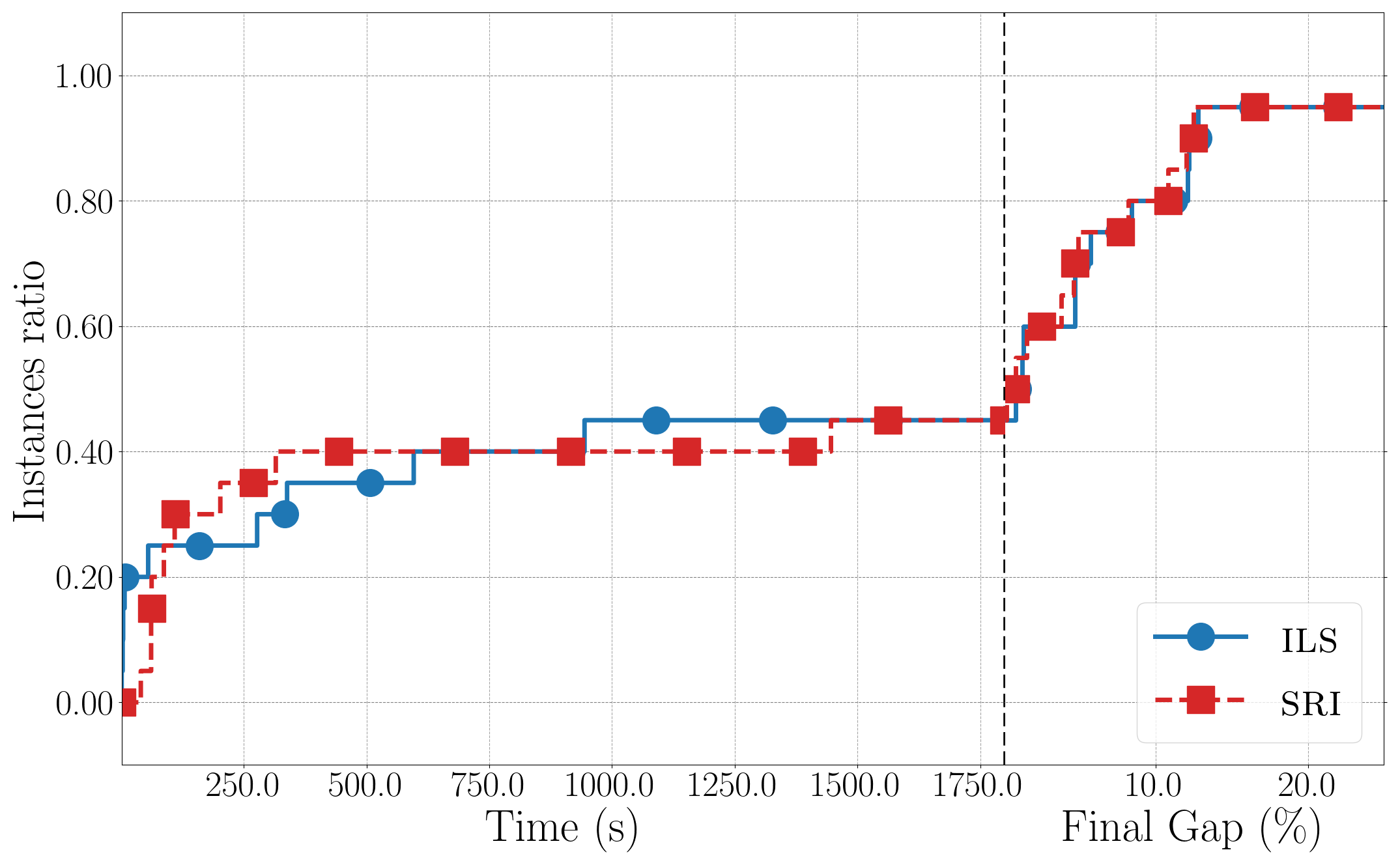} \label{figure:dinh_cvrp_scenopt_time}}
    \hfill
    \subfloat[Root gaps.]{\includegraphics[scale=0.15]{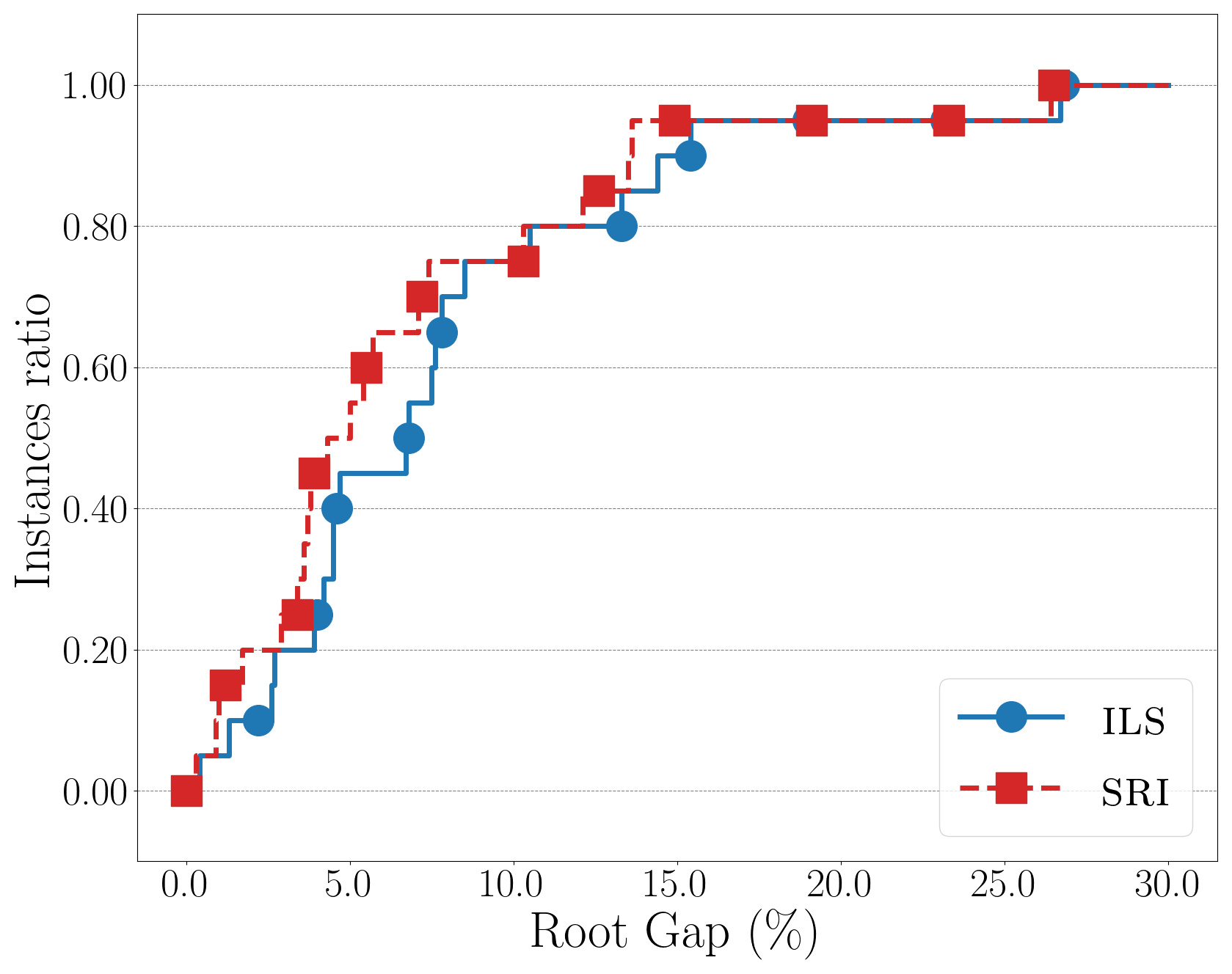}\label{figure:dinh_cvrp_scenopt_gap}}
    \\[0.5cm]
    \caption{Empirical cumulative distribution of the execution times and root gaps for~\cite{Dinh2018} instances, with~$\mc{X} = \Xcvrp$ and~$\mc{Q} = \optQ$.}
    \label{figure:dinh_cvrp_scenopt}
\end{figure}

\begin{figure}
    \centering
    \subfloat[Execution time and final optimality gaps.]{\includegraphics[scale=0.15]{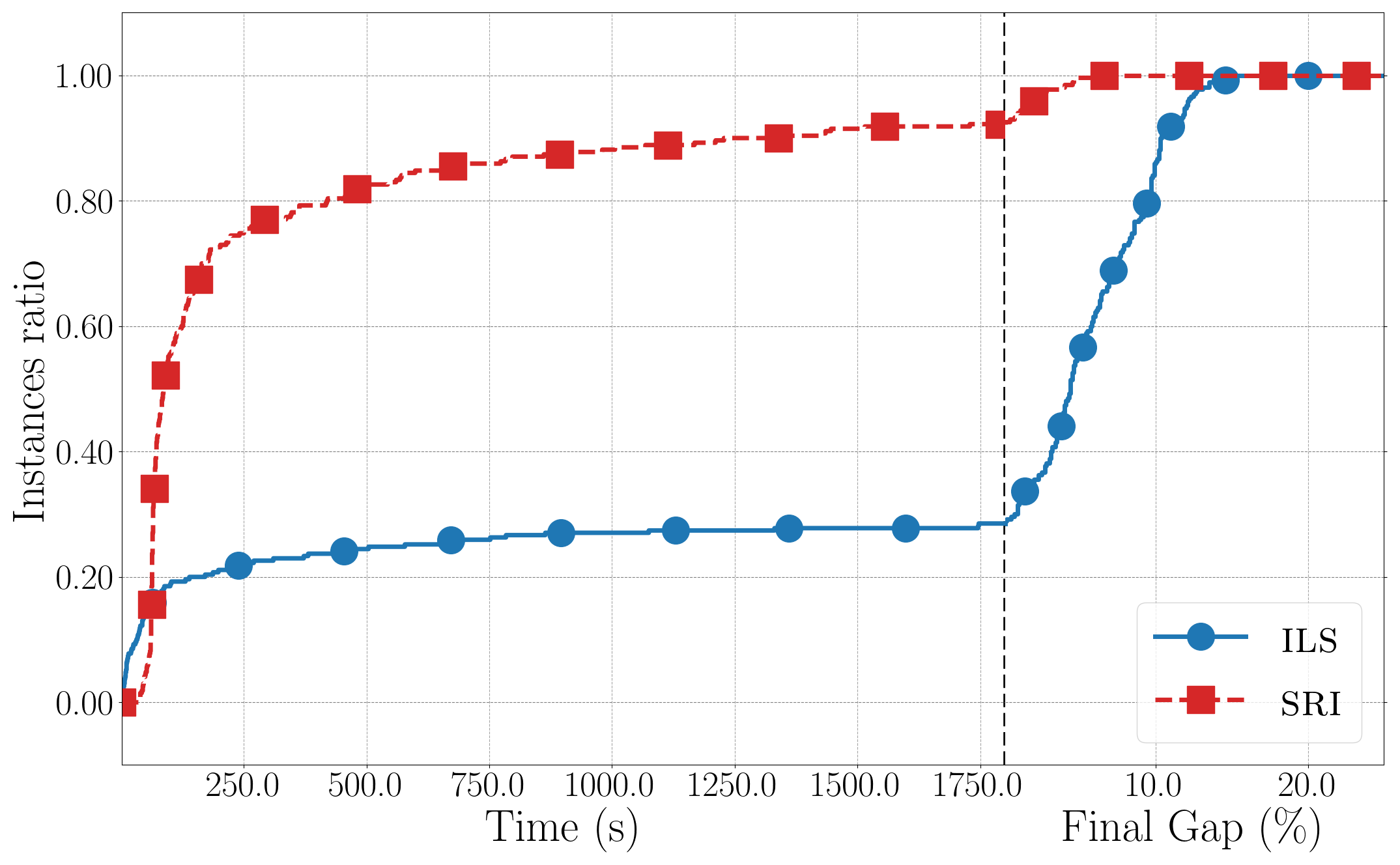} \label{figure:jabali_basic_scenopt_time}}
    \hfill
    \subfloat[Root gaps.]{\includegraphics[scale=0.15]{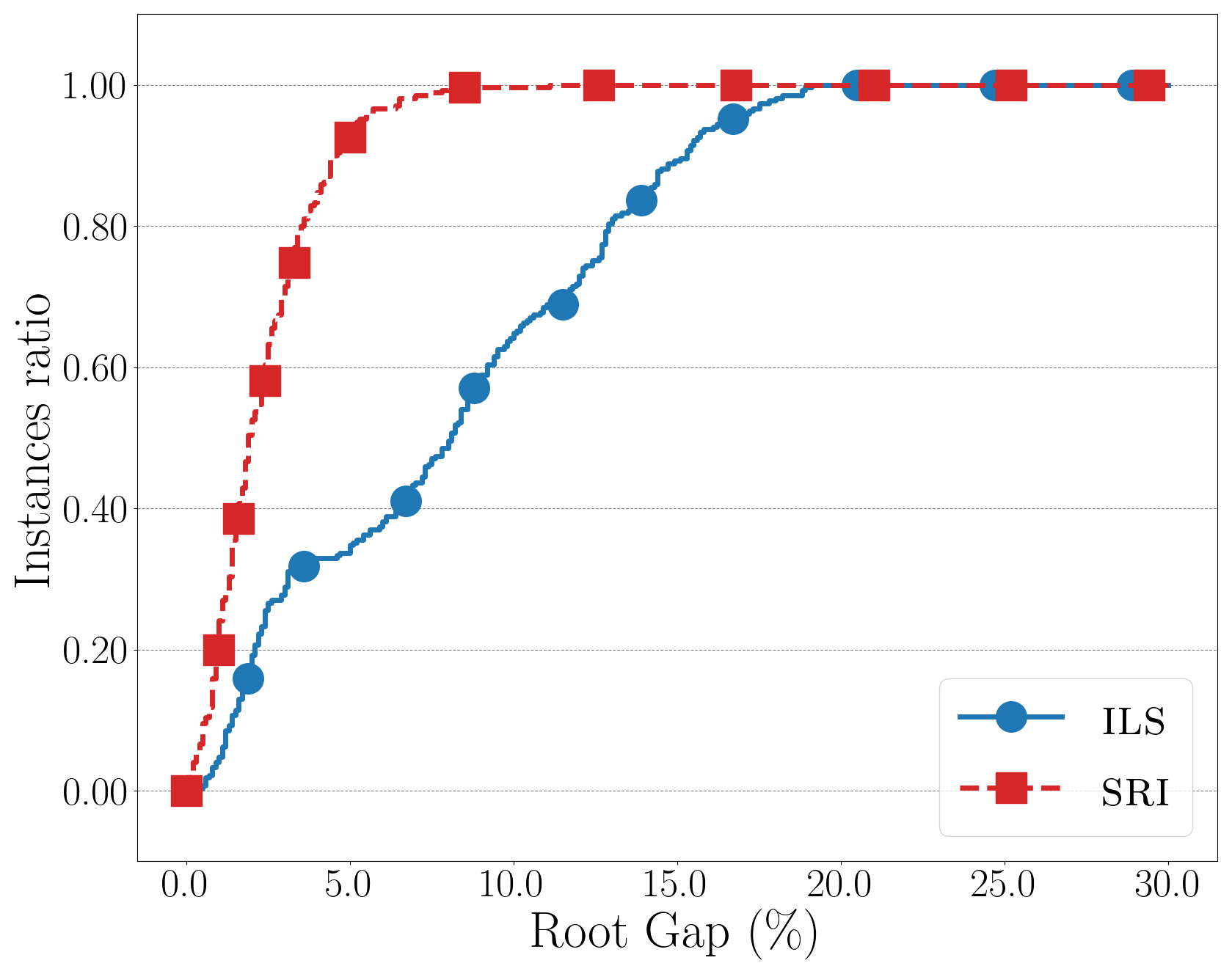}\label{figure:jabali_basic_scenopt_gap}}
    \\[0.5cm]
    \caption{Empirical cumulative distribution of the execution times and root gaps for~\cite{jabali2014} instances, with~$\mc{X} = \Xsub$ and~$\mc{Q} = \optQ$.}
    \label{figure:jabali_basic_scenopt}
\end{figure}

\begin{figure}
    \centering
    \subfloat[Execution time and final optimality gaps.]{\includegraphics[scale=0.15]{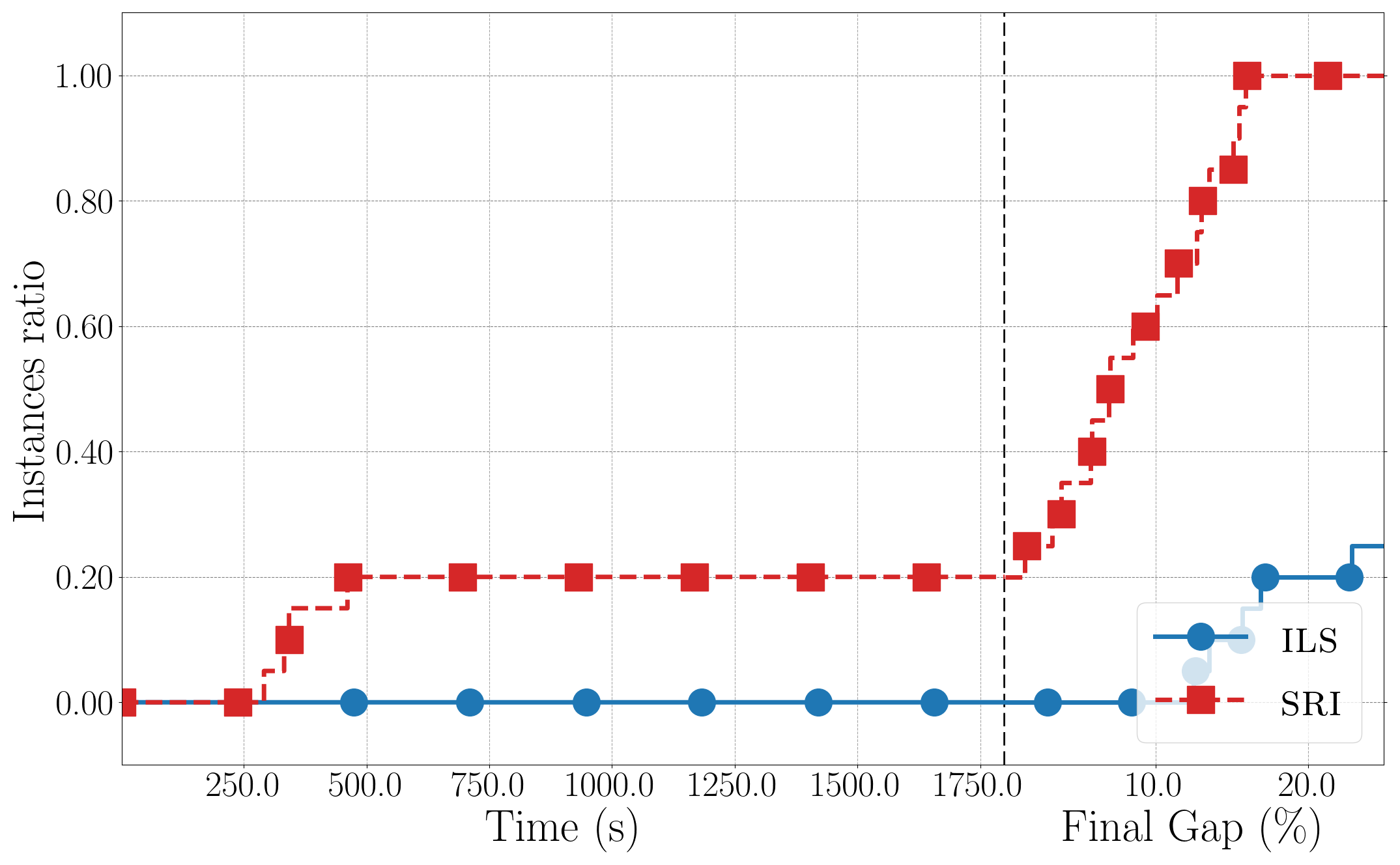} \label{figure:dinh_basic_scenopt_time}}
    \hfill
    \subfloat[Root gaps.]{\includegraphics[scale=0.15]{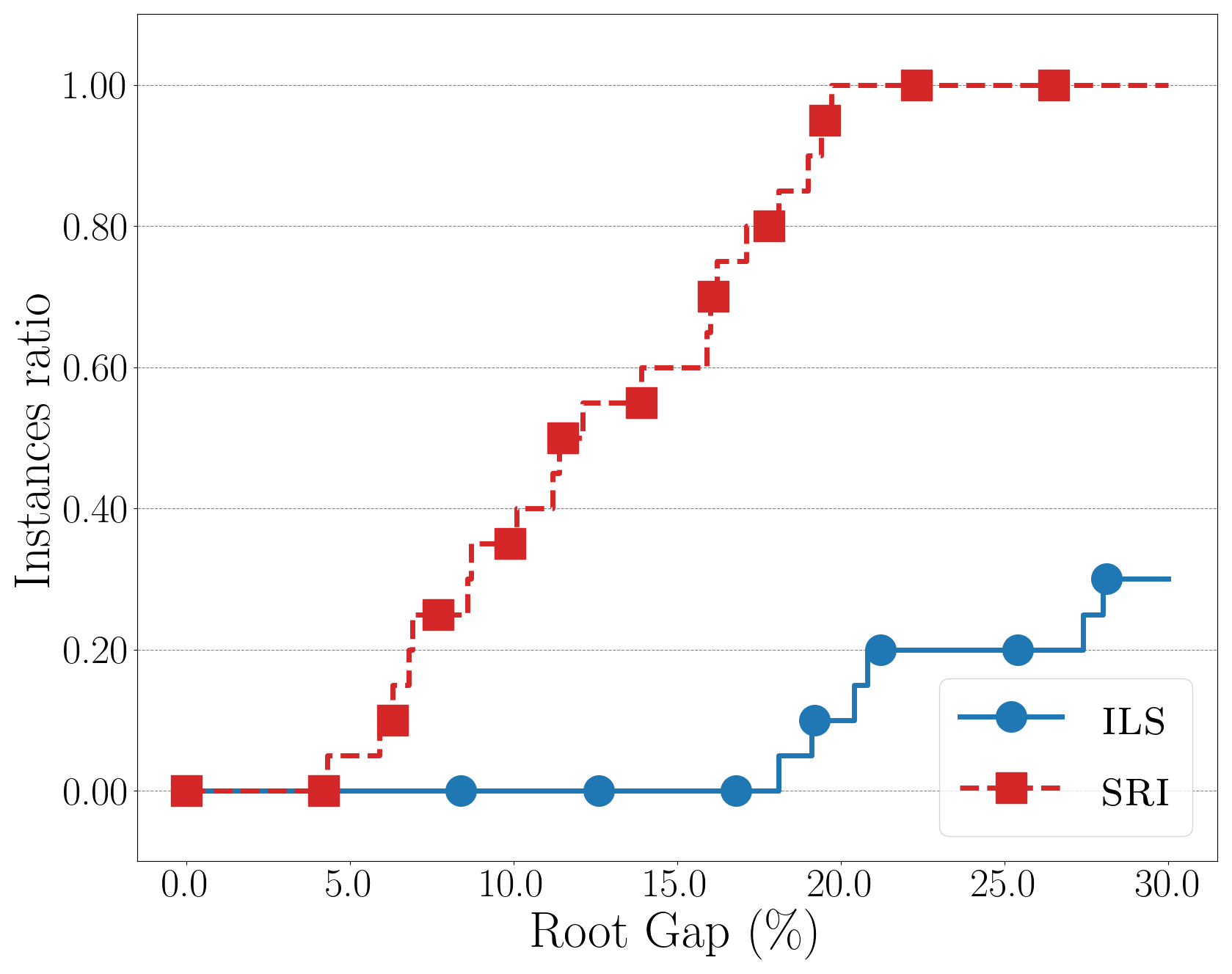}\label{figure:dinh_basic_scenopt_gap}}
    \\[0.5cm]
    \caption{Empirical cumulative distribution of the execution times and root gaps for~\cite{Dinh2018} instances, with~$\mc{X} = \Xsub$ and~$\mc{Q} = \optQ$.}
    \label{figure:dinh_basic_scenopt}
\end{figure}

\begin{figure}
    \centering
    \subfloat[Execution time and final optimality gaps.]{\includegraphics[scale=0.15]{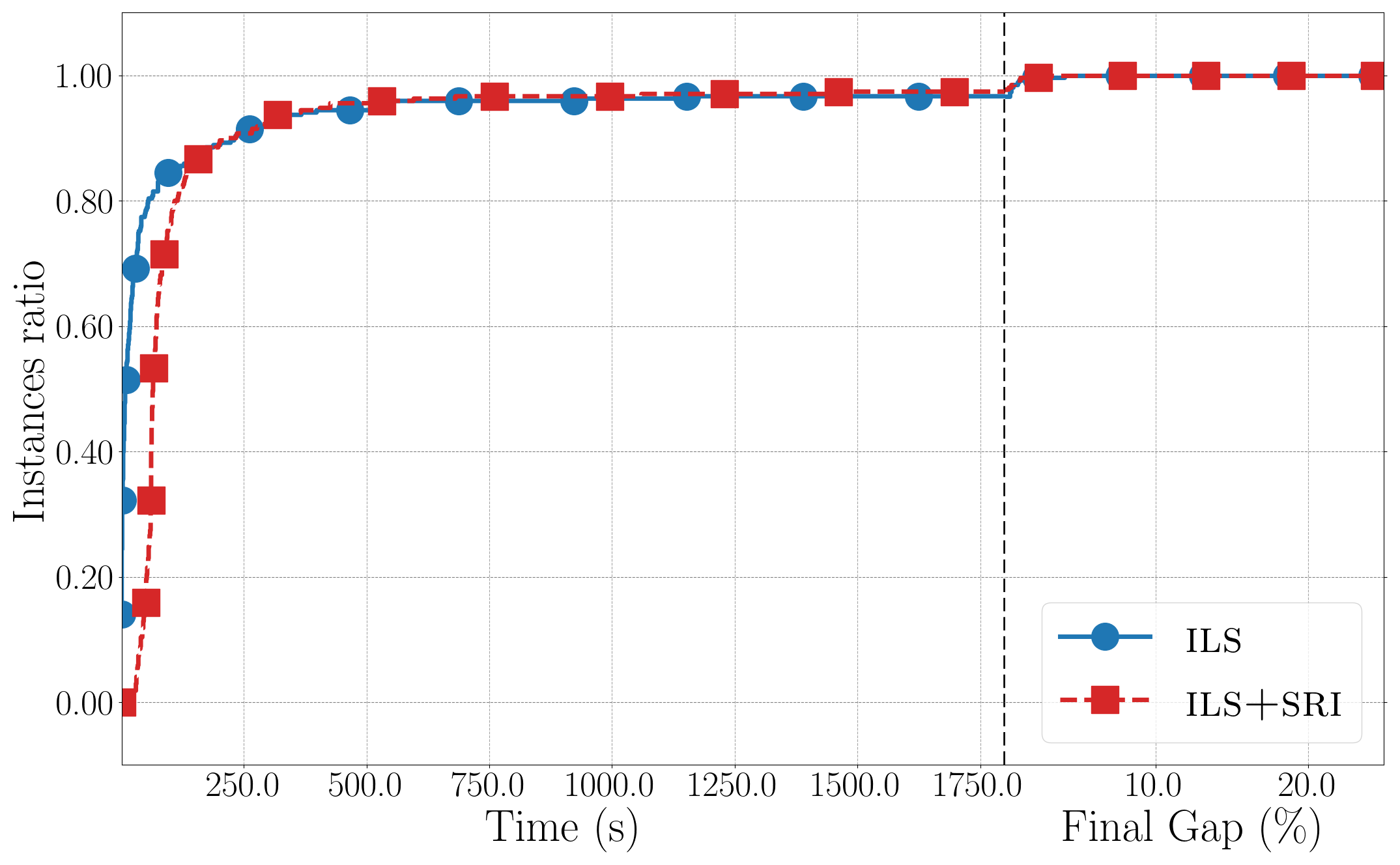} \label{figure:jabali_cvrp_classical_time}}
    \hfill
    \subfloat[Root gaps.]{\includegraphics[scale=0.15]{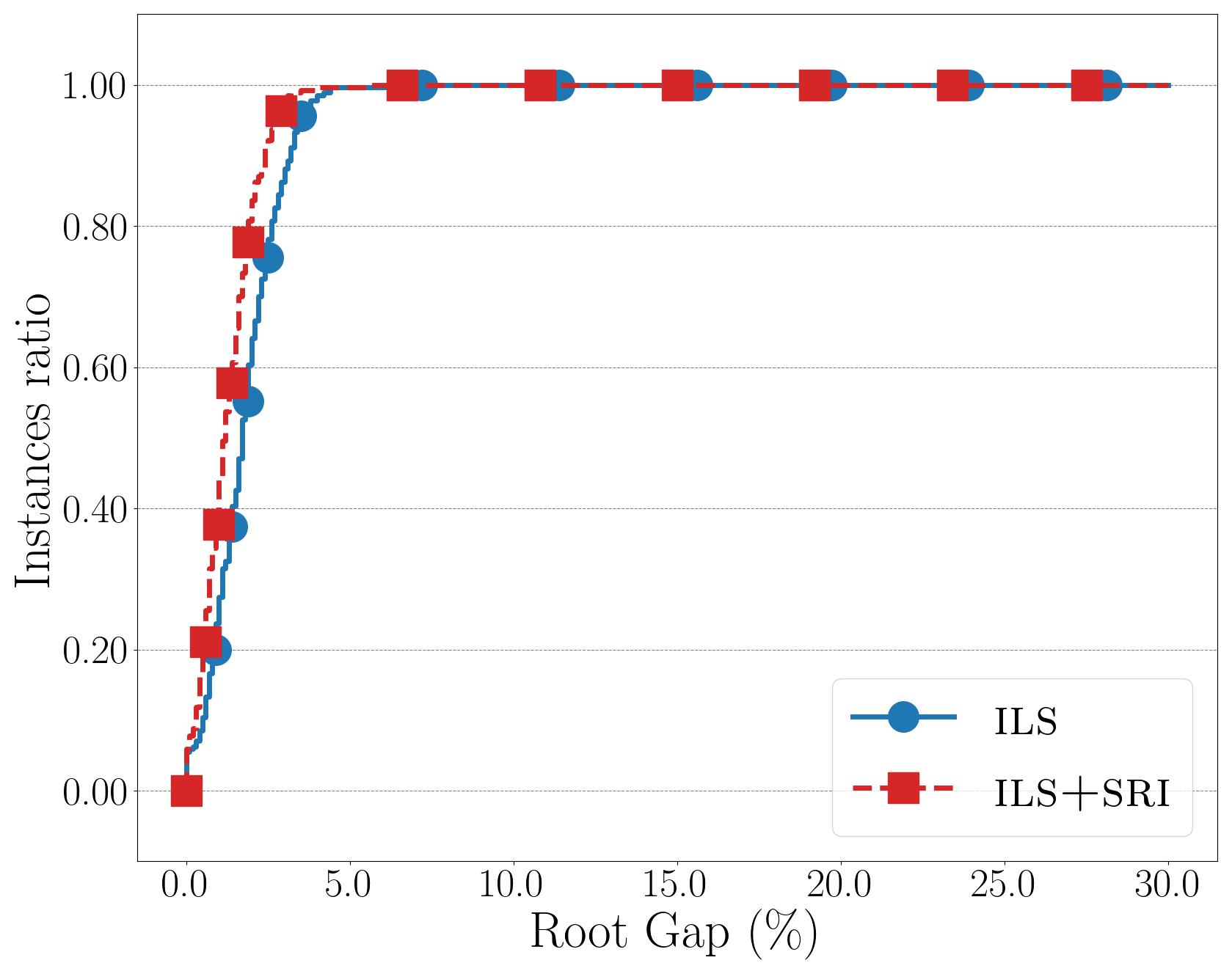}\label{figure:jabali_cvrp_classical_gap}}
    \\[0.5cm]
    \caption{Empirical cumulative distribution of the execution times and root gaps for~\cite{jabali2014} instances, with~$\mc{X} = \Xcvrp$ and~$\mc{Q} = \Qc$.}
    \label{figure:jabali_cvrp_classical}
\end{figure}

\begin{figure}
    \centering
    \subfloat[Execution time and final optimality gaps.]{\includegraphics[scale=0.15]{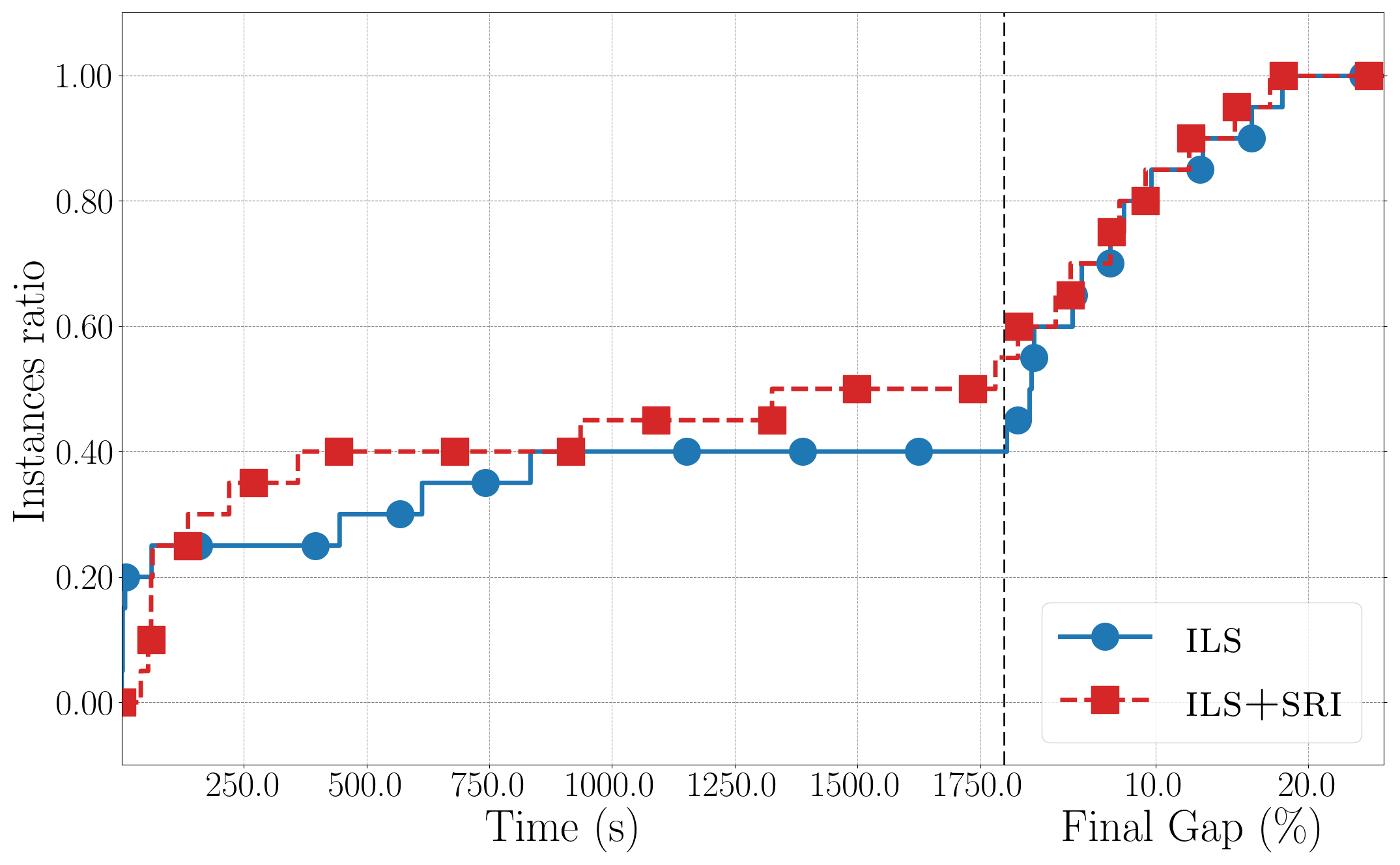} \label{figure:dinh_cvrp_classical_time}}
    \hfill
    \subfloat[Root gaps.]{\includegraphics[scale=0.15]{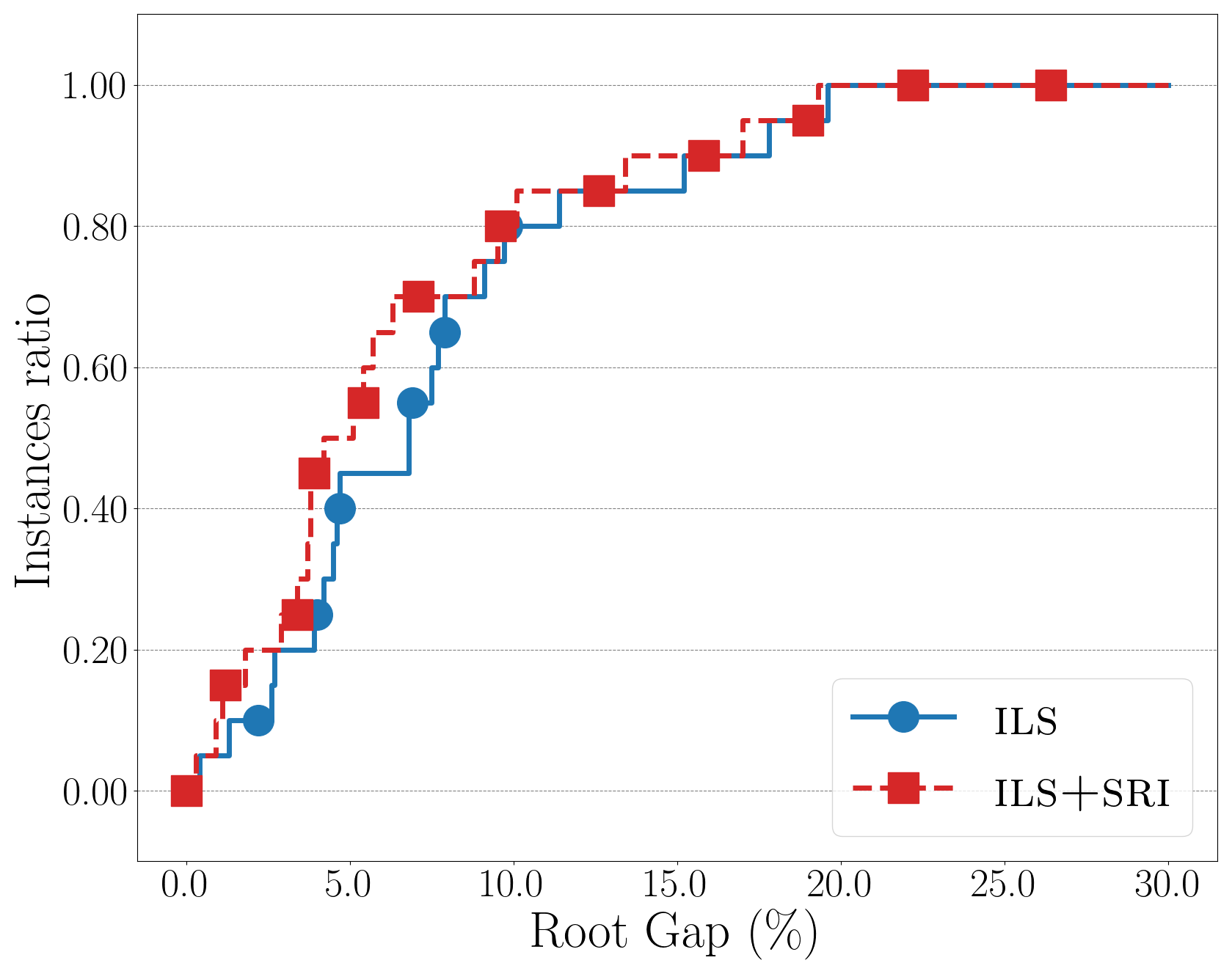}\label{figure:dinh_cvrp_classical_gap}}
    \\[0.5cm]
    \caption{Empirical cumulative distribution of the execution times and root gaps for~\cite{Dinh2018} instances, with~$\mc{X} = \Xcvrp$ and~$\mc{Q} = \Qc$.}
    \label{figure:dinh_cvrp_classical}
\end{figure}

\begin{figure}
    \centering
    \subfloat[Execution time and final optimality gaps.]{\includegraphics[scale=0.15]{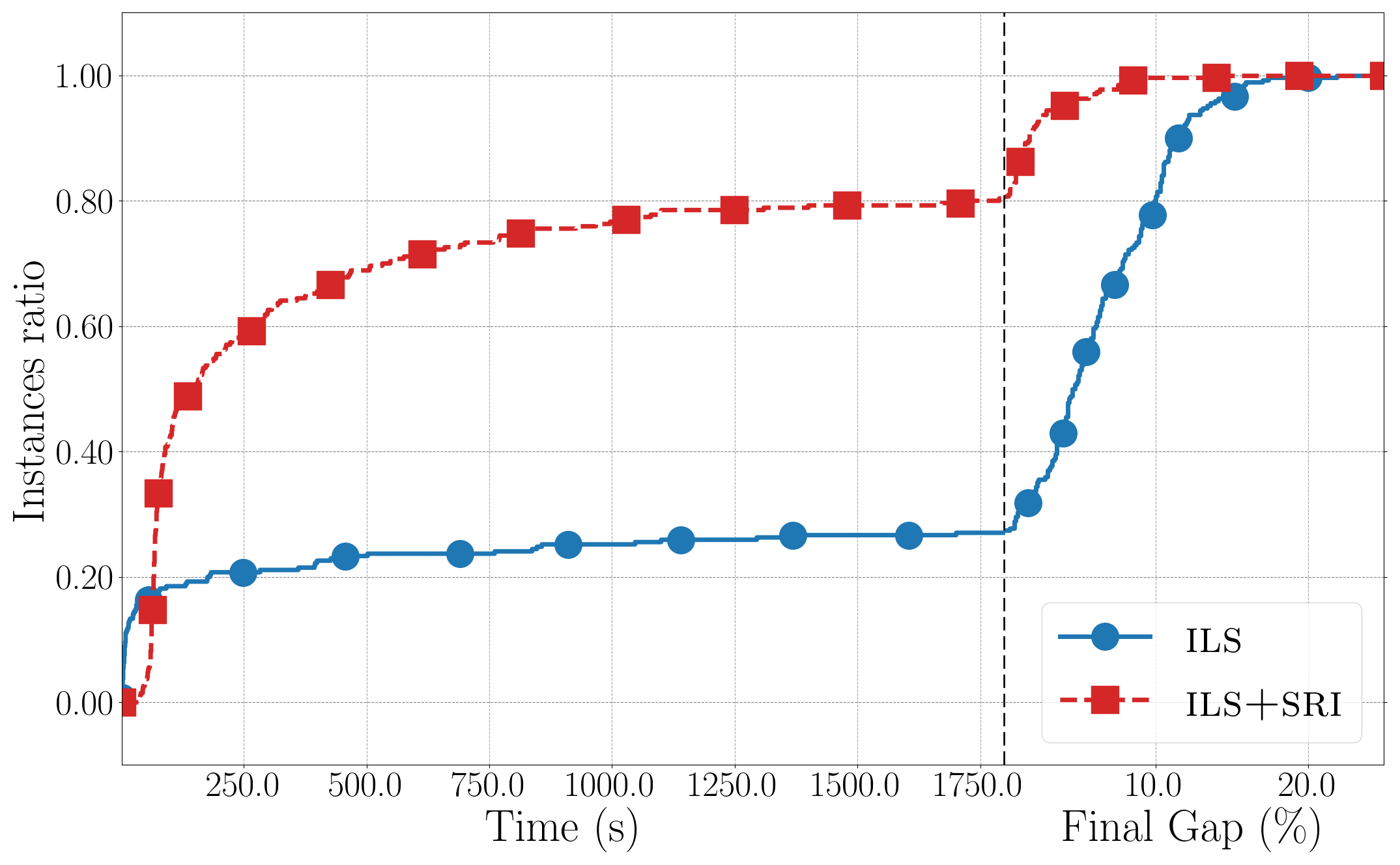} \label{figure:jabali_basic_classical_time}}
    \hfill
    \subfloat[Root gaps.]{\includegraphics[scale=0.15]{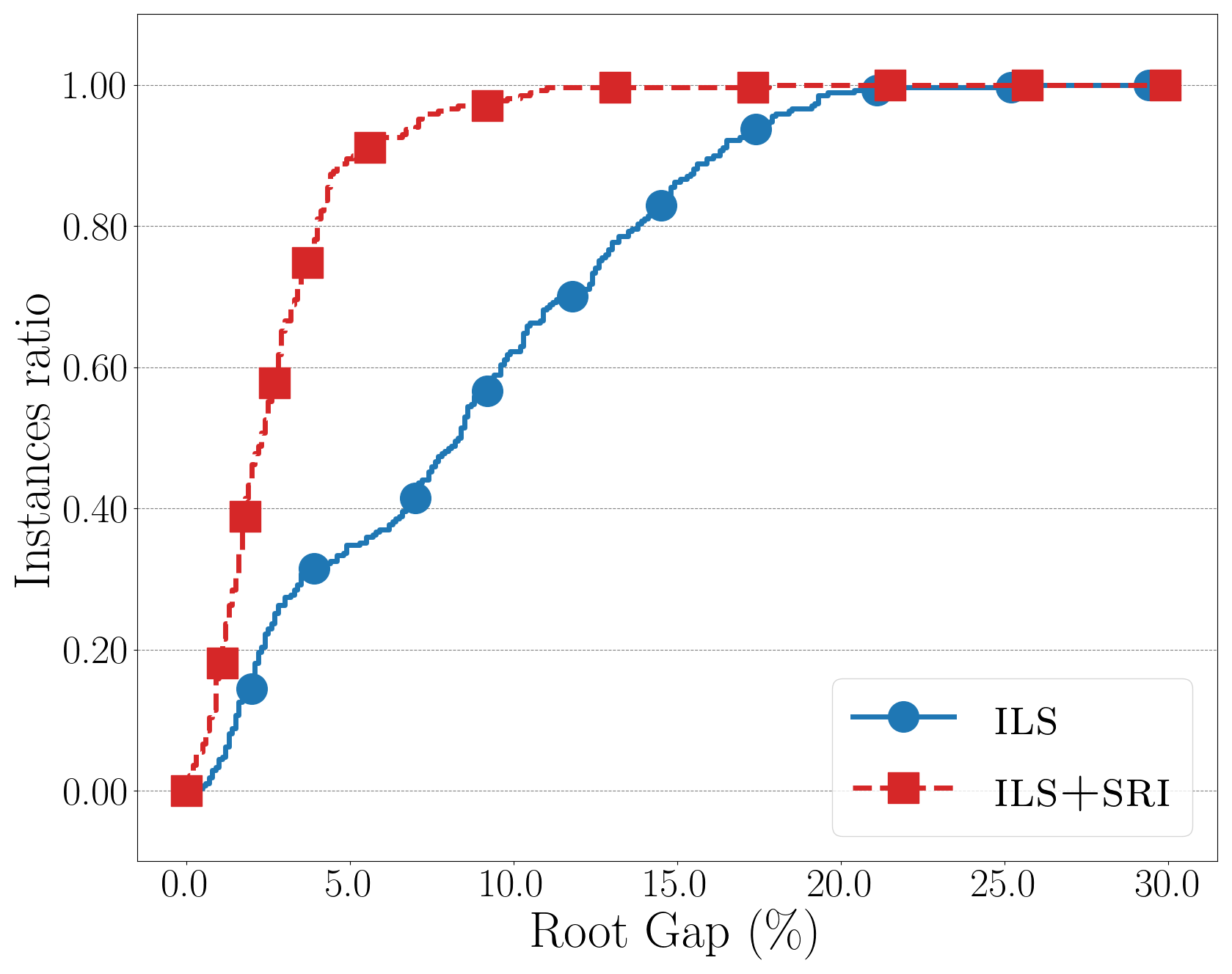}\label{figure:jabali_basic_classical_gap}}
    \\[0.5cm]
    \caption{Empirical cumulative distribution of the execution times and root gaps for~\cite{jabali2014} instances, with~$\mc{X} = \Xsub$ and~$\mc{Q} = \Qc$.}
    \label{figure:jabali_basic_classical}
\end{figure}

\begin{figure}
    \centering
    \subfloat[Execution time and final optimality gaps.]{\includegraphics[scale=0.15]{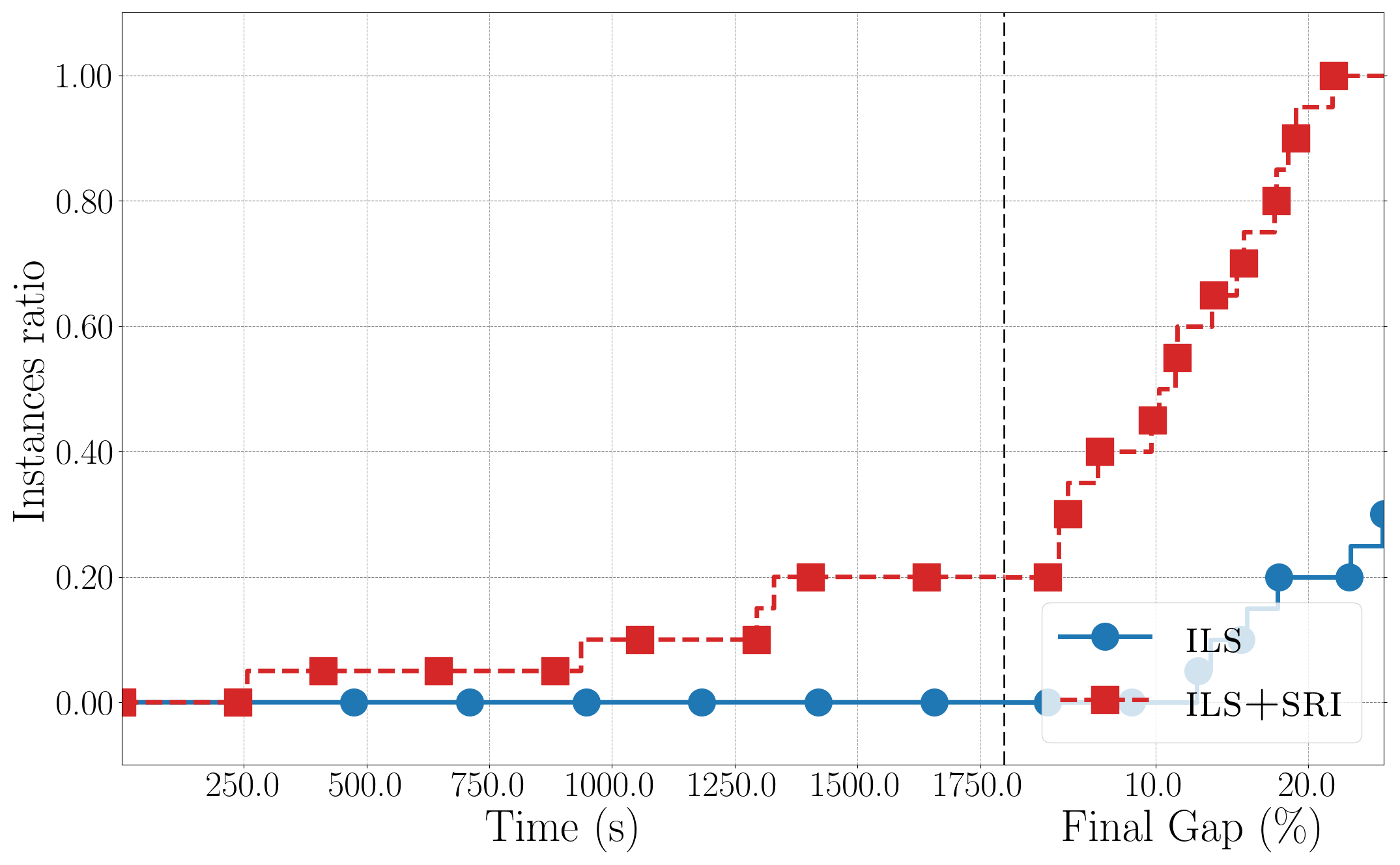} \label{figure:dinh_basic_classical_time}}
    \hfill
    \subfloat[Root gaps.]{\includegraphics[scale=0.15]{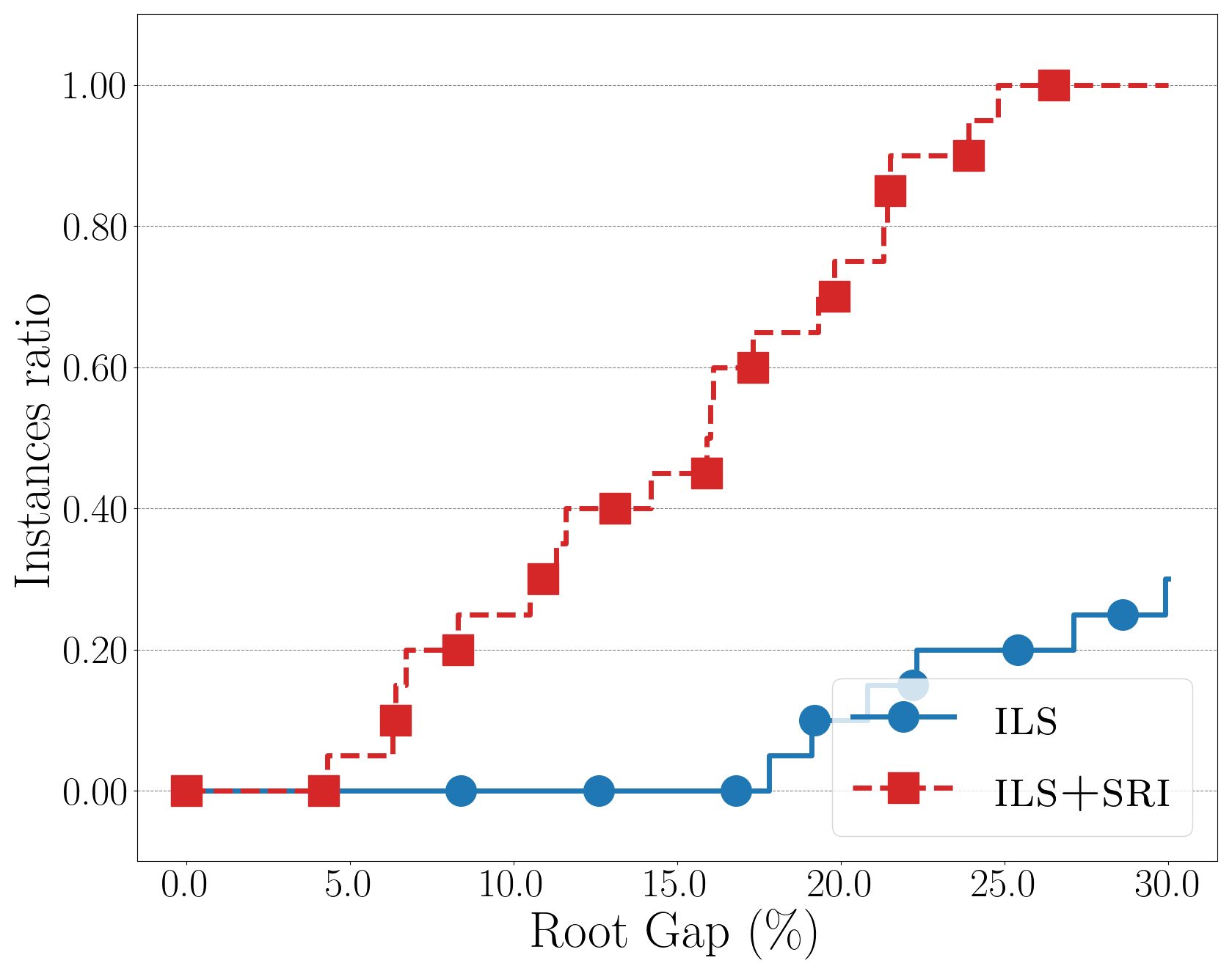}\label{figure:dinh_basic_classical_gap}}
    \\[0.5cm]
    \caption{Empirical cumulative distribution of the execution times and root gaps for~\cite{Dinh2018} instances, with~$\mc{X} = \Xsub$ and~$\mc{Q} = \Qc$.}
    \label{figure:dinh_basic_classical}
\end{figure}

\end{APPENDICES}

\end{document}